\documentclass[a4paper, reqno]{amsart}
\usepackage[utf8x]{inputenc} 
\usepackage{amsmath}
\usepackage{mathtools}
\usepackage{amssymb}
\usepackage{subcaption}
\usepackage{hyperref} 
\usepackage{bookmark}
\usepackage{amsthm}
\usepackage{yhmath}
\usepackage{float}
\usepackage[absolute,overlay]{textpos}
\usepackage{titletoc}
\usepackage{tikz-cd}
\usepackage[left=2cm, right=2cm]{geometry}
\usepackage{enumitem}
\usetikzlibrary{bbox}
\usepackage{adjustbox}

\let\originalmiddle=\middle
\def\middle#1{\mathrel{}\originalmiddle#1\mathrel{}}

\newtheorem{mythm}{Theorem}[section]
\newtheorem*{mythm*}{Theorem}
\newtheorem{myprop}[mythm]{Proposition}
\newtheorem*{myprop*}{Proposition}
\newtheorem{mylemma}[mythm]{Lemma}
\newtheorem{mycor}[mythm]{Corollary}
\newtheorem*{mycor*}{Corollary}

\theoremstyle{definition}
\newtheorem{mydef}[mythm]{Definition}
\newtheorem{myrem}[mythm]{Remark}
\newtheorem{myex}[mythm]{Example}

\newtheorem{mydefprop}[mythm]{Definition/Proposition}

\counterwithin*{equation}{mythm}

\newcommand\lsub[1]{{}_{#1}\,\!}
\newcommand\lsup[1]{{}^{#1}\,\!}
\newcommand{\q}{\mathbb{Q}[\eta]}
\newcommand{\C}{\mathbb{C}}
\newcommand{\Z}{\mathbb{Z}}
\newcommand{\Q}{\mathbb{Q}}

\newcommand{\K}{\mathbb{K}}

\newcommand{\id}{\mathrm{Id}}
\newcommand{\R}{\mathbb{R}}

\renewcommand{\t}{\mathfrak{t}}
\renewcommand{\b}{\mathfrak{b}}

\DeclareMathOperator{\coker}{\text{coker}}
\DeclareMathOperator{\im}{\text{im}}
\DeclareMathOperator{\Hom}{\mathrm{Hom}}
\DeclareMathOperator{\End}{\mathrm{End}}
\DeclareMathOperator{\add}{\mathrm{add}}
\DeclareMathOperator{\rad}{\mathrm{rad}}

\DeclareMathOperator{\sgn}{\mathrm{sgn}}
\DeclareMathOperator{\gldim}{\mathrm{gldim}}
\DeclareMathOperator{\Mat}{\mathrm{Mat}}

\title[Mutating Species with Potentials and Cluster Tilting Objects]{Mutating Species with Potentials and Cluster Tilting Objects}
\author{Christoffer S\"oderberg}

\begin{document}
\begin{abstract}
	Buan, Iyama, Reiten and Smith proved in \cite{BIRS2011Mutation} that cluster-tilting objects in triangulated 2-Calabi--Yau categories are closely connected with mutation of quivers with potentials over an algebraically closed field. We prove a more general statement where instead of working with quivers with potentials we consider species with potential over a perfect field.

	We describe the $3$-preprojective algebra of the tensor product of two tensor algebras of acyclic species using a species with potential. In the case when the Jacobian algebra of a species with potential is self-injective, we provide a description of the Nakayama automorphism of a particular case of mutation of the species with potential where you mutate along orbits of the Nakayama permutation, which preserves self-injectivity.

	For certain types of Jacobian algebras of species with potentials, we prove that they lie in the scope of the derived Auslander-Iyama correspondence due to Jasso-Muro. Mutating along orbits of the Nakayama permutation stays within this setting, yielding a rich source of examples. All $2$-representation finite $l$-homogeneous algebras that are constructed using certain species with potential and mutations of such species with potentials are considered.
\end{abstract}
\maketitle
\tableofcontents

\section{Introduction}
The theory of cluster algebras was introduced and studied in several articles by Fomin-Zelevinsky in \cite{FZ2002clustermutation,FZ2003clusteralgebras,FZ2007clusteralgebras} and together with Berenstein in \cite{FZ2005withberenstein}. Cluster categories were formed as a categorification of cluster algebras. The cluster category associated to a hereditary algebra was first introduced by \cite{BMRRT2006tiltingtheory} and \cite{CC06clusteralgashallalgebras} for the Dynkin type $A_n$ case. Several more articles delved deeper in the theory of cluster categories (e.g. \cite{MRZ03generalizedassociahedra,BMR07clustertiltedalgebras,BMR08clustermutationquiverrep,BMRT07clustersandseedsacycliccluster,CK06triangcattoclusteralg2,CK08triangcattoclusteralg}). They considered the cluster category $\mathcal{C}_Q$ associated to the path algebra of an acyclic quiver $Q$ in order to categorify the cluster algebra associated to $Q$. Keller and Reiten proved in \cite{KR08acycliccalabiyaucat} that an algebraic $2$-Calabi--Yau triangulated category over an algebraically closed field is a cluster category if it contains a cluster tilting subcategory whose quiver has no oriented cycles. Amiot generalised the approach of Keller and Reiten and introduced the cluster category $\mathcal{C}_\Lambda$ for an algebra $\Lambda$ of global dimension at most $2$ in \cite{amiotclustercategories}, where $\Lambda$ can be viewed as a cluster tilting object in $\mathcal{C}_\Lambda$. Amiot also introduced the construction of a new class of $2$-Calabi--Yau cluster categories $\mathcal{C}_{(Q, W)}$ for Jacobi-finite quivers with potential $(Q, W)$, i.e. its Jacobian algebra $\mathcal{P}(Q, W)$ is finite dimensional, with a cluster tilting object $T$ such that $\mathrm{End}_{\mathcal{C}_{(Q, W)}}(T)\cong \mathcal{P}(Q, W)$. In the case when $\Lambda = \K Q'/R$ is an algebra of global dimension at most $2$, for some ideal $R$, we have that $\mathcal{C}_\Lambda\simeq \mathcal{C}_{(Q, W)}$ (\cite[Theorem 6.12]{Keller2011CYcompletions}). The equivalence sends $\Lambda$ to $T$.

Iyama and Oppermann introduced $d$-representation finite algebras in \cite{IO11nRFalgandnAPRtilt} using the notion of $d$-cluster tilting objects. The $(d+1)$-preprojective algebra $\Pi_{d+1}(\Lambda)$ of a $\K$-algebra $\Lambda$ of global dimension at most $d$ was later introduced in \cite{IO13Stablecateofhigherpreproj}. In this setting, they studied the natural higher analogue of Amiot's cluster categories for algebras of global dimension at most $d$ called the $d$-Amiot cluster category and denoted by $\mathcal{C}_\Lambda^d$. If we assume that $\Lambda$ is $d$-representation finite, then there is a $d$-cluster tilting object $\Lambda$ in $\mathcal{C}_\Lambda^d$ such that $\mathrm{End}_{\mathcal{C}_\Lambda^d}(\Lambda)\cong \Pi_{d+1}(\Lambda)$. In particular, $\Pi_3(\Lambda)\cong\mathrm{End}_{\mathcal{C}_\Lambda^2}(\Lambda)$ for the $2$-cluster tilting object $\Lambda$ in the category $\mathcal{C}_\Lambda^2 = \mathcal{C}_\Lambda$ if $\Lambda$ is $2$-representation finite, which is the focus in this article. Iyama and Oppermann also showed in \cite{IO13Stablecateofhigherpreproj} that, in this case, $\Pi_3(\Lambda)$ is self-injective. By \cite[Proposition 3.6]{IO13Stablecateofhigherpreproj} it means that $T\cong T[2]$, where $[2]$ is the Serre functor in $\mathcal{C}_\Lambda$. An immediate consequence is that $\Pi_3(\Lambda)$ is a $2\Z$-cluster tilting object in the sense of \cite[Definition 1.3.1]{jasso2023derived}. In the case when $\Lambda = \K Q'/R$, as before, its $3$-preprojective algebra $\Pi_3(\Lambda)$ is thus isomorphic to a Jacobian algebra $\mathcal{P}(Q, W)$.

A quiver with potential $(Q, W)$ is said to be self-injective if $\mathcal{P}(Q, W)$ is a finite dimensional self-injective $\K$-algebra (\cite[Definition 3.6]{HerschendOsamu2011quiverwithpotential}). For a $d$-representation finite $\K$-algebra $\Lambda$, Herschend and Iyama introduced the notion of being $l$-homogeneous in \cite{herschend2011n}. It means that the $\tau_d$-orbits, where $\tau_d$ is the higher analogue of the Auslander-Reiten translation functor $\tau$, for all projective $\Lambda$-modules are of the same length $l$. Assuming that $\K$ is perfect Herschend and Iyama proved in \cite{herschend2011n} that if $\Lambda_1$ and $\Lambda_2$ are both $l$-homogeneous and $d_1$- respectively $d_2$-representation finite $\K$-algebras, then $\Lambda_1\otimes_\K \Lambda_2$ is a $l$-homogeneous $d_1+d_2$-representation finite $\K$-algebra. For the case $d_1 = d_2 = 1$ and when $\Lambda_1$ and $\Lambda_2$ are path algebras, the algebra $\Lambda = \Lambda_1\otimes_\K \Lambda_2$ will be a $2$-representation finite algebra, and its $3$-preprojective algebra $\Pi_3(\Lambda)\cong \mathcal{P}(Q, W)$ for some self-injective quiver with potential $(Q, W)$. In previous work \cite{soderberg2022preprojective} we showed that when $\Lambda_1$ and $\Lambda_2$ are given by tensor algebras of representation finite species we have that $\Pi_3(\Lambda)$ is isomorphic to a tensor algebra of a certain species with relations. In this article we prove that $\Pi_3(\Lambda)$ is in fact given by a self-injective species with potential $(S, W)$ (Proposition~\ref{prop - preprojective algebra of tensor species}), i.e. the Jacobian algebra $\mathcal{P}(S, W)$, which is a notion derived from the more general theory studied by \cite{Nquefack2012PotentialsJacobian}, is a finite dimensional self-injective $\K$-algebra.

The derived Auslander-Iyama correspondence was introduced in \cite{jasso2023derived}, see also Theorem~\ref{theorem - Jasso-Muro}. Assuming that $\K$ is perfect and $d\ge 1$ is an integer, it gives a bijection between pairs $(\mathcal{C}, T)$, consisting of
\begin{enumerate}
	\item a triangulated category $\mathcal{C}$ with certain properties and
	\item a basic $d\Z$-cluster tilting object $T\in \mathcal{C}$,
\end{enumerate}
and pairs $(\Pi, I)$, consisting of
\begin{enumerate}
	\item a basic finite dimensional self-injective $\K$-algebra $\Pi$ that is twisted $(d+2)$-periodic and
	\item a certain invertible $\Pi$-bimodule $I$.
\end{enumerate}
This equivalence is given by $(\mathcal{C}, T)\mapsto (\mathrm{End}_\mathcal{C}(T), \mathrm{Hom}_\mathcal{C}(T, T[-d]))$. If $\Lambda$ is $d$-representation finite, then by the above discussion the pair $(\mathcal{C}_\Lambda^d, \Lambda)$ is an example to which the correspondence applies. In particular, if $\Lambda = \Lambda_1\otimes_\K \Lambda_2$, where $\Lambda_1$ and $\Lambda_2$ are given by $l$-homogeneous tensor algebras of representation finite species, then $\Pi = \Pi_3(\Lambda)$ is in this case an algebra that appears in the correspondence above for $d=2$. In this way we obtain a new family of twisted periodic self-injective $\K$-algebras parametrised by certain pairs of Dynkin diagrams, which is explicitly given by species with potentials.

The mutation theory of cluster algebras introduced in \cite{FZ2002clustermutation} motivated several mutation theories on different objects. We have a mutation theory of cluster tilting objects in hom-finite triangulated $2$-Calabi--Yau categories over an algebraically closed field \cite{BMRRT2006tiltingtheory,IyamaYoshino2008Mutation} defined using exchange sequences. Later on Derksen, Weyman and Zelevinsky defined a mutation on quivers with potentials in \cite{dwz2008mutation}. The relationship between these types of mutation is discussed in \cite{BIRS2011Mutation}, where they proved that the mutation of cluster tilting objects directly corresponds to mutation of certain Jacobian algebras over quivers with potentials. In this article we partially extend this relationship to a more general setting of species with potential. In the paper \cite{HerschendOsamu2011quiverwithpotential} it was proved that mutating along Nakayama orbits of a self-injective quiver with potential $(Q, W)$ preserves the self-injective property under certain conditions, i.e. if $\mu_{(k)}$ denotes the mutation along a Nakayama orbit then $\mu_{(k)}(Q, W)$ is a self-injective quiver with potential. Note that if $(S, W)$ is self-injective then $\mathcal{P}(S, W)$ is a Frobenius algebra, meaning that $\mathcal{P}(S, W) \cong D\mathcal{P}(S, W)_\gamma$ as $\mathcal{P}(S, W)$-bimodules where $\gamma$ is the Nakayama automorphism of $\mathcal{P}(S, W)$. We show in this article that the Nakayama automorphism $\gamma_{(k)}$ of $\mathcal{P}(\mu_{(k)}(S, W))$ is given in terms of $\gamma$. This is described in our main theorem, where we assume that $(S, W)$ is reduced, i.e. $W$ does not have any $2$-cycles.

\begin{mythm*}(Theorem~\ref{Theorem A}, Theorem~\ref{Theorem B})\label{Theorem - Main mutation theorem}
	Let $\mathcal{C}$ be a $2$-Calabi--Yau triangulated category and $T = \bigoplus_{i = 1}^n T_i\in\mathcal{C}$ a basic cluster tilting object with indecomposable summands $T_i$, that satisfies the vanishing condition at $i$ for all $i\in Q_0$ and $k\in Q_0$ is mutable.
	\begin{enumerate}
		\item If $\mathrm{End}_\mathcal{C}(T)^{op}\cong \mathcal{P}(S, W)$ for a reduced species with potential $(S, W)$, then $\mathrm{End}_\mathcal{C}(\mu_{T_k}(T))^{op}\cong \mathcal{P}(\mu_k(S, W))$.
		\item Assume furthermore, that $\mathcal{P}(S, W)$ is self-injective with Nakayama permutation $\sigma$ and Nakayama automorphism $\gamma$ satisfying conditions (A) and (B) in Section~\ref{Section - nakayama automorphism}. If the Nakayama orbit $(k) = \{\sigma^i(k)\mid i\ge 0\}$ is sparse (see Definition~\ref{definition - sparse orbit}), then $\mathcal{P}(\mu_{(k)}(S, W))$ is self-injective with Nakayama automorphism $\mu_{(k)}(\gamma)$.
	\end{enumerate}
\end{mythm*}

The vanishing condition is a technical assumption which appears in \cite{BIRS2011Mutation} and is presented in the beginning of Subsection~\ref{subsection - relation}. Part (1) is a species version of \cite[Theorem 5.2]{BIRS2011Mutation}. Part (2) says that mutating along Nakayama orbits preserves the self-injective property, and so we obtain a new $2\Z$-cluster tilting object (compare with \cite[Theorem 4.2]{HerschendOsamu2011quiverwithpotential}). Hence mutation yields a rich source of examples that fit in the derived Auslander-Iyama correspondence (Theorem~\ref{theorem - Jasso-Muro}). The new Nakayama automorphism $\mu_{(k)}(\gamma)$ satisfies the conditions (A) and (B), thus we can iterate the process with another sparse Nakayama orbit.

The proof shares a similar outline as in \cite[Theorem 5.2]{BIRS2011Mutation} but suited for species with potentials. The main differences are due to the richer structure that species enjoy compared to quivers. For the first part we show that any isomorphism $\mathrm{End}_\mathcal{C}(T)^{op}\cong \mathcal{P}(S, W)$ can be defined using certain weak $2$-almost split sequences. Then in order to define an isomorphism $\mathrm{End}_\mathcal{C}(\mu_{T_k}T)^{op}\cong \mathcal{P}(\mu_k(S, W))$ we study weak $2$-almost split sequences for $\mu_{T_k}(T)$ and compare them to $\mathcal{P}(\mu_k(S, W))$. For the second part we note that we immediately have $\mathrm{End}_\mathcal{C}(\mu_{T_{(k)}}T)^{op}\cong \mathcal{P}(\mu_{(k)}(S, W))$. Since the mutation $\mu_{T_{(k)}}(T)$ is given in terms of exchange sequences, we use exchange sequences to define $\mathcal{P}(\mu_{(k)}(S, W))$. The vanishing condition and the mutable condition are technical assumptions needed to define the isomorphism in the first part.

Let us produce an example where the main theorem is applicable, given in more detail in Example~\ref{example - A times B}. Let $S^1$ and $S^2$ be the species
\begin{equation*}
	S^1: \R \xrightarrow{\R} \R \xleftarrow{\R} \R, \qquad S^2:\C\xrightarrow{\C} \R
\end{equation*}
which are of type $A_3$ and $B_2$ respectively. Both tensor algebras $T(S^1)$ and $T(S^2)$ are $2$-homogeneous and representation finite and therefore $\Lambda = T(S^1)\otimes_\R T(S^2)$ is a $2$-representation finite algebra. Its $3$-preprojective algebra $\Pi_3(\Lambda)$ is isomorphic to $\mathcal{P}(S, W)$, where $S$ is the species
\begin{equation*}
	\begin{tikzcd}[row sep = 1cm, column sep = 1cm]
		\C \arrow[r, "\C"] \arrow[d, "\C"] & \C \arrow[d, "\C"] & \C \arrow[l, "\C"'] \arrow[d, "\C"] \\
		\R \arrow[r, "\R"] & \R \arrow[lu, "\C"'] \arrow[ru, "\C"'] & \R \arrow[l, "\R"']
	\end{tikzcd}
\end{equation*}
over the quiver
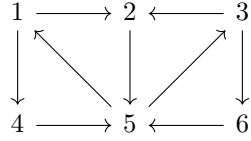
\begin{equation*}
	\begin{tikzcd}[row sep = 1cm, column sep = 1cm]
		1 \arrow[r] \arrow[d] & 2 \arrow[d] & 3 \arrow[l] \arrow[d] \\ 
		4 \arrow[r] & 5 \arrow[lu] \arrow[ru] & 6 \arrow[l]
	\end{tikzcd}
\end{equation*}
and $W$ is the potential in \eqref{eq - potential in example at end 1}. In this case the Nakayama permutation is the identity. Mutating $(S, W)$ at vertex $5$ yields the species with potential $(S', W')$, where $S'$ is
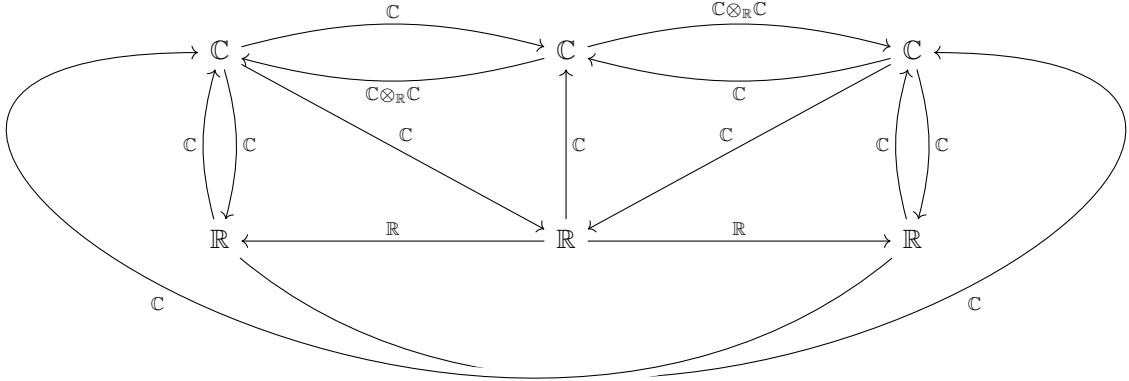
\begin{equation*}
	\begin{tikzcd}[row sep = 2cm, column sep = 4cm, bezier bounding box]
		\C \arrow[r, "\C", bend left=15] \arrow[rd, "\C"] \arrow[d, "\C", bend left=15] & \C \arrow[l, "\C\otimes_\R \C", bend left=15] \arrow[r, "\C\otimes_\R \C", bend left=15] & \C \arrow[ld, "\C"'] \arrow[l, "\C", bend left=15] \arrow[d, "\C", bend left=15] \\
		\R \arrow[u, "\C", bend left=15] \arrow[urr, "\C"', out=-40, in=0, looseness = 2] & \R \arrow[l, "\R"'] \arrow[r, "\R"] \arrow[u, "\C"'] & \R \arrow[u, "\C", bend left=15] \arrow[ull, "\C", out=220, in=-180, looseness = 2, crossing over]
	\end{tikzcd}
\end{equation*}
and $W'$ is the potential in \eqref{eq - potential example at end 2}. Here $(S', W')$ is not reduced since its potential includes $2$-cycles. Due to our main theorem its Jacobian algebra $\mathcal{P}(S', W')$ is a basic finite dimensional self-injective $\K$-algebra that is twisted $4$-periodic. Therefore it will show up as a pair $(\mathcal{P}(S', W'), I)$ in the Jasso-Muro equivalence, where $I$ is given as the $4$-syzygy in the projective $\mathcal{P}(S', W')$-bimodule resolution of $\mathcal{P}(S', W')$. \\

\noindent\textbf{Acknowledgements.} The author wishes to thank his supervisor Martin Herschend for valuable discussions during the making of this article.

\subsection{Outline}
In Section~\ref{Section - prel} we introduce our conventions and recall the dual basis lemma. In Section~\ref{Section - species with potential and preproj} we define several of the notions used throughout the article such as species, potentials, derivative operators, preprojective algebras and study tensor products of tensor algebras of species. In Section~\ref{Section - mutation} we state various results from \cite{IyamaYoshino2008Mutation,HerschendOsamu2011quiverwithpotential,IO13Stablecateofhigherpreproj} regarding the mutation theory of cluster tilting objects in a cluster category. We also define the mutation of a species with potential using the theory of \cite{Nquefack2012PotentialsJacobian}. We prove the first part of Theorem~\ref{Theorem - Main mutation theorem}. In Section~\ref{Section - nakayama automorphism} we introduce self-injective species with potential and prove the second part of Theorem~\ref{Theorem - Main mutation theorem}. In Section~\ref{Section - Jasso-muro plus examples} we state the connection to the derived Auslander-Iyama correspondence (Theorem~\ref{theorem - Jasso-Muro}) and compute some examples.

\section{Preliminaries}\label{Section - prel}
In this section we will cover the different definitions used throughout the article such as species with potentials, Jacobian algebras of species with potentials and higher preprojective algebras. We also compute a particular example of a $3$-preprojective algebra of a tensor product of tensor algebras of two species.

\subsection{Conventions}
By default $R$-modules means left $R$-modules and right $R$-modules are considered as left $R^{op}$-modules. Composition of arrows (respectively morphisms) goes from right to left. We denote by $\mathcal{Z}(D)$ the center of an algebra $D$. The $D$-center $\mathcal{Z}_D(A)$, of some $D$-bimodule $A$, is defined to be the $\mathcal{Z}(D)$-subbimodule of $A$ consisting of all elements $a\in A$ such that $ad = da$ for all $d\in D$.

\subsection{Casimir Elements}
The following lemma is well-known and we will refer to \cite{kulshammer2017pro} for the proof. It can also be found in \cite{PareigisLnotes}.

\begin{mylemma}\label{lemma - dual basis}(Dual basis lemma)
	Let $P_R$ be a right $R$-module. Then, the following statements are equivalent:
	\begin{enumerate}
		\item $P$ is finitely generated and projective.
		\item There are $x_1, x_2, \dots, x_m\in P$ and $f_1, f_2, \dots, f_m\in \mathrm{Hom}_{R^{op}}(P, R)$ such that for every $x\in P$ we have
		\begin{equation*}
			x = \sum_{i=1}^m x_if_i(x).
		\end{equation*}
		\item The dual basis homomorphism $P\otimes_R \mathrm{Hom}_{R^{op}}(P, R) \to \mathrm{End}_{R^{op}}(P)$, $p\otimes f\mapsto (q\mapsto pf(q))$ is an isomorphism.
	\end{enumerate}
\end{mylemma}

\begin{mydef}
	Let $x_i$ and $f_i$ be as in Lemma~\ref{lemma - dual basis}. The element
	\begin{equation*}
		\sum_{i=1}^n x_i\otimes_R f_i\in P\otimes_R \mathrm{Hom}_{R^{op}}(P, R)
	\end{equation*}
	is called the Casimir element of $P\otimes_R \mathrm{Hom}_{R^{op}}(P, R)$. The Casimir element does not depend on the choice of $x_i$ and $f_i$ by \cite[Lemma 1.1]{Dlab_1980}.
\end{mydef}

Let $\Gamma$ be a $\K$-algebra with Jacobson radical $\mathcal{J}_\Gamma$. We can view $\Gamma$ as a topological algebra with the $\mathcal{J}_\Gamma$-adic topology with the basic system $\{\mathcal{J}_\Gamma^i\}_{i\ge 0}$ of open neighbourhoods of $0$. Thus the closure of a subset $S\subseteq \Gamma$ becomes
\begin{equation}\label{eq - closure of a set}
	\overline{S} = \bigcap_{l\ge 0}(S + \mathcal{J}_\Gamma^l).
\end{equation}
For background on the adic topology see \cite[Section 10]{atiyahmacdonald1969}.

\section{Species with Potential and Higher Preprojective Algebras}\label{Section - species with potential and preproj}
\subsection{Species}
Throughout this article we will assume that $\K$ is a perfect field. The main reason being that we need that $D\otimes_\K D'$ is semisimple where $D$ and $D'$ are two division $\K$-algebras. Although many results are still true without this assumption, we make this assumption for the whole article for convenience.

\begin{mydef}\label{Definition - Species}(Species)
	Let $Q$ be a finite quiver. 
	\begin{enumerate}
		\item A species $S=(D_i, M_\alpha)_{i\in Q_0, \alpha\in Q_1}$ is a collection where each $D_i$ is a semisimple $\K$-algebra and $M_\alpha\in D_j$-$D_i$-$\mathrm{mod}$, where $\alpha:i\rightarrow j$, such that $\mathrm{Hom}_{D_i^{op}}(M_\alpha, D_i)\cong \mathrm{Hom}_{D_j}(M_\alpha, D_j)$ as $D_i$-$D_j$-modules, and such that $\mathcal{Z}_\K(D_i) = D_i$ and $\mathcal{Z}_\K(M_\alpha) = M_\alpha$ for all $i\in Q_0$ and $\alpha\in Q_1$.
		\item A finite dimensional species $S$ is a species $S$ such that $\dim_\K D_i<\infty$ and $\dim_\K M_\alpha<\infty$ for all $i\in Q_0$ and $\alpha\in Q_1$.
		\item We say that a finite dimensional species $S$ is a $\mathbb{K}$-species if all $D_i$ are division $\K$-algebras over a common central subfield $\mathbb{K}$.
		\item We call a species $S$ acyclic if $Q$ is acyclic.
	\end{enumerate}
\end{mydef}

\begin{mydef}
	Let $S$ be a species, $D=\bigoplus_{i\in Q_0}D_i$ and $M=\bigoplus_{\alpha\in Q_1}M_\alpha\in D$-$D\mathrm{-mod}$. 
	\begin{enumerate}
		\item We define the tensor algebra $T(S)$ to be the tensor ring $T(D, M)$. More explicitly,
		\begin{equation*}
			T(S) = T(D, M) = \bigoplus_{k\ge 0} M^k = \bigoplus_{k\ge 0} M^{\otimes_D k}, \quad M^0 = D.
		\end{equation*}
		\item We define the complete tensor algebra $\widehat{T}(S)$ as
		\begin{equation*}
			\widehat{T}(S) = \prod_{k\ge 0} M^k = \prod_{k\ge 0} M^{\otimes_D k}, \quad M^0 = D.
		\end{equation*}
	\end{enumerate}
\end{mydef}

The benefit of using $\widehat{T}(S)$ is that its Jacobson radical is $\langle M\rangle = \widehat{T}(S)_{\ge 1}$.

\begin{myrem}
	If $S$ is a finite dimensional species then the following are equivalent
	\begin{enumerate}
		\item $S$ is acyclic,
		\item $T(S)\cong \widehat{T}(S)$,
		\item $\dim_\K T(S)<\infty$.
	\end{enumerate}
	Note that in this case $T(S)$ is a finite dimensional hereditary algebra.
\end{myrem}

\begin{mylemma}\label{lemma - division algebra species}
	Let $S$ be a species and assume that $\dim_\K D_i<\infty$ for all $i\in Q_0$. There exists a species $S'=(D_k', M_\alpha')_{k\in Q_0', \alpha\in Q_1'}$ over a quiver $Q'$, where all $D_k'$ matrix rings over some division $\K$-algebra, such that $T(S)\cong T(S')$ and $\widehat{T}(S)\cong \widehat{T}(S')$.
\end{mylemma}

\begin{proof}
	By the Artin-Wedderburn theorem there exist for each $i\in Q_0$ division $\K$-algebras $D_j'$ such that $D_i \cong \bigoplus_{j=1}^{n_i}M_{m_j}(D_j')$, where $M_{m_j}(D_j')$ is the matrix ring of $D_j'$ of some size $m_j$. Let $e_{ij}$ be the identity element in $M_{m_j}(D_j')$. Clearly
	\begin{equation*}
		1_{T(S)} = \sum_{\substack{i\in Q_0 \\ j\in \{1, 2, \dots, n_i \}}}e_{ij}.
	\end{equation*}
	Let $Q'$ be the quiver given by
	\begin{equation*}
		\begin{aligned}
			Q_0' &= \{(i, j) \mid i\in Q_0, j\in \{1, 2, \dots, n_i\} \}, \\
			Q_1' &= \{\alpha_{jj'}: (s(\alpha), j) \to (t(\alpha), j') \mid \alpha\in Q_1, j\in \{1, 2, \dots, n_{s(\alpha)}\}, j'\in \{1, 2, \dots, n_{t(\alpha)}\} \}.
		\end{aligned}
	\end{equation*}
	Define the species $S'$ over $Q'$ as
	\begin{equation*}
		\begin{aligned}
			D_{(i, j)} = M_{m_j}(D_j'), \\
			M_{\alpha_{jj'}} = e_{t(\alpha)j'}M_\alpha e_{s(\alpha)j}.
		\end{aligned}
	\end{equation*}
	By construction we immediately have that $T(S)\cong T(S')$ and $\widehat{T}(S)\cong \widehat{T}(S')$. 
\end{proof}

\begin{myrem}
	Definition~\ref{Definition - Species} of a species is not the same as the definitions given in \cite{berg2011structure, ringel1976representations, gabriel1973indecomposable}. They require that all $D_i$ are division $\K$-algebras for all $i\in Q_0$. However, our definition for $\K$-species coincide. We make this change in the definition of a species since we want $S(S^1, S^2)$ defined in Definition~\ref{definition - S(S^1, S^2)} to be a species. For a finite dimensional species $S$ there exists a Morita equivalent $\K$-species $S'$ described later in Proposition~\ref{proposition - morita equivalent k-species}. Thus, in practice, we can assume without loss of generality that our finite dimensional species are $\K$-species.
\end{myrem}

\begin{myrem}
	From the construction of $S'$ we have a natural surjective quiver morphism $\pi: Q' \to Q$ given by
	\begin{equation*}
		\begin{aligned}
			\pi_0: Q_0' &\to Q_0, \\
			(i, j) &\mapsto i, \\
			\pi_1: Q_1' &\to Q_1, \\
			\alpha_{jj'} &\mapsto \alpha.
		\end{aligned}
	\end{equation*}
\end{myrem}

Throughout this article we will assume that all species are finite dimensional. Since $D$ is semisimple and finite dimensional it is a symmetric Frobenius algebra by \cite[Corollary 5.17]{skowronski2011frobenius}. In other words, there exists a linear map $\t: D\to \K$ such that there is no non-zero ideal contained in the kernel of $\t$ and that $\t(ab) = \t(ba)$ for all $a,b\in D$.

We assume that each species is equipped with such a map. Hence we get an isomorphism
\begin{equation}\label{eq - iso frobenius D}
	\begin{aligned}
		D&\xrightarrow{\sim} \mathrm{Hom}_\K(D, \K),  \\
		d&\mapsto \t(d-),
	\end{aligned}
\end{equation}
of $D$-bimodules. The module $D$ is projective over itself, and therefore by applying Lemma~\ref{lemma - dual basis} we can write down the Casimir element of $D$ as
\begin{equation*}
	\sum_{i\in Q_0}\sum_{j=1}^{\dim_\K D_i} e_i^j\otimes \hat{e}_i^j\in D\otimes_\K \mathrm{Hom}_\K(D, \K).
\end{equation*}
Using the isomorphism in \eqref{eq - iso frobenius D} we can define elements $\bar{e}_i^j\in D$ such that $\hat{e}_i^j = \t(\bar{e}_i^j-)$. We call the element
\begin{equation*}
	\sum_{i\in Q_0}\sum_{j=1}^{\dim_\K D_i} e_i^j\otimes \bar{e}_i^j\in D\otimes_\K D
\end{equation*}
the Casimir element of $D$ associated to $\t$.

\begin{myrem}\label{remark - dualising condition}\cite[Remark 2.2]{Nquefack2012PotentialsJacobian}
	The dualising condition, i.e. $\mathrm{Hom}_{D_i^{op}}(M_\alpha, D_i)\cong \mathrm{Hom}_{D_j}(M_\alpha, D_j)$ as $D_i$-$D_j$-modules, in Definition~\ref{Definition - Species} is automatic if $D$ is a symmetric Frobenius algebra (see also \cite[Section 2.1]{BSWsuperpotential} for the case when $D$ is a semisimple $\C$-algebra). The $D$-bimodule morphisms
	\begin{equation*}
		\begin{aligned}
			\t_l &= \t\circ -: \mathrm{Hom}_D(M, D) \xrightarrow{\sim} \mathrm{Hom}_\K(M, \K) \\
			\t_r &= \t\circ -: \mathrm{Hom}_{D^{op}}(M, D) \xrightarrow{\sim} \mathrm{Hom}_\K(M, \K)
		\end{aligned}
	\end{equation*}
	are isomorphisms with the inverses given as
	\begin{equation*}
		\begin{aligned}
			\t_l^{-1} : \mathrm{Hom}_\K(M, \K) &\to \mathrm{Hom}_D(M, D), \\
			f&\mapsto \left[x\mapsto \sum_{i\in Q_0}\sum_{j=1}^{\dim_\K D_i} e_i^jf(\bar{e}_i^jx) \right], \\
			\t_r^{-1} : \mathrm{Hom}_\K(M, \K) &\to \mathrm{Hom}_{D^{op}}(M, D), \\
			f&\mapsto \left[x\mapsto \sum_{i\in Q_0}\sum_{j=1}^{\dim_\K D_i} e_i^jf(x\bar{e}_i^j) \right].
		\end{aligned}
	\end{equation*}
\end{myrem}

Let us denote $M^* = \mathrm{Hom}_\K(M, \K)$. Following \cite{Nquefack2012PotentialsJacobian} we define the $D$-bimodule morphism
\begin{equation*}
	\begin{aligned}
		\mathfrak{b}: M\otimes_D M^* \oplus M^*\otimes_D M &\to D, \\
		x\otimes x^* + y^*\otimes y &\mapsto \t_l^{-1}(x^*)(x) + \t_r^{-1}(y^*)(y).
	\end{aligned}
\end{equation*}

\begin{myrem}\label{remark - b determined by t}
	Note that $\b$ is completely determined by $\t$.
\end{myrem}

We will denote the morphism $\b$ as $\b_S$ to make it clear that it is defined for a given species $S$.

\begin{mylemma}
	The morphism $\mathfrak{b}$ satisfies:
	\begin{enumerate}
		\item $\t$ is a symmetrising trace for $\b$, i.e. $\t(\mathfrak{b}(x\otimes x^*)) = \t(\b(x^*\otimes x))$ for all $x\in M$ and $x^*\in M^*$.
		\item The morphisms
		\begin{equation*}
			\begin{aligned}
				\b_{1, l}:& M^* \to \mathrm{Hom}_D(M, D), \\
				&\hspace*{0.2cm}x^*\mapsto \b(-\otimes x^*), \\
				\b_{1, r}:& M \to \mathrm{Hom}_{D^{op}}(M^*, D), \\
				&\hspace*{0.2cm}x \mapsto \b(x\otimes -),
			\end{aligned}
		\end{equation*}
		are bijective. Or equivalently, the morphisms
		\begin{equation*}
			\begin{aligned}
				\b_{2, r}:& M^* \to \mathrm{Hom}_{D^{op}}(M, D), \\
				&\hspace*{0.2cm}x^*\mapsto \b(x^*\otimes -), \\
				\b_{2, l}:& M \to \mathrm{Hom}_D(M^*, D), \\
				&\hspace*{0.2cm}x \mapsto \b(-\otimes x),
			\end{aligned}
		\end{equation*}
		are bijective.
	\end{enumerate}
\end{mylemma}

\begin{proof}
	\begin{enumerate}
		\item This follows from Remark~\ref{remark - dualising condition}, that is that $\t\circ -$ is the inverse to $\t_l^{-1}$ and $\t_r^{-1}$.
		\item Again, from Remark~\ref{remark - dualising condition} the inverse of $\b_{1,l}$ is $\t\circ -$, thus $\b_{1,l}$ is bijective. To show that $\b_{1,r}$ is bijective we will first denote $\phi: M\to \mathrm{Hom}_\K (M^*, \K)$ as the standard isomorphism. Note that Remark~\ref{remark - dualising condition} holds for any $D$-bimodule, therefore we can observe that $\b_{1,r}=\t_l^{-1}\circ \phi^{-1}$, hence bijective.
	\end{enumerate}
\end{proof}

\begin{mycor}
	If $S$ is a $\K$-species, then the ordered data $\{M, M^*; \mathfrak{b}_S\}$ becomes a symmetrisable dualising pair of bimodules as in \cite[Definition 2.3]{Nquefack2012PotentialsJacobian}.
\end{mycor}

Let
\begin{equation*}
	\sum_{i=1}^n \hat{x}_i\otimes x_i\in \mathrm{Hom}_D(M, D)\otimes_D M, \qquad \sum_{j=1}^m y_j\otimes \hat{y}_j\in M\otimes_D \mathrm{Hom}_{D^{op}}(M, D),
\end{equation*}
be Casimir elements. Setting $x^*_i = \b_{1, l}^{-1}(\hat{x}_i)$ and $y^*_j = \b_{2, r}^{-1}(\hat{y}_j)$ we get elements
\begin{equation}\label{eq - Casimir elements of M M^*}
	\sum_{i=1}^n x_i^*\otimes x_i\in M^*\otimes_D M, \qquad \sum_{j=1}^m y_j\otimes y_j^*\in M\otimes_D M^*,
\end{equation}
satisfying
\begin{equation}\label{eq - casimir elements equations ass. to b}
	\sum_{i=1}^n \b(x\otimes x_i^*)x_i = x = \sum_{j=1}^my_j\b(y_j^*\otimes x), \qquad \sum_{i=1}^n x_i^*\b(x_i\otimes \xi) = \xi = \sum_{j=1}^m\b(\xi\otimes y_j)y_j^*,
\end{equation}
for all $x\in M$ and $\xi\in M^*$.

\begin{myrem}\label{remark - delta of casimir elements}
	Notice that by \eqref{eq - casimir elements equations ass. to b} we have that $\b(x_i\otimes x_j^*) = \delta_{ij}$ and $\b(y_i^*\otimes y_j) = \delta_{ij}$.
\end{myrem}

\subsection{Jacobian Algebras}
In this subsection we define the Jacobian algebra for a species with potential as in \cite{Nquefack2012PotentialsJacobian}.

\begin{mydef}\cite[Definition 3.5]{Nquefack2012PotentialsJacobian}
	Let $S$ be a species. We call elements in $\mathcal{Z}_D(\widehat{T}(S)_{\ge 2})$ potentials. If a potential $W\in \mathcal{Z}_D(\widehat{T}(S)_{\ge 3})$, i.e. the potential $W$ does not have any $2$-cycles, we call it a reduced potential.
\end{mydef}

\begin{mydef}\cite[Proposition 3.2]{Nquefack2012PotentialsJacobian}\label{definition - derivative operators}
	Let $S$ be a species. We define the derivative operators as
	\begin{equation*}
		\begin{aligned}
			\partial^l: M^*\otimes_D M\otimes_D \widehat{T}(S) &\to \widehat{T}(S), \\
			\xi\otimes a\otimes b &\mapsto \partial^l_\xi(a\otimes b) = \b(\xi\otimes a)b, \\
			\partial^r: \widehat{T}(S) \otimes_D M\otimes_D M^* &\to \widehat{T}(S), \\
			a\otimes b\otimes \xi &\mapsto \partial^r_\xi(a\otimes b) = a\b(b\otimes \xi).
		\end{aligned}
	\end{equation*}
\end{mydef}

For each $\xi\in M^*$ we use the notation $\partial^l_\xi = \partial^l(\xi\otimes-)$ and $\partial^r_\xi = \partial(-\otimes \xi)$. We extend these operators to $\widehat{T}(S)$ by setting $\partial^l_\xi(d) = 0 = \partial^r_\xi(d)$ for all $d\in D$.

\begin{mylemma}\cite[Lemma 2.5, Proposition 3.2]{Nquefack2012PotentialsJacobian}
	Let $S$ be a species and let $W\in \mathcal{Z}_D(\widehat{T}(S)_{\ge 2})$ be a potential. We define the permutation operators $\varepsilon_l, \varepsilon_r: \mathcal{Z}_D(\widehat{T}(S)_{\ge 2})\to \mathcal{Z}_D(\widehat{T}(S)_{\ge 2})$ by
	\begin{equation*}
		\begin{aligned}
			\varepsilon_l(W) = \sum_{i=1}^m \partial^l_{x^*_i}(W)x_i, \\
			\varepsilon_r(W) = \sum_{i=1}^n y_i\partial^r_{y^*_i}(W),
		\end{aligned}
	\end{equation*}
	where $x_i, x^*_i, y_i, y^*_i$ make up the Casimir elements in \eqref{eq - Casimir elements of M M^*}. We define the cyclic permutation operator $\varepsilon_c: \mathcal{Z}_D(\widehat{T}(S)_{\ge 2})\to \mathcal{Z}_D(\widehat{T}(S)_{\ge 2})$ on the homogeneous element of degree $d+1$ by $\varepsilon_c(W) = \sum_{k=0}^d \varepsilon^k_l(W) = \sum_{k=0}^d \varepsilon^k_r(W)$.
\end{mylemma}

\begin{myrem}\label{remark - comapare to nquefack}
	Recall that the Casimir element is independent of the choice of basis. The operators $\varepsilon_l$ and $\varepsilon_r$ are also independent of the choice of $x_i$, $x_i^*$, $y_i$ and $y_i^*$, which is a consequence of $W\in \mathcal{Z}_D(\widehat{T}(S)_{\ge 2})$. Hence $\varepsilon_l$ and $\varepsilon_r$ are well-defined. Note that in \cite{Nquefack2012PotentialsJacobian} $\varepsilon_l^k$ is defined in terms of the Casimir element for $M^{\otimes_D k}$. However this gives the same result as $\varepsilon_l^{\circ k}$. In particular, $\varepsilon_l^d(W) = W$ for $W \in \mathcal{Z}_D(\widehat{T}(S)_d)$.
\end{myrem}

\begin{myrem}
	The derivative operators in Definition~\ref{definition - derivative operators} satisfy the following:
	\begin{equation*}
		\begin{aligned}
			\partial^l_\xi(d) &= 0 = \partial^r_\xi(d), \\
			\partial^l_\xi(mx) &= \partial^l_\xi(m)x, \\
			\partial^r_\xi(xm) &= x\partial^r_\xi(m),
		\end{aligned}
	\end{equation*}
	for all $d\in D$, $m\in M$ and $x\in \widehat{T}(S)$.
\end{myrem}

\begin{mylemma}\cite[Proposition 3.2]{Nquefack2012PotentialsJacobian}\label{lemma - epsilon property}
	Let $S$ be a species and $W\in \mathcal{Z}_D(\widehat{T}(S)_{\ge 2})$.
	\begin{enumerate}
		\item The derivative operators in Definition~\ref{definition - derivative operators} satisfy $\partial^r_\xi(\varepsilon_l (W)) = \partial^l_\xi(W)$ and $\partial^l_\xi(\varepsilon_r(W)) = \partial^r_\xi(W)$ for $\xi\in M^*$.
		\item $\varepsilon_l\circ \varepsilon_r(W) = W = \varepsilon_r\circ \varepsilon_l(W)$. 
	\end{enumerate}
\end{mylemma}

\begin{mydef}
	We define the cyclic derivative operator
	\begin{equation*}
		\begin{aligned}
			\partial: M^*\otimes_\K \mathcal{Z}_D(\widehat{T}(S)_{\ge 2}) &\to \widehat{T}(S)_{\ge 1} \\
			\xi\otimes W &\mapsto \partial^l(\xi\otimes \varepsilon_c(W)) = \partial^r(\varepsilon_c(W)\otimes \xi).
		\end{aligned}
	\end{equation*}
\end{mydef}

\begin{mydef}
	For a species with potential $(S, W)$ we define the Jacobian algebra $\mathcal{P}(S, W) = \widehat{T}(S)/\mathcal{J}(S, W)$, where $\mathcal{J}(S, W) = \overline{\langle \partial_\xi (W) \mid \xi\in M^*\rangle}$. We say that a species with potential $(S, W)$ is reduced if $W$ is reduced.
\end{mydef}

Let us introduce the following notation. For every $\alpha\in Q_1$ we define $\underline{\alpha}\subseteq M_\alpha$ as sets of elements such that
\begin{equation*}
	c_\alpha = \sum_{a\in \underline{\alpha}} a \otimes a^* \in M_\alpha \otimes_{D_{s(\alpha)}} M_\alpha^*
\end{equation*}
is the Casimir element of $M_\alpha\otimes_{D_{s(\alpha)}}M_\alpha^*$ associated to $\b_S$, i.e.
\begin{equation*}
	\sum_{a\in \underline{\alpha}} a \otimes \b(a^*\otimes -)\in M_\alpha \otimes_{D_{s(\alpha)}} \mathrm{Hom}_{D_{s(\alpha)}^{op}}(M, D_{s(\alpha)})
\end{equation*}
is the Casimir element of $M_\alpha\otimes_{D_{s(\alpha)}} \mathrm{Hom}_{D_{s(\alpha)}^{op}}(M, D_{s(\alpha)})$. Similarly, we define $\overline{\alpha}\subseteq M_\alpha$ such that
\begin{equation*}
	c_{\alpha^*} = \sum_{a'\in \overline{\alpha}} {a'}^* \otimes a' \in M_\alpha^* \otimes_{D_{t(\alpha)}} M_\alpha
\end{equation*}
is the Casimir element of $M_\alpha^*\otimes_{D_{s(\alpha)}} M_\alpha$ associated to $\b_S$, i.e.
\begin{equation*}
	\sum_{a'\in \overline{\alpha}} \b(-\otimes {a'}^*)\otimes a'\in \mathrm{Hom}_{D_{t(\alpha)}}(M, D_{t(\alpha)}) \otimes_{D_{t(\alpha)}}  M_\alpha
\end{equation*}
is the Casimir element of $\mathrm{Hom}_{D_{t(\alpha)}}(M, D_{t(\alpha)}) \otimes_{D_{t(\alpha)}}  M_\alpha$. We also introduce the notation that $\overline{\alpha^*} = \{a^* \mid a\in \underline{\alpha} \}$ and $\underline{\alpha^*} = \{a^* \mid a\in \overline{\alpha} \}$. By construction we have
\begin{equation}\label{eq - M alpha basis presentation}
	M_\alpha = \bigoplus_{a\in \underline{\alpha}}aD_{s(\alpha)} = \bigoplus_{a'\in \overline{\alpha}} D_{s(\alpha)}a'.
\end{equation}
An immediate consequence of this is that
\begin{equation*}
	|\underline{\alpha}| = \dim_{D_{s(\alpha)}}M_\alpha,\quad |\overline{\alpha}| = \dim_{D_{t(\alpha)}}M_\alpha.
\end{equation*}

\begin{myrem}
	A useful observation is that
	\begin{equation*}
		\mathcal{J}(S, W) = \overline{\langle \partial_{a^*}(W) \mid a\in \underline{\alpha}, \alpha\in Q_1 \rangle} = \overline{\langle \partial_{a^*}(W) \mid a\in \overline{\alpha}, \alpha\in Q_1 \rangle}
	\end{equation*}
	which is a consequence of \eqref{eq - M alpha basis presentation}.
\end{myrem}

Up until now all constructions have been independent of the choice of bases. However this is not true for all results. Thus from now on, we assume that each species $S$ comes equipped with a choice of bases $\underline{\alpha}$ and $\overline{\alpha}$ for all $Q_1$. These choices are mainly used in Subsection~\ref{subsection - product of species} when computing the $3$-preprojective algebra of tensor product of species, in Section~\ref{Section - mutation} when defining mutation of species with potential, in Section~\ref{Section - nakayama automorphism} when defining a Nakayama automorphism of the mutated species with potential and in Section~\ref{Section - Jasso-muro plus examples} when computing examples.

\subsection{Basic Version of a Species with Potential}
The following general definition is motivated by \cite[Section I.6]{assem2006elements}.

\begin{mydef}\label{definition - basic idemp}
	Let $R$ be a $\K$-algebra. We say that an element $e\in R$ is a basic idempotent if the following hold:
	\begin{enumerate}[label = (\alph*)]
		\item $e = \sum_{i=1}^n e_i$, where $\{e_1, e_2, \dots, e_n \}\subseteq I$ and $I$ is a complete set of primitive orthogonal idempotents of $R$.
		\item $Re_i\not\cong Re_j$ for $i\not= j$ and for $s\in I$ there exists a $k\in \{1, 2, \dots, n\}$ such that $Re_s\cong Re_k$.
	\end{enumerate}
\end{mydef}

\begin{mylemma}\label{lemma - ReR = R}
	Given a basic idempotent $e\in R$ we have that $ReR = R$.
\end{mylemma}

\begin{proof}
	Let $s\in I$. Then by (b) in Definition~\ref{definition - basic idemp} there exists $k\in \{1, 2, \dots, n\}$ such that $Re_s \cong Re_k$. This implies that there exist $p, q\in R$ such that $pq = e_s$ and $qp = e_k$. Now
	\begin{equation*}
		Re_sR = Re_s^2R = RpqpqR = Rpe_kqR \subseteq Re_kR \subseteq ReR.
	\end{equation*}
	Hence
	\begin{equation*}
		R = R1_R R = \sum_{s\in I}Re_sR \subseteq \sum_{k\in \{1, 2, \dots, n\}} Re_kR = ReR.
	\end{equation*}
\end{proof}

The name stems from the fact that $eRe$ is a basic $\K$-algebra. The main use of basic idempotents is the following result.

\begin{mycor}\cite[Corollary I.6.10]{assem2006elements}\label{corollary - assem corollary morita equivalent}
	For a $\K$-algebra $R$ and a basic idempotent $e\in R$ we have that $eRe$ and $R$ are Morita equivalent.
\end{mycor}

\begin{myprop}\label{proposition - morita equivalent k-species}
	Let $(S, W)$ be a species with potential. There is a Morita equivalent $\K$-species with potential $(S', W')$, which we call the basic version of $(S, W)$. In other words, $T(S)\mathrm{-mod}\cong T(S')\mathrm{-mod}$, $\widehat{T}(S)\mathrm{-mod}\cong \widehat{T}(S')\mathrm{-mod}$ and $\mathcal{P}(S, W)\mathrm{-mod}\cong \mathcal{P}(S', W')\mathrm{-mod}$. 
\end{myprop}

\begin{proof}
	By Lemma~\ref{lemma - division algebra species} we can assume that $D_i\cong M_{m_i}(D_i')$ for some division $\K$-algebra $D_i'$ for all $i\in Q_0$.

	Let $E_{i, kl}\in M_{m_i}(D_i')$ be the matrix with zeros everywhere except at $(k, l)$ where it has a $1$. First note that
	\begin{equation*}
		\{e_{ij} = E_{i, jj} \mid i\in Q_0, j\in \{1, 2, \dots, n_i \} \}
	\end{equation*}
	is a complete set of primitive orthogonal idempotents of $T(S)$. Let $e = \sum_{i\in Q_0}e_{i1}$. We will first show that $e$ is a basic idempotent of $T(S)$. Condition (a) follows immediately. Note that $D_i'\subseteq T(S)e_{j1}$ if and only if $i=j$. This together with the isomorphisms
	\begin{equation*}
		\begin{tikzcd}
			T(S)e_{i1} \arrow[r, "-\cdot E_{i, 1j}", shift left] & T(S)e_{ij} \arrow[l, "-\cdot E_{i, j1}", shift left]
		\end{tikzcd}
	\end{equation*}
	shows condition (b). Note that $e$ is also a basic idempotent of $\widehat{T}(S)$ by the same argument. By Corollary~\ref{corollary - assem corollary morita equivalent} $eT(S)e$ is Morita equivalent to $T(S)$ and $e\widehat{T}(S)e$ is Morita equivalent to $\widehat{T}(S)$.

	Let $M' = eMe$. Then $D' = \bigoplus_{i\in Q_0}D_i'$ and $M'$ defines a species $S'$. Here we identify $D'$ with $eDe$, which is symmetric via $\t|_{eDe}$. Note that $S'$ is a $\K$-species by construction. We need to show that $T(S') \cong eT(S)e$. Note that
	\begin{equation*}
		eT(S)e = \bigoplus_{k\ge 0}eM^ke = \bigoplus_{i\in Q_0}eD_ie \oplus \bigoplus_{k\ge 1}eM^ke = \bigoplus_{i\in Q_0}D_i' \oplus \bigoplus_{k\ge 1}eM^ke.
	\end{equation*}
	Therefore it is enough to show that $eM^ke \cong (M')^k$. We claim that $T(S)eT(S) = T(S)$. Indeed $e_{ij} = E_{i, j1}eE_{i, 1j}$ and therefore $e_{ij}\in T(S)eT(S)$, which implies that $1_{T(S)} = \sum_{i\in Q_0}\sum_{j=1}^{n_i}e_{ij}\in T(S)eT(S)$. In particular $De \otimes_{eDe}eD \cong D$. Thus
	\begin{equation}\label{eq - M x M = Me x eM}
		M\otimes_DM \cong M\otimes_D De\otimes_{eDe}eD\otimes_D M \cong Me\otimes_{eDe}eM.
	\end{equation}
	Using that $eDe = D'$ and induction on $k$ yields
	\begin{equation*}
		eM^ke \cong (eMe)^{\otimes_{D'}k} = (M')^k.
	\end{equation*}
	Hence $T(S)$ and $T(S')$ are Morita equivalent. The argument above holds in the complete case as well, i.e.
	\begin{equation}\label{eq - Morita isomorphism equation}
		e\widehat{T}(S)e = \prod_{k\ge 0}eM^ke \cong \prod_{k\ge 0}(M')^k = \widehat{T}(S').
	\end{equation}
	In other words, $\widehat{T}(S)$ is Morita equivalent to $\widehat{T}(S')$. Lastly, we claim that $\mathcal{P}(S, W)$ is Morita equivalent to $\mathcal{P}(S', W')$ when we put $W'$ to be the potential given by the image of $eWe$ under the isomorphism in \eqref{eq - Morita isomorphism equation}. Before we prove this claim, let us first introduce some notation. We write $\varepsilon_l^S$ and $\varepsilon_r^S$ for the permutation operators for $S$ and similarly $\varepsilon_l^{S'}$ and $\varepsilon_r^{S'}$ for the permutation operators for $S'$. These induces operators $\varepsilon_c^S$ and $\varepsilon_c^{S'}$ for $S$ and $S'$ respectively. Thus we can define $\partial^S$ and $\partial^{S'}$. To show the claim, it is enough to show that
	\begin{equation}\label{eq - jacobian ideal times e}
		e\mathcal{J}(S, W)e = \overline{\langle \partial_{ea^*e}^{S'}(eWe) \mid a\in M_\alpha, \alpha\in Q_1\rangle} \subseteq e\widehat{T}(S)e.
	\end{equation}
	Let $a^*\in M^*$. Then we have the following.
	\begin{equation*}
		\begin{aligned}
			e\partial_{a^*}^S(W)e &= e(\partial_{a^*}^S)^l(\varepsilon_c^S(W))e = e(\partial_{a^*}^S)^l(\varepsilon_c^S(W)e) = e(\partial_{a^*}^S)^l(e\varepsilon_c^S(W)e) = \\
			&= (\partial_{ea^*e}^S)^l(e\varepsilon_c^S(W)e) = (\partial_{ea^*e}^{S'})^l(e\varepsilon_c^S(W)e).
		\end{aligned}
	\end{equation*}
	The last equality holds since $\b_{S'} = \b_S|_{e\widehat{T}(S)e}$. We are done if we show that
	\begin{equation*}
		e\varepsilon_c^S(W)e = \varepsilon_c^{S'}(eWe).
	\end{equation*}
	In other words, we need to show that
	\begin{equation}\label{eq - vareps^k = vareps^k}
		e(\varepsilon_l^S)^k(W)e = (\varepsilon_l^{S'})^k(eWe)
	\end{equation}
	for all $k\in \Z_{\ge 0}$. Immediately we have that \eqref{eq - vareps^k = vareps^k} is true for $k=0$. For all $k\ge 1$ we can use similar arguments as in Remark~\ref{remark - comapare to nquefack}. Let $k = 1$. Since $1 = \sum_{i\in Q_0, j\in \{1, 2, \dots, n_i\}} E_{i, j1}eE_{i, 1j}$ we can write $W = \sum_{i\in Q_0, j\in \{1, 2, \dots, n_i\}} WE_{i, j1}eE_{i, 1j} = \sum_{i\in Q_0, j\in \{1, 2, \dots, n_i\}} E_{i, j1}eWeE_{i, 1j}$. Now using the Casimir element $\sum_{k=1}^m x_k^* \otimes x_k$ on $M^*\otimes_D M$ as in \eqref{eq - Casimir elements of M M^*} we get
	\begin{equation*}
		\begin{aligned}
			e\varepsilon_l^S(W)e &= \sum_{k=1}^m e(\partial^S)^l_{x_i^*}(W)\otimes x_ie = \sum_{\substack{k\in \{1, 2, \dots, m\} \\ i\in Q_0 \\ j\in \{1, 2, \dots, n_i\}}} e(\partial^S)^l_{x_i^*}(E_{i, j1}eWeE_{i, 1j})\otimes x_ie = \\ 
			&= \sum_{\substack{k\in \{1, 2, \dots, m\} \\ i\in Q_0 \\ j\in \{1, 2, \dots, n_i\}}} (\partial^S)^l_{ex_i^*E_{i, j1}e}(eWe)\otimes eE_{i, 1j}x_ie = \varepsilon_l^{S'}(eWe).
		\end{aligned}
	\end{equation*}
	The last equality holds since
	\begin{equation*}
		\sum_{\substack{k\in \{1, 2, \dots, m\} \\ i\in Q_0 \\ j\in \{1, 2, \dots, n_i\}}} ex_i^*E_{i, j1}e \otimes eE_{i, 1j}x_ie
	\end{equation*}
	is a Casimir element in $eMe\otimes_{eDe}eM^*e$. This proves the equality in \eqref{eq - jacobian ideal times e}.
\end{proof}

\subsection{Higher Preprojective Algebras}
In this subsection we will review the (complete) higher preprojective algebra.

For the rest of this subsection, assume that $\Lambda$ is a $\K$-algebra of global dimension $d$.

\begin{mydef}\cite{Iyama2003maximal}
	The $d$-Auslander-Reiten translations $\tau_d$ and $\tau_d^-$ are defined by
	\begin{equation*}
		\begin{aligned}
			\tau_d&=D\mathrm{Ext}^d_\Lambda(-, \Lambda): \mathrm{mod}\Lambda \to \mathrm{mod}\Lambda, \\
			\tau_d^-&=\mathrm{Ext}^d_{\Lambda^{op}}(D-, \Lambda): \mathrm{mod}\Lambda \to \mathrm{mod}\Lambda.
		\end{aligned}
	\end{equation*}
\end{mydef}

\begin{mydef}\cite[2.11]{IO13Stablecateofhigherpreproj}\label{Definition - preprojective algebra BGL perspective}
	\space 
	\begin{enumerate}
		\item The $(d+1)$-preprojective algebra of $\Lambda$ is defined by
		\begin{equation}\label{Definition - eq preprojective algebra decomposition}
			\Pi_{d+1}(\Lambda) = \bigoplus_{i=0}^\infty \Hom_\Lambda(\Lambda, \tau^{-i}_d\Lambda)
		\end{equation}
		with multiplication
		\begin{equation*}
			\begin{aligned}
				\Hom_\Lambda(\Lambda, \tau^{-i}_d\Lambda)\times \Hom_\Lambda(\Lambda, \tau^{-j}_d\Lambda)&\rightarrow \Hom_\Lambda(\Lambda, \tau^{-(i+j)}_d\Lambda), \\
				(u, v)&\mapsto uv = (\tau^{-i}(v)\circ u: \Lambda\rightarrow \tau^{-(i+j)}\Lambda).
			\end{aligned}
		\end{equation*}
		\item The complete $(d+1)$-preprojective algebra of $\Lambda$ is defined as the completion of $\Pi_{d+1}(\Lambda)$, i.e.,
		\begin{equation}\label{Definition - eq complete preprojective algebra decomposition}
			\widehat{\Pi}_{d+1}(\Lambda) = \prod_{i=0}^\infty \Hom_\Lambda(\Lambda, \tau^{-i}_d\Lambda).
		\end{equation}
	\end{enumerate}
\end{mydef}

\begin{myrem}\label{remark - preprojective grading}
	There is a natural $\Z$-grading on $\Pi_{d+1}(\Lambda)$ given by the composition in Definition~\ref{Definition - eq preprojective algebra decomposition}. Moreover, $\widehat{\Pi}_{d+1}(\Lambda)$ is the completion of $\Pi_{d+1}(\Lambda)$ with respect to this grading.
\end{myrem}

\subsection{Product of Species and Higher Preprojective Algebra}\label{subsection - product of species}
We now present the explicit construction of the preprojective algebra $\Pi(S)$ of a species $S$ introduced by \cite{Dlab_1980}.

\begin{mydef}\cite{Dlab_1980}\label{Definition - preprojective algebra of a species}
	Let $S$ be an acyclic $\K$-species.
	\begin{enumerate}
		\item The double quiver $\overline{Q}$ is defined to be $\overline{Q}_0=Q_0$ and $\overline{Q}_1 = Q_1 \cup Q_1^*$ where
		\begin{equation*}
		Q_1^* = \{\alpha^*: j\rightarrow i\mid \alpha:i\rightarrow j\in Q_1 \}.
		\end{equation*}
		\item Let $\overline{S}$ be the species over $\overline{Q}$ where $\overline{D_i}=D_i$ for all $i\in \overline{Q}_0$ and $\overline{M}_\alpha = M_\alpha$ and $\overline{M}_{\alpha^*} = M_\alpha^*$ for all $\alpha\in Q_1$.
		\item For each $\alpha\in \overline{Q}_1$ let $c_\alpha$ be the Casimir element associated to $\b$ of $\overline{M}_\alpha\otimes_{D_{s(\alpha)}} \overline{M}_{\alpha^*}$, i.e. $c_\alpha = \sum_{a\in \underline{\alpha}} a\otimes a^*$. Define
		\begin{equation*}
		c_S = \sum_{\alpha\in \overline{Q}_1}\mathrm{sgn}(\alpha)c_\alpha,
		\end{equation*}
		where
		\begin{equation*}
		\mathrm{sgn}(\alpha) = \begin{cases}
		1, & \alpha\in Q_1 \\
		-1, & \mbox{else}.
		\end{cases}
		\end{equation*}
		We define the preprojective algebra of $S$ as $\Pi(S)=T(\overline{S})/\langle c_S\rangle$.
	\end{enumerate}
\end{mydef}

We will now present an explicit description $\Pi_3(T(S^1)\otimes_\K T(S^2))$ (see Proposition~\ref{prop - preprojective algebra of tensor species}), where $S^1$ and $S^2$ are two species, using the following definitions of $Q^1\tilde{\otimes}Q^2$, $S(S^1, S^2)$ and $W(S^1, S^2)$.

\begin{mydef}
	For finite quivers $Q^1$ and $Q^2$ we define $Q^1\tilde{\otimes} Q^2$ as the quiver given by
	\begin{equation*}
		\begin{aligned}
			(Q^1\tilde{\otimes} Q^2)_0 &= Q^1_0\times Q^2_0, \\
	  		(Q^1\tilde{\otimes} Q^2)_1 &= (Q^1_0\times Q^2_1) \coprod (Q^1_1\times Q^2_0)\coprod (Q^{1*}_1\times Q^{2*}_1).
		\end{aligned}
	\end{equation*}
	We extend the definition of source and target, $s$ and $t$, to be the identity on $Q_0$ for any given quiver $Q$, and define
	\begin{equation*}
		s, t: (Q^1\tilde{\otimes} Q^2)_1 \to (Q^1\tilde{\otimes} Q^2)_0
	\end{equation*}
	by $s(\alpha_1, \alpha_2) = (s(\alpha_1), s(\alpha_2))$ and $t(\alpha_1, \alpha_2) = (t(\alpha_1), t(\alpha_2))$.
\end{mydef}

\begin{mydef}\label{definition - S(S^1, S^2)}
	Let $S^i$ be an acyclic species for each $i\in \{1, 2\}$. 
	\begin{enumerate}
		\item Define $S(S^1, S^2) = (D, M)$ as the species over $Q^1\tilde{\otimes}Q^2$ given by
		\begin{equation*}
			\begin{aligned}
				D &= D^1\otimes_\K D^2, \\
				M &= M^1\otimes_\K D^2\oplus D^1\otimes_\K M^2\oplus {M^1}^*\otimes_\K {M^2}^*.
			\end{aligned}
		\end{equation*}
		Using $\t_1$ and $\t_2$ we identify
		\begin{equation*}
			M^* = {M^1}^*\otimes_\K D^2 \oplus D^1\otimes_\K {M^2}^* \oplus M^1\otimes_\K M^2.
		\end{equation*}
		We need to define
		\begin{equation*}
			\b_{S(S^1, S^2)}: M\otimes_{D^1\otimes_\K D^2} M^* \oplus M^*\otimes_{D^1\otimes_\K D^2} M \to D^1\otimes_\K D^2.
		\end{equation*}
		By Remark~\ref{remark - b determined by t} we define $\b_{S(S^1, S^2)}$ via the morphism
		\begin{equation*}
			\t := \t_1\otimes \t_2: D = D^1\otimes_\K D^2 \to \K.
		\end{equation*}
		In other words, $\b_{S(S^1, S^2)}$ is defined such that
		\begin{equation*}
			\b_{S(S^1, S^2)}((a\otimes b)\otimes (c\otimes d)) = \begin{cases}
				\b_{S^1}(a\otimes c)\otimes bd, & \text{if } a\in M^1, b\in D^2, c\in {M^1}^*, d\in D^2 \\
				\b_{S^1}(a\otimes c)\otimes bd, & \text{if } a\in {M^1}^*, b\in  D^2, c\in M^1, d\in D^2 \\
				ac\otimes \b_{S^2}(b\otimes d), & \text{if } a\in D^1, b\in M^2, c\in D^1, d\in {M^2}^* \\
				ac\otimes \b_{S^2}(b\otimes d), & \text{if } a\in D^1, b\in {M^2}^*, c\in D^1, d\in M^2 \\
				\b_{S^1}(a\otimes c)\otimes \b_{S^2}(b\otimes d), & \text{if } a\in M^1, b\in M^2, c\in {M^1}^*, d\in {M^2}^* \\
				\b_{S^1}(a\otimes c)\otimes \b_{S^2}(b\otimes d), & \text{if } a\in {M^1}^*, b\in {M^2}^*, c\in M^1, d\in M^2, \\
				0, & \text{otherwise.}
			\end{cases}
		\end{equation*}
		\item We define $W(S^1, S^2)\in T(S(S^1, S^2))$ as
		\begin{equation}\label{eq - potential for tensor species}
			\begin{aligned}
				W(S^1, S^2) =& \sum_{\substack{(\alpha, \beta)\in Q_1^1\times Q_1^2 \\ (a, b')\in \underline{\alpha}\times \overline{\beta}}} (a\otimes e_{s(\beta)})\otimes(a^*\otimes {b'}^*)\otimes(e_{t(\alpha)}\otimes b') + \\
				&-\sum_{\substack{(\alpha, \beta)\in Q_1^1\times Q_1^2 \\ (a', b)\in \overline{\alpha}\times \underline{\beta}}} (e_{s(\alpha)}\otimes b)\otimes({a'}^*\otimes b^*)\otimes(a'\otimes e_{t(\beta)}) = \\
				=& W_1 - W_2.
			\end{aligned}
		\end{equation}
		It is indeed a potential, i.e. $W(S^1, S^2)\in \mathcal{Z}_{D^1\otimes_\K D^2}(T(S^1, S^2))$, since $c_\alpha\in \mathcal{Z}_{D^1}(T(\overline{S^1}))$ and $c_\beta\in \mathcal{Z}_{D^2}(T(\overline{S^2}))$ for any arrows $\alpha\in Q_1^1$ and $\beta\in Q_1^2$ which is a consequence of \eqref{eq - casimir elements equations ass. to b}.
	\end{enumerate}
\end{mydef}

\begin{myrem}
	Since we have assumed that $\K$ is a perfect field, we have that $D^1\otimes_\K D^2$ is semisimple, which implies that $S(S^1, S^2)$ is a species. But $S(S^1, S^2)$ is not necessarily a $\K$-species even if $S^1$ and $S^2$ are $\K$-species, see \cite[Remark 11.6]{soderberg2022preprojective}. For another example, take $D_1 = \C$ and $D_2 = \C$ over $\R$. Their product then becomes $\C\otimes_\R \C \cong \C\times \C$ which is not a division $\K$-algebra. In practice we may replace $S(S^1, S^2)$ with a $\K$-species using Proposition~\ref{proposition - morita equivalent k-species}. Explanations and examples are presented in Section~\ref{Section - Jasso-muro plus examples}.
\end{myrem}

In the spirit of the above remark we note the following.

\begin{myprop}\label{proposition - product is division alg}
	Let $S^1$ and $S^2$ be two species. If $D^1_i\otimes_\K D^2_j$ is a division $\K$-algebra for all $i\in Q^1_0$ and $j\in Q^2_0$, then $S(S^1, S^2)$ is a $\K$-species. In particular, if $D^1 = \K^{|Q^1_0|}$ and $D^2 = \K^{|Q^2_0|}$, i.e. in the case when $T(S^1)$ and $T(S^2)$ are given as path algebras over $\K$, then $S(S^1, S^2)$ is a $\K$-species.
\end{myprop}

\begin{proof}
	Follows from Definition~\ref{Definition - Species}.
\end{proof}

\begin{myprop}\label{prop - preprojective algebra of tensor species}
	Let $S^i$ be acyclic $\K$-species for $i\in \{1, 2\}$. Let $\Lambda = T(S^1)\otimes_\K T(S^2)$. Then $\gldim(\Lambda) = 2$, $\Pi_3(\Lambda) \cong \mathcal{P}(S(S^1, S^2), W(S^1, S^2))$, and $\widehat{\Pi}_3(\Lambda) \cong \widehat{\mathcal{P}}(S(S^1, S^2), W(S^1, S^2))$.
\end{myprop}

\begin{proof}
	By \cite[Corollary 11.7]{soderberg2022preprojective} we have that
	\begin{equation*}
		\Pi_3(\Lambda) \cong T(S(S^1, S^2))/\langle R \rangle
	\end{equation*}
	where
	\begin{equation}\label{eq - definition of R}
		R = \{ c_{S^1}\otimes_\K M^2_1, M^1_1\otimes_\K c_{S^2}, [m^1\otimes 1_{D^2}, 1_{D^1}\otimes m^2]\mid m^1\in M^1_0, m^2\in M^2_0\}.
	\end{equation}
	Here $M_0^i$ and $M_1^i$ denote the degree $0$ respectively degree $1$ part of $M^i$ which is induced by the grading in Remark \ref{remark - preprojective grading} and
	\begin{equation*}
		c_{S^i} = \sum_{\alpha\in Q_1^i}\sgn(\alpha)c_\alpha
	\end{equation*}
	is the Casimir element of $T(\overline{S^i})$. The set $c_{S^1}\otimes_\K M_1^2$ is interpreted as a subset of $M$ via the isomorphism
	\begin{equation*}
		\begin{aligned}
			c_{S^1}\otimes_\K M_1^2 \subseteq (M^1_0\otimes_{D^1}M_1^1 \oplus M_1^1\otimes_{D^1}M_0^1)\otimes_\K M^2_1 &\cong \\
			\cong(M_0^1\otimes_\K D^2)\otimes_{D^1\otimes_\K D^2} (M_1^1\otimes_\K M_1^2) &\oplus (M_1^1\otimes_\K M_1^2)\otimes_{D^1\otimes_\K D^2}(M_0^1\otimes_\K D^2) \\
			(a\otimes b + c\otimes d)\otimes e &\mapsto (a\otimes 1_{D^2})\otimes (b\otimes e) \oplus (c\otimes e)\otimes (d\otimes 1_{D^2}).
		\end{aligned}
	\end{equation*}
	Similarly, the set $M_1^1\otimes_\K c_{S^2}$ is interpreted as a subset of $M$. 

	We are done if we show that $\mathcal{J}(S(S^1, S^2), W(S^1, S^2)) = \langle R\rangle$. Let us first compute $\partial_{x^*\otimes e_i}(W(S^1, S^2))$ for $x^*\in {M_\alpha^1}^*$ and $e_i\in D^2_i$. By \eqref{eq - casimir elements equations ass. to b} we have
	\begin{equation}\label{eq - tensor species eq 1}
		\begin{aligned}
			\partial_{x^*\otimes e_i}(W_1) &= \partial_{x^*\otimes e_i}^l(\varepsilon_cW_1) = \partial_{x^*\otimes e_i}^l(W_1) = \\
			&= \sum_{\substack{(\alpha, \beta)\in Q_1^1\times Q_1^2 \\ (a, b')\in \underline{\alpha}\times \overline{\beta}}} \b_{S(S^1, S^2)}((x^*\otimes e_i) \otimes (a\otimes e_{s(\beta)}))(a^*\otimes {b'}^*)\otimes(e_{t(\alpha)}\otimes b') = \\
			&= \sum_{\substack{(\alpha, \beta)\in Q_1^1\times Q_1^2 \\ s(\beta) = i \\ (a, b')\in \underline{\alpha}\times \overline{\beta}}} (\b_{S^1}(x^*\otimes a)a^*\otimes {b'}^*)\otimes(e_{t(\alpha)}\otimes b') = \\
			&= \sum_{\substack{\beta\in Q_1^2 \\ s(\beta) = i \\ b'\in \overline{\beta}}} (x^*\otimes {b'}^*)\otimes(e_{t(\alpha)}\otimes b') = \sum_{\substack{\beta\in Q_1^2 \\ s(\beta) = i}}x^*\otimes c_{\beta^*}.
		\end{aligned}
	\end{equation}
	For the other part of the potential, also using \eqref{eq - casimir elements equations ass. to b}, we get
	\begin{equation}\label{eq - tensor species eq 2}
		\begin{aligned}
		 	\partial_{x^*\otimes e_i}(W_2) &= \partial^r_{x^*\otimes e_i}(\varepsilon_cW_2) = \partial^r_{x^*\otimes e_i}(W_2) = \\
		 	&= \sum_{\substack{(\alpha, \beta)\in Q_1^1\times Q_1^2 \\ (a', b)\in \overline{\alpha}\times \underline{\beta}}} (e_{s(\alpha)}\otimes b)\otimes({a'}^*\otimes b^*)\b_{S(S^1, S^2)}((a'\otimes e_{t(\beta)})\otimes (x^*\otimes e_i)) = \\
		 	&= \sum_{\substack{(\alpha, \beta)\in Q_1^1\times Q_1^2 \\ t(\beta) = i \\ (a', b)\in \overline{\alpha}\times \underline{\beta}}} (e_{s(\alpha)}\otimes b)\otimes({a'}^*\b_{S^1}(a'\otimes x^*)\otimes b^*) = \sum_{\substack{\beta\in Q_1^2 \\ t(\beta) = i}} x^*\otimes c_\beta
		 \end{aligned}
	\end{equation}
	Thus combining \eqref{eq - potential for tensor species}, \eqref{eq - tensor species eq 1} and \eqref{eq - tensor species eq 2} yields
	\begin{equation*}
		\partial_{\sum_{i\in Q_0^2}x^*\otimes e_i}(W) = -x^*\otimes c_{S^2}. 
	\end{equation*}
	Similarly, one can show that $c_{S^1}\otimes y^*\in \mathcal{J}(S(S^1, S^2), W(S^1, S^2))$ for all $y^*\in {M^2}^*$.

	Let us now show that the commutator relations are in $\mathcal{J}(S(S^1, S^2), W(S^1, S^2))$. Let $x\in M_\alpha$ and $y\in M_\beta$ for some $\alpha\in Q_1^1$ and $\beta\in Q_1^2$. We claim that $\partial_{x\otimes y}(W) = (x\otimes e_{t(\beta)}) \otimes (e_{s(\alpha)}\otimes y) - (e_{t(\alpha)}\otimes y) \otimes (x\otimes e_{s(\beta)})$. First note that $\partial_{x\otimes y}(W(S^1, S^2)) = \partial^l_{x\otimes y}(\varepsilon_c(W(S^1, S^2))) = \partial^l_{x\otimes y}(\varepsilon_l(W(S^1, S^2)))$. We will first compute $\partial^l_{x\otimes y}(\varepsilon_l(W_1))$. Using \eqref{eq - casimir elements equations ass. to b} we have
	\begin{equation*}
		\begin{aligned}
			\varepsilon_l(W_1) &= \sum_{\substack{(\alpha, \beta)\in Q_1^1\times Q_1^2 \\ (a, a', b')\in \underline{\alpha}\times \overline{\alpha}\times \overline{\beta}}} \b_{S(S^1, S^2)}(((a')^*\otimes e_{s(\beta)}) \otimes (a\otimes e_{s(\beta)}))(a^*\otimes {b'}^*)\otimes(e_{t(\alpha)}\otimes b') \otimes (a'\otimes e_{s(\beta)}) = \\
			&= \sum_{\substack{(\alpha, \beta)\in Q_1^1\times Q_1^2 \\ (a', b')\in \overline{\alpha}\times \overline{\beta}}} ({a'}^*\otimes {b'}^*)\otimes(e_{t(\alpha)}\otimes b') \otimes (a'\otimes e_{s(\beta)}).
		\end{aligned}
	\end{equation*}
	Now we see that
	\begin{equation*}
		\begin{aligned}
			\partial^l_{x\otimes y}(\varepsilon_lW_1) &= \sum_{\substack{(\alpha, \beta)\in Q_1^1\times Q_1^2 \\ (a', b')\in \overline{\alpha}\times \overline{\beta}}} \b_{S(S^1, S^2)}((x\otimes y)\otimes ({a'}^*\otimes {b'}^*))(e_{t(\alpha)}\otimes b') \otimes (a'\otimes e_{s(\beta)}) = \\
			&= \sum_{\substack{(\alpha, \beta)\in Q_1^1\times Q_1^2 \\ (a', b')\in \overline{\alpha}\times \overline{\beta}}} (e_{t(\alpha)}\otimes \b_{S^2}(y\otimes {b'}^*)b') \otimes (\b_{S^1}(x\otimes {a'}^*)a'\otimes e_{s(\beta)}) = \\
			&= (e_{t(\alpha)}\otimes y)\otimes (x\otimes e_{s(\beta)}).
		\end{aligned}
	\end{equation*}
	The last equality holds due to \eqref{eq - casimir elements equations ass. to b}. Similarly, one computes that $\partial_{x\otimes y}(W_2) = (x\otimes e_{t(\beta)})\otimes (e_{s(\alpha)}\otimes y)$. Thus $[x\otimes 1_{D^2}, 1_{D^1}\otimes y]\in \mathcal{J}(S(S^1, S^2), W(S^1, W^2))$. Hence $\langle R\rangle\subseteq \mathcal{J}(S(S^1, S^2), W(S^1, S^2))$. We get the other inclusion, i.e. $\mathcal{J}(S(S^1, S^2), W(S^1, S^2))\subseteq \langle R\rangle$, by noting that $M^*$ is generated by elements $x^*\otimes 1_{D^2}$, $1_{D^1}\otimes y^*$ and $x\otimes y$ for $x\in M^1$, $y\in M^2$, $x^*\in {M^1}^*$ and $y^*\in {M^2}^*$.

	To obtain the second statement we show that $\widehat{\Pi}_3(\Lambda) \cong \widehat{T}(S(S^1, S^2))/\overline{\langle R \rangle}$. Equip $T(S(S^1, S^2))$ with the induced grading coming from $\Pi_3(\Lambda)$. This in turn yields a grading on the quiver where arrows in ${Q^1_1}^*\times {Q^2_1}^*$ are of degree $1$ and the rest of the arrows are of degree $0$. Note that the degree $0$ part of $Q^1\tilde{\otimes}Q^2$ is acyclic since $Q^1$ and $Q^2$ are acyclic. An immediate consequence is that $\dim_\K T(S(S^1, S^2))_i<\infty$, and thus $\widehat{T}(S(S^1, S^2)) \cong \prod_{i=1}^\infty T(S(S^1, S^2))_i$. Via this isomorphism we have that $\overline{\langle R\rangle}$ corresponds to $\prod_{i=1}^\infty \langle R\rangle_i$ because $R$ consists of homogeneous elements. Now $\Pi_3(\Lambda)_i = \frac{T(S(S^1, S^2))_i}{\langle R\rangle_i}$ gives the claim.
\end{proof}

\section{Mutation}\label{Section - mutation}
In this section we first recall the general theory of mutation in a cluster category. Then we define mutation for species with potential and prove part one of the main theorem of this paper (Theorem~\ref{Theorem A}).

\subsection{Mutation of Cluster Tilting Objects}
We say that a $\K$-linear triangulated category $\mathcal{C}$ is $2$-Calabi--Yau if it is hom-finite, i.e. each morphisms space is finite dimensional over $\K$, and there exists a functorial isomorphism
\begin{equation*}
	\mathrm{Hom}_\mathcal{C}(X, Y) \cong D\mathrm{Hom}_\mathcal{C}(Y, X[2])
\end{equation*}
for any $X, Y\in \mathcal{C}$. We will review results regarding mutations of self-injective cluster tilting objects in $\mathcal{C}$. In this case we say that $T\in \mathcal{C}$ is cluster tilting if
\begin{equation*}
	\mathrm{add}(T) = \{X\in \mathcal{C} \mid \mathrm{Ext}_\mathcal{C}^1(T, X) = 0 \}.
\end{equation*}
Throughout this section, assume that $\mathcal{C}$ is a hom-finite Krull-Schmidt $2$-Calabi--Yau category. We define $\mathrm{rad}_\mathcal{C}(X, Y) = \mathrm{rad}(\mathrm{Hom}_\mathcal{C}(X, Y))$, i.e. the radical of $\mathrm{Hom}_\mathcal{C}(X, Y)$.

\begin{mydef}\cite[Definition 4.3]{HerschendOsamu2011quiverwithpotential}
	We say that a cluster tilting object $T\in \mathcal{C}$ is self-injective if $\mathrm{End}_\mathcal{C}(T)$ is a finite dimensional self-injective algebra.
\end{mydef}

\begin{myprop}\cite[Proposition 3.6]{IO13Stablecateofhigherpreproj}\label{Prop - selfinjectie T=T[2]}
	Let $T = \bigoplus_{i\in I} T_i\in \mathcal{C}$ be a basic cluster tilting object with indecomposable summands $T_i$ and for some index set $I$.
	\begin{enumerate}
		\item $T$ is self-injective if and only if $T\cong T[2]$.
		\item In this case we define a permutation $\sigma: I\to I$ by $T_i[2]\cong T_{\sigma(i)}$. Then $\sigma$ gives the Nakayama permutation of $\mathrm{End}_\mathcal{C}(T)$.
	\end{enumerate}
\end{myprop}

Note that (2) follows from $\mathcal{C}$ being a $2$-Calabi-Yau category as
\begin{equation*}
	\mathrm{Hom}_\mathcal{C}(T_i, T) \cong D\mathrm{Hom}_\mathcal{C}(T, T_i[2]).
\end{equation*}

Let us now recall the definition of cluster tilting mutation. We follow \cite[Section 4.2]{HerschendOsamu2011quiverwithpotential}, which relies on \cite{IyamaYoshino2008Mutation}, mainly \cite[Definition 2.5, Theorem 5.1, Theorem 5.3]{IyamaYoshino2008Mutation}. Let $T\in \mathcal{C}$ be a basic cluster tilting object, and let $T = U\oplus V$ be a decomposition. We take triangles, known as exchange triangles,
\begin{equation*}
	U\xrightarrow{f}V'\xrightarrow{g} \lsup{*}U\to U[1] \quad \text{and}\quad U^*\xrightarrow{f'} V''\xrightarrow{g'} U\to U^*[1]
\end{equation*}
where $f$ and $f'$ are minimal left $(\mathrm{add}V)$-approximations and $g$ and $g'$ are minimal right $(\mathrm{add}V)$-approximations. Let
\begin{equation*}
	\mu_U^+(T) = U^*\oplus V\quad \text{and}\quad \mu_U^-(T) = \lsup{*}U\oplus V.
\end{equation*}
Notice that we do not assume that $U$ indecomposable.

\begin{myprop}\cite[Proposition 4.5]{HerschendOsamu2011quiverwithpotential}, \cite{IyamaYoshino2008Mutation}$\space$ 
	\begin{enumerate}
		\item $\mu_U^+(T)$ and $\mu_U^-(T)$ are basic cluster tilting objects in $\mathcal{C}$.
		\item If $U$ is indecomposable, then $\mu_U^+(T)\cong \mu_U^-(T)$, and is denoted by $\mu_U(T)$.
	\end{enumerate}
\end{myprop}

For a cluster-tilting object $T\in \mathcal{C}$ we define the quiver $Q_T$ of $T$ by:
\begin{enumerate}
	\item $(Q_T)_0 = I$,
	\item $(Q_T)_1 = \{\alpha: i\to j \mid i, j\in I, \mathrm{rad}_\mathcal{C}(T_i, T_j) / \mathrm{rad}_\mathcal{C}^2(T_i, T_j) \neq 0 \}$.
\end{enumerate}

The following general result shows that if $\mathrm{End}_\mathcal{C}(\mu_{T_k}(T))$ is given by a species with potential, its semisimple part is the same as $\mathrm{End}_\mathcal{C}(T)$.

\begin{myprop}\label{proposition - end mod rad equivalence}
	Let $T = \bigoplus_{i\in I}T_i\in \mathcal{C}$ be a basic cluster tilting object with indecomposable summands $T_i$. Assume that there are no loops at vertex $k\in I$ in the quiver of $T$ and in the quiver of $\mu_{T_k}(T)$. If
	\begin{equation}\label{eq - exhancge triangle end iso}
		T_k\xrightarrow{f_k} U_k\xrightarrow{g_k} T_k^*\xrightarrow{\delta_k} T_k[1]
	\end{equation}
	is an exchange triangle in $\mathcal{C}$ then
	\begin{equation*}
		\phi: \mathrm{End}_\mathcal{C} (T_k)/\mathrm{rad}(\mathrm{End}_\mathcal{C} (T_k)) \cong \mathrm{End}_\mathcal{C} (T_k^*)/\mathrm{rad}(\mathrm{End}_\mathcal{C} (T_k^*)).
	\end{equation*}
\end{myprop}

\begin{proof}
	We apply $\mathrm{Hom}_\mathcal{C}(- ,T_k[1])$ to the triangle \eqref{eq - exhancge triangle end iso} to get an exact sequence
	\begin{equation*}
		\cdots \to \mathrm{Hom}_\mathcal{C}(U_k[1], T_k[1])\to \mathrm{End}_\mathcal{C}(T_k[1]) \to \mathrm{Hom}_\mathcal{C}(T_k^*, T_k[1]) \to \mathrm{Hom}_\mathcal{C}(U_k, T_k[1]) \to \cdots.
	\end{equation*}
	Note that $\mathrm{Hom}_\mathcal{C}(U_k, T_k[1])\cong \mathrm{Ext}^1_\mathcal{C}(U_k, T_k) = 0$ since $T$ is a cluster tilting object. Since $Q_T$ has no loop at $k$ and $f_k$ is a left $\mathrm{add}(T/T_k)$-approximation, we get that the sequence
	\begin{equation*}
		\mathrm{Hom}_\mathcal{C}(U_k[1], T_k[1])\to \mathrm{rad}(\mathrm{End}_\mathcal{C}(T_k[1]))\to 0
	\end{equation*}
	is exact. Combining these two results yields
	\begin{equation*}
		\mathrm{End}_\mathcal{C} (T_k[1])/\mathrm{rad}(\mathrm{End}_\mathcal{C} (T_k[1])) \cong \mathrm{Hom}_\mathcal{C}(T_k^*, T_k[1]).
	\end{equation*}
	
	Similarly, if we apply $\mathrm{Hom}_\mathcal{C}(T_k^*, -)$ to the triangle \eqref{eq - exhancge triangle end iso} we get an exact sequence
	\begin{equation*}
		\cdots \to \mathrm{Hom}_\mathcal{C}(T_k^*, U_k) \to \mathrm{End}_\mathcal{C}(T_k^*) \to \mathrm{Hom}_\mathcal{C}(T_k^*, T_k[1]) \to \mathrm{Hom}_\mathcal{C}(T_k^*, U_k[1]) \to \cdots.
	\end{equation*}
	Again, we have that $\mathrm{Hom}_\mathcal{C}(T_k^*, U_k[1]) = \mathrm{Ext}^1_\mathcal{C}(T_k^*, U_k) = 0$ since $\mu_{T_k}^-(T)$ is a basic cluster tilting object in $\mathcal{C}$. Since $Q_{\mu_{T_k}(T)}$ has no loop at $k$ and $g_k$ is a right $\mathrm{add}(T/T_k)$-approximation the sequence
	\begin{equation*}
		\mathrm{Hom}_\mathcal{C}(T_k^*, U_k) \to \mathrm{rad}(\mathrm{End}_\mathcal{C}(T_k^*))\to 0
	\end{equation*}
	is exact and therefore
	\begin{equation*}
		\mathrm{End}_\mathcal{C} (T_k^*)/\mathrm{rad}(\mathrm{End}_\mathcal{C} (T_k^*)) \cong \mathrm{Hom}_\mathcal{C}(T_k^*, T_k[1]).
	\end{equation*}
	Hence
	\begin{equation}\label{eq - end isomorphism}
		\mathrm{End}_\mathcal{C} (T_k[1])/\mathrm{rad}(\mathrm{End}_\mathcal{C} (T_k[1])) \cong \mathrm{End}_\mathcal{C} (T_k^*)/\mathrm{rad}(\mathrm{End}_\mathcal{C} (T_k^*))
	\end{equation}
	and using that $[1]$ is a fully-faithful functor we get the result. The morphism $\phi$ is defined by composing $[1]$ with the isomorphism \eqref{eq - end isomorphism}. To see that $\phi$ is indeed a ring homomorphism we observe the following. If $h\in \mathrm{End}_\mathcal{C} (T_k)/\mathrm{rad}(\mathrm{End}_\mathcal{C} (T_k))$ and $h^*\in \mathrm{End}_\mathcal{C} (T_k^*)/\mathrm{rad}(\mathrm{End}_\mathcal{C} (T_k^*))$, then $h[1]\circ \delta_k = \delta_k \circ h$ if and only if $\phi(h) = h^*$. Thus
	\begin{equation*}
		(h_1 \circ h_2)[1]\circ \delta_k = h_1[1]\circ h_2[1]\circ \delta_k = h_1[1]\circ \delta_k \circ  \phi(h_2) = \delta_k \circ  \phi(h_1)\circ \phi(h_2),
	\end{equation*}
	where $h_1, h_2\in \mathrm{End}_\mathcal{C} (T_k)/\mathrm{rad}(\mathrm{End}_\mathcal{C} (T_k))$. In other words, $\phi(h_1\circ h_2) = \phi(h_1)\circ \phi(h_2)$.
\end{proof}

\begin{myprop}\cite[Proposition 4.6]{HerschendOsamu2011quiverwithpotential}\label{Prop - U=U[2] implies selfinjective}
	Let $T$ be a self-injective cluster tilting object in $\mathcal{C}$. If $U\cong U[2]$, then $\mu_U^+(T)$ and $\mu_U^-(T)$ are self-injective cluster tilting objects in $\mathcal{C}$.
\end{myprop}

\begin{myprop}\cite[Proposition 4.8]{HerschendOsamu2011quiverwithpotential}\label{prop - no arrows mu + and mu -}
	If there are no arrows between vertices corresponding to $T_1, T_2, \dots, T_m$ in the quiver $Q_T$, then we have $\mu_U^+(T)\cong \mu_U^-(T) \cong \mu_{T_m}\circ \mu_{T_{m-1}}\circ \cdots \circ \mu_{T_1}(T)$, where $U = \bigoplus_{i = 1}^mT_i$.
\end{myprop}

\subsection{Mutation of Species with Potential}
Throughout the rest of this article we will assume that all species are $\K$-species. Let $(S, W)$ be a species with potential over a quiver $Q$, with no loops or $2$-cycles at $k$. Assume that $W$ does not start at vertex $k$, i.e. $e_kW = 0 = We_k$ where $e_k\in D_k$ is the identity element, we define a semi-mutation and mutation as in \cite{Nquefack2012PotentialsJacobian}. Note that we can always assume this since $Q$ has no loops, at $k$.

\begin{mydef}\cite[Definition 7.1]{Nquefack2012PotentialsJacobian}
	Assume that $Q$ has no loops or $2$-cycles. The semi-mutation of $(S, W)$ at $k$ is again a species with potential $\mu_k(S, W) = (S', W')$ over a quiver $Q'$ obtained by:
	\begin{enumerate}
		\item Adding a new arrow $[\alpha\beta]$ for all $\alpha, \beta\in Q_1$ such that $s(\alpha) = k = t(\beta)$.
		\item Replacing $\alpha: k\to i$ and $\beta: j\to k$ with arrows $\alpha^*: i\to k$ and $\beta^*: k\to j$ respectively.
	\end{enumerate}
	We get $S'$ from $S$ by keeping $D$ and modifying $M$ by:
	\begin{enumerate}
		\item Setting $M_{[\alpha\beta]} = M_\alpha\otimes_{D_k}M_\beta$.
		\item Setting $M_{\alpha^*} = (M_\alpha)^*$ and $M_{\beta^*} = (M_\beta)^*$.
	\end{enumerate}
	We let $W' = [W] + \Delta$ where:
	\begin{enumerate}
		\item $[W]$ is obtained from $W$ by substituting $m_\alpha \otimes m_\beta\in M_\alpha\otimes_{D_k} M_\beta$ appearing in $W$ with $[m_\alpha \otimes m_\beta]\in M_{[\alpha\beta]}$ for all $m_\alpha\in M_\alpha$ and $m_\beta\in M_\beta$ such that $s(\alpha) = k = t(\beta)$.
		\item Let
		\begin{equation*}
			c_{\alpha, \beta} = \sum_{(a, b)\in \overline{\alpha}\times\underline{\beta}} a^* [a\otimes b] b^*.
		\end{equation*}
		and define
		\begin{equation*}
			\Delta = \sum_{\substack{\alpha, \beta\in Q_1 \\ s(\alpha) = k = t(\beta)}} c_{\alpha,\beta}.
		\end{equation*}
	\end{enumerate} 
\end{mydef}

\begin{myrem}
	The mutation of a species with potential at $k$ is given as the reduced part of $\mu_k(S, W)$. We will not give any details about the reduction process in this paper. The important detail is that the mutation and the semi-mutation at $k$ have isomorphic Jacobian algebras. We only need the Jacobian algebras in this paper, and therefore it is enough to consider the semi-mutation. We refer the reader to \cite{Nquefack2012PotentialsJacobian} for more details.
\end{myrem}

We keep $\t: D\to \K$, which gives Casimir elements for $\mu_k(S, W)$. For the species $\mu_k(S, W)$ we keep the same bases $\underline{\alpha}$ and $\overline{\alpha}$ when $\alpha\in Q_1$ is not adjacent to $k$. Moreover, if $\alpha: k\to i$ and $\beta: j\to k$ we choose the bases $\underline{\alpha^*}$ and $\overline{\alpha^*}$ for $\alpha^*$, the bases $\underline{\beta^*}$ and $\overline{\beta^*}$ for $\beta^*$, and finally for $[\alpha\beta]$ we choose the bases
\begin{equation*}
	\begin{aligned}
		\underline{[\alpha\beta]} &= \{a\otimes b \mid a\in \underline{\alpha}, b\in \underline{\beta} \}, \\
		\overline{[\alpha\beta]} &= \{a\otimes b \mid a\in \overline{\alpha}, b\in \overline{\beta} \}.
	\end{aligned}
\end{equation*}

\begin{myex}
	Let $Q$ be the quiver
	\begin{equation*}
		\begin{tikzcd}
			1 \arrow[r] & 2 \arrow[r] & 3 \arrow[r] & 4. \\
		\end{tikzcd}
	\end{equation*}
	Let $S$ be the $\R$-species over $Q$ of Dynkin type $F_4$ given by 
	\begin{equation*}
		\C \xrightarrow{\lsub{2}\C_1}\C \xrightarrow{\lsub{3}\C_2}\R \xrightarrow{\lsub{4}\R_3} \R.
	\end{equation*}
	We can view $S$ as a species with potential by choosing the potential $W = 0$. The mutated species with potential $\mu_3(S, W) = (S', W')$ is given by
	\begin{equation*}
		Q': \begin{tikzcd}
			& & 3 \arrow[dl] \\
			1 \arrow[r] & 2 \arrow[rr] & & 4 \arrow[lu]
		\end{tikzcd}, \qquad S' :\begin{tikzcd}
			& & \R \arrow[ld, "\lsub{2}\C_3"'] \\
			\C \arrow[r, "\lsub{2}\C_1"] & \C \arrow[rr, "\lsub{4}\C_2"] & &  \R \arrow[lu, "\lsub{3}\R_4"']
		\end{tikzcd}
	\end{equation*}
	where we have made the identifications
	\begin{equation*}
		\begin{aligned}
			\lsub{3}\C_2^* &\cong \C = \lsub{2}\C_3, \quad (1^*\mapsto 1, i^*\mapsto -i) \\
			\lsub{4}\R_3^* &\cong \R = \lsub{3}\R_4, \quad (f\mapsto f(1)) \\
			\lsub{4}\R_3\otimes_\R \lsub{3}\C_2 &\cong \C = \lsub{4}\C_2, \quad (r\otimes s\mapsto rs)
		\end{aligned}
	\end{equation*}
	and the potential is given as $W' = \lsub{3}1_4\otimes \lsub{4}1_2\otimes \lsub{2}1_3$.
\end{myex}

\subsection{Relation Between Mutation of Cluster Tilting Objects and of Species with Potentials}\label{subsection - relation}
In this section we will prove the following theorem, which is a generalisation of \cite[Theorem 5.2]{BIRS2011Mutation}. We need the following assumption. Let $T\in \mathcal{C}$ be a basic cluster tilting object. Following \cite[see Section 4]{BIRS2011Mutation} we say that $T$ satisfies the vanishing condition at $k$ if
\begin{equation*}
	\mathrm{Hom}_\Lambda(\mathrm{Ext}^1_\Lambda(D\Lambda, S_k), S_k) = 0
\end{equation*}
for $\Lambda = \mathrm{End}_\mathcal{C}(T)^{op}$ and the simple $\Lambda$-module $S_k$.

A vertex $k\in Q_0$ is mutable if the following holds:
\begin{enumerate}
	\item $k$ is not contained in loops or $2$-cycles in $Q$.
	\item The isomorphism $\Phi: \mathrm{End}_\mathcal{C}(T)^{op} \cong \mathcal{P}(S, W)$ satisfies the mutation compatibility condition at $k$ defined in Definition~\ref{definition - mut condition at k}.
\end{enumerate}

\begin{mythm}\label{Theorem A}
	Let $\mathcal{C}$ be a $2$-Calabi--Yau triangulated category and $T = \bigoplus_{i = 1}^n T_i\in \mathcal{C}$ a basic cluster tilting object with indecomposable summands $T_i$, that satisfies the vanishing condition at $i$ for all $i\in Q_0$. If $\mathrm{End}_\mathcal{C}(T)^{op}\cong \mathcal{P}(S, W)$ for a reduced species with potential $(S, W)$ and $k\in Q_0$ is mutable, then $\mathrm{End}_\mathcal{C}(\mu_{T_k}(T))^{op}\cong \mathcal{P}(\mu_k(S, W))$.
\end{mythm}

\begin{myrem}
	Note that if $\mathrm{End}_\mathcal{C}(T)^{op}$ is a self-injective algebra, then it automatically satisfies the vanishing condition at $i$ for all $1\le i\le n$.
\end{myrem}

\begin{myrem}
	It follows from Remark~\ref{remark - mutation comp involutative} that $k\in Q_0$ is mutable for $\mu_k(S, W)$.
\end{myrem}

To prove Theorem~\ref{Theorem A} we will use the same construction as in \cite{BIRS2011Mutation}, but we have to make modification in some places to compensate for the more rich structure of species compared to quivers. We will cite \cite{BIRS2011Mutation} where we have used the same arguments but formulated for species. We will skip some arguments where it follows by exactly the same arguments. The main difference in the proof is that we require that $\Phi$ satisfies the mutation compatibility condition at $k$. All quivers with potential automatically satisfy the mutation compatibility condition due to Remark~\ref{remark - (Q, W) satisfy mut condition}.

\begin{mylemma}\label{Lemma - kernel ideal rewritten}\cite[Lemma 3.2]{BIRS2011Mutation}
	If $(S, W)$ be a species with potential. We have
	\begin{equation*}
		\mathcal{J}(S, W) = \sum_{\substack{\beta \in Q_1 \\ b\in \underline{\beta}}} \widehat{T}(S)\cdot\partial_{b^*}(W) + \mathcal{J}(S, W)\cdot M.
	\end{equation*}
\end{mylemma}

\begin{proof}
	Define
	\begin{equation*}
		I = \langle \partial_{b^*}(W) \mid b\in \underline{\beta}, \beta\in Q_1 \rangle.
	\end{equation*}
	By definition, we have that $\overline{I} = \mathcal{J}(S, W)$. Let $y\in \mathcal{J}(S, W)$. Naturally we have that
	\begin{equation}\label{eq - split of ideal I}
		\begin{aligned}
			I &= \sum_{\substack{b\in \underline{\beta} \\ \beta\in Q_1}} \widehat{T}(S)\cdot \partial_{b^*}(W)\cdot D + \widehat{T}(S)\cdot I\cdot \widehat{T}(S)_{\ge 1} = \\
			&= \sum_{\substack{b\in \underline{\beta} \\ \beta\in Q_1}} \widehat{T}(S)\cdot \partial_{b^*}(W) + I\cdot M.
		\end{aligned}
	\end{equation}
	According to the definition of the closure of $I$ we can write $y = x_l + x_l'$, where $x_l\in I$ and $x_l'\in \mathcal{J}_{\widehat{T}(S)}^l$ for any $l>0$. Let us define $y_l = x_{l+1} - x_l = x_l' - x_{l+1}'$. If we define $y_0 = x_1$ we get that $\sum_{i=0}^{l}y_i = x_{l+1} = y - x_{l+1}'$. Taking the limit we see that $y = \sum_{i\ge 0}y_i$. Due to \eqref{eq - split of ideal I} we have
	\begin{equation*}
		y_l \in \sum_{\substack{b\in \underline{\beta} \\ \beta\in Q_1}}y_{l, b}'\partial_{b^*}(W) + I\cdot M
	\end{equation*}
	where $y_{l, b}'\in \widehat{T}(S)$. Thus taking the infinite sum over $l$ we get
	\begin{equation*}
		y = \sum_{\substack{b\in \underline{\beta} \\ \beta\in Q_1}}\left(\sum_{l\ge 0}y_{l, b}'\right)\partial_{b^*}(W) + \overline{I}\cdot M
	\end{equation*}
	by using the definition of the closure of $I$. Noting that $\sum_{l\ge 0}y_{l, b}'\in \widehat{T}(S)$ by definition of $\widehat{T}(S)$ we are done.
\end{proof}

\begin{mydef}\cite[Definition 4.4]{BIRS2011Mutation}
	Let $T\in \mathcal{C}$ an object.
	\begin{enumerate}
		\item We call a complex
		\begin{equation*}
			U_1 \xrightarrow{f_1}U_0\xrightarrow{f_0}X \to 0
		\end{equation*}
		in $\mathrm{add}(T)$ a right $2$-almost split sequence if
		\begin{equation*}
			\mathrm{Hom}_\mathcal{C}(T, U_1) \xrightarrow{f_1\circ -} \mathrm{Hom}_\mathcal{C}(T, U_0) \xrightarrow{f_0\circ -} \mathrm{rad}_\mathcal{C}(T, X) \to 0
		\end{equation*}
		is exact. In other words, $f_0$ is right almost split in $\mathrm{add}(T)$ and $f_1$ is a pseudo-kernel of $f_0$ in $\mathrm{add}(T)$. 
		\item Dually, we call a complex
		\begin{equation*}
			0\to X \xrightarrow{f_2}U_1\xrightarrow{f_1} U_0
		\end{equation*}
		in $\mathrm{add}(T)$ a left $2$-almost split sequence if
		\begin{equation*}
			\mathrm{Hom}_\mathcal{C}(U_0, T)\xrightarrow{-\circ f_1}\mathrm{Hom}_\mathcal{C}(U_1, T)\xrightarrow{-\circ f_2}\mathrm{rad}_\mathcal{C}(X, T)\to 0
		\end{equation*}
		is exact. In other words, $f_2$ is left almost split and $f_1$ is a pseudo-cokernel of $f_2$ in $\mathrm{add}(T)$.
		\item We call a complex
		\begin{equation*}
			0\to X\xrightarrow{f_2}U_1\xrightarrow{f_1}U_0\xrightarrow{f_0}X\to 0
		\end{equation*}
		in $\mathrm{add}(T)$ a weak $2$-almost split sequence if $U_1\xrightarrow{f_1}U_1\xrightarrow{f_0}X\to 0$ is right $2$-almost split and $0\to X\xrightarrow{f_2}U_1\xrightarrow{f_1}U_1$ is left $2$-almost split.
	\end{enumerate}
\end{mydef}

\begin{mydefprop}\label{lemma - complex and derivation}
	Let $(S, W)$ be a species with potential.
	\begin{enumerate}
		\item $\partial_{b^*, a^*}(W) := \partial^r_{a^*}(\partial_{b^*}(W)) = \partial^l_{b^*}(\partial_{a^*}(W))$.
		\item $\partial_{\underline{\alpha^*}}(W)\overline{\alpha} := \sum_{a\in \overline{\alpha}} \partial_{a^*}(W)a = W$.
		\item $\underline{\beta}\partial_{\overline{\beta^*}}(W) := \sum_{b\in \underline{\beta}} b\partial_{b^*}(W) = W$.
	\end{enumerate}
	Furthermore, we also introduce the notation that
	\begin{equation*}
		\partial_{\overline{\beta^*}, \underline{\alpha^*}}W: P_{s(\beta)}^{|\underline{\beta}|} \to P_{t(\alpha)}^{|\overline{\alpha}|}
	\end{equation*}
	is defined as the matrix $\partial_{\overline{\beta^*}, \underline{\alpha^*}}W := (\partial_{b^*, a^*}W)_{a^*\in \underline{\alpha^*}, b^* \in \overline{\beta^*}}$. This notation allows us to write
	\begin{equation*}
		\partial_{\overline{\beta^*}, \underline{\alpha^*}}W = \partial^r_{\underline{\alpha^*}}(\partial_{\overline{\beta^*}}(W)) = \partial^l_{\overline{\beta^*}}(\partial_{\underline{\alpha^*}}(W)).
	\end{equation*}
\end{mydefprop}

\begin{proof}
	\begin{enumerate}
		\item Follows by \cite[Proposition 3.2]{Nquefack2012PotentialsJacobian}.
		\item[(2)-(3)] Follows from \eqref{eq - casimir elements equations ass. to b}.
	\end{enumerate}
\end{proof}

One immediate consequence is the following lemma.

\begin{mylemma}\label{lemma - complex and derivation 2}
	Let $(S, W)$ be a species with potential.
	\begin{enumerate}
		\item $\partial_{\overline{\beta^*}, \underline{\alpha^*}}(W)\overline{\alpha}:= \sum_{a\in \overline{\alpha}, b\in \underline{\beta}} \partial_{b^*, a^*}(W)a = \partial_{\overline{\beta^*}}(W)$.
		\item $\underline{\beta}\partial_{\overline{\beta^*}, \underline{\alpha^*}}(W):= \sum_{a\in \overline{\alpha}, b\in \underline{\beta}} b\partial_{b^*, a^*}(W) = \partial_{\underline{\alpha^*}}(W)$.
	\end{enumerate}
\end{mylemma}

Let $(S, W)$ be a species with potential with no loops and no $2$-cycles at $k$ and let $(S', W') = \mu_k(S, W)$. If $\alpha,\beta\in Q_1$ such that $t(\alpha) = s(\beta) = k$, then we define
\begin{equation*}
	[\overline{\beta}\underline{\alpha}] := (\partial_{b, a}W')_{a\in \underline{\alpha}, b \in \overline{\beta}} = (\partial_{b, a}\Delta)_{a\in \underline{\alpha}, b \in \overline{\beta}} = ([b\otimes a])_{a\in \underline{\alpha}, b \in \overline{\beta}}.
\end{equation*}

We denote the category of finitely generated projective $\widehat{T}(S)$-modules by $\mathrm{proj}\widehat{T}(S)$. Note that any $\K$-linear functor $\Phi: \mathrm{proj}\widehat{T}(S) \to \mathcal{C}$ can be viewed as a $\K$-linear functor $\Phi: \mathrm{proj}\widehat{T}(S)^{op}\to \mathcal{C}^{op}$. Let $P = \bigoplus_{i\in Q_0} P_i\in \mathrm{proj}\widehat{T}(S)$ and $T = \Phi(P)$ where $P_i = \widehat{T}(S)e_i$. The functor $\Phi$ is determined on morphisms by the induced ring morphism $\Phi: \mathrm{End}_{\mathrm{proj}(\widehat{T}(S))}(P)^{op} \to \mathrm{End}_\mathcal{C}(T)^{op}$. Since $\mathrm{End}_{\mathrm{proj}(\widehat{T}(S))}(P)^{op}\cong \widehat{T}(S)$, this information is encoded in a ring morphism $\Phi: \widehat{T}(S)\to \mathrm{End}_\mathcal{C}(T)^{op}$. We will switch freely between these perspectives. Furthermore, we have the following interpretation.

\begin{mylemma}\label{lemma - bijection sets algebra and pairs}
	Assume $T = \bigoplus_{i\in Q_0}T_i$ is basic where all $T_i$ are indecomposable. There is a bijection between the sets of
	\begin{enumerate}
		\item algebra morphisms $\Phi: \widehat{T}(S)\to \End_\mathcal{C}(T)^{op}$ such that $e_i\mapsto 1_{T_i}$, and $\Phi(\rad(\widehat{T}(S)))\subseteq \rad(\mathrm{End}_\mathcal{C}(T)^{op})$.
		\item pairs $(\Phi_0, \Phi_1)$, where
		\begin{enumerate}
			\item $\Phi_0: D\to \End_\mathcal{C}(T)^{op}$ such that $\Phi_0|_{D_i}: D_i\to \End_\mathcal{C}(T_i)^{op}$ is an algebra morphism,
			\item $\Phi_1: M\to \rad_\mathcal{C}(T, T)$ such that $\Phi_1|_{M_\alpha}: M_\alpha \to \rad_\mathcal{C}(T_{t(\alpha)}, T_{s(\alpha)})$ is a $D$-bimodule morphism.
		\end{enumerate}
	\end{enumerate}
\end{mylemma}

\begin{proof}
	If we have $\Phi: \widehat{T}(S)\to \End_\mathcal{C}(T)^{op}$, then we define $\Phi_0 = \Phi|_D$ and $\Phi_1 = \Phi|_M$. They are naturally a $D$-algebra morphism and $D$-bimodule morphism respectively. By assumption $\Phi_1(M) \subseteq \rad_\mathcal{C}(T, T)$.

	For the other direction, we let $(\Phi_0, \Phi_1)$ be a pair with the given assumptions. We can define $\Phi: T(S)\to \End_\mathcal{C}(T)^{op}$ by setting $\Phi|_D = \Phi_0$ and $\Phi|_M = \Phi_1$ and extend it on the whole tensor algebra $T(S)$ by the property of tensor algebras.  Since $\mathcal{C}$ is hom-finite, and $\Phi(M)\subseteq \rad_\mathcal{C}(T, T)$, we have that $\Phi(M^{\otimes N}) = 0$ for some $N\in \Z_{\ge 0}$. Thus we can extend $\Phi$ to $\widehat{T}(S)$ by setting
	\begin{equation*}
		\Phi(\sum_{i\ge 0}m_i) = \sum_{i\ge 0}\Phi(m_i) = \sum_{i=1}^N\Phi(m_i)
	\end{equation*}
	where $m_i\in M^{\otimes i}$. This completes the proof.
\end{proof}

\begin{myrem}
	Note that in (2) the condition $\Phi(\rad(\widehat{T}(S)))\subseteq \rad(\mathrm{End}_\mathcal{C}(T)^{op})$ is automatic if $\Phi$ is surjective.
\end{myrem}

As abelian groups $\Hom_{\mathrm{proj}\widehat{T}(S)}(P_i, P_j) = \Hom_{\mathrm{proj}\widehat{T}(S)}(P_i, P_j)^{op}$ and $\Hom_\mathcal{C}(T_i, T_j) = \Hom_\mathcal{C}(T_i, T_j)^{op}$. We write the superscript $op$ to indicate that we consider these together with the $D_i$-$D_j$-bimodule structures defined via the inclusions $D_i\to \End_{\mathrm{proj}\widehat{T}(S)}(P_i)^{op}$ and $\Phi_0|_{D_i}: D_i\to \End_\mathcal{C}(T_i)^{op}$ respectively for $i\in Q_0$. In particular, we identity $D_i$-$D_j$-bimodules $\Hom_{\mathrm{proj}(\widehat{T}(S))}(P_i, P_j)^{op} = e_i\widehat{T}(S)e_j$.

To avoid confusion we will use ``$\circ$'' for the composition in $\mathrm{proj}\widehat{T}(S)$ and $\mathcal{C}$ and ``$\cdot$'' for composition in $\mathrm{proj}\widehat{T}(S)^{op}$ and $\mathcal{C}^{op}$. We also interpret $\mathrm{Hom}_\mathcal{C}(X, Y)^{op}$ to be equal to $\mathrm{Hom}_\mathcal{C}(X, Y)$ as abelian groups.

\begin{mythm}\label{Theorem - iso and weak 2}\cite[Theorem 4.5]{BIRS2011Mutation}
	Let $(S, W)$ be a species with potential. Assume that $\Phi: \mathrm{proj}\widehat{T}(S)\to \mathcal{C}$ is a $\K$-linear functor and let $T = \bigoplus_{i\in Q_0}T_i$, where $T_i = \Phi(P_i)$ is indecomposable. Also assume that
	\begin{equation*}
		\dim_\K \mathrm{End}_\mathcal{C}(T_i)/\mathrm{rad}(\mathrm{End}_\mathcal{C}(T_i)) = \dim_\K D_i
	\end{equation*}
	for all $i\in Q_0$. Then the following are equivalent:
	\begin{enumerate}
		\item $\Phi$ induces an isomorphism $\mathcal{P}(S, W)\cong \mathrm{End}_\mathcal{C}(T)^{op}$.
		\item For all $i\in Q_0$ we have the following weak $2$-almost split sequence
		\begin{equation}\label{eq - phi 2 almost split sequence}
			0\to T_i\xrightarrow{\Phi((\underline{\beta})_\beta)} \bigoplus_{\substack{\beta\in Q_1 \\ t(\beta) = i}}T_{s(\beta)}^{|\underline{\beta}|} \xrightarrow{\Phi((\partial_{\overline{\beta^*}, \underline{\alpha^*}}W)_{\alpha, \beta})} \bigoplus_{\substack{\alpha\in Q_1 \\ s(\alpha) = i}}T_{t(\alpha)}^{|\overline{\alpha}|}\xrightarrow{\Phi((\overline{\alpha})_\alpha)} T_i\to 0
		\end{equation}
		in $\mathrm{add}(T)$.
	\end{enumerate}
\end{mythm}

\begin{proof}
	It is enough to show that the following are equivalent:
	\begin{enumerate}[label=(\alph*)]
		\item $\Phi$ induces a surjection $\Phi: \widehat{T}(S)\to \mathrm{End}_\mathcal{C}(T)^{op}$ with $\ker \Phi = \mathcal{J}(S, W)$.
		\item For all $i\in Q_0$ the complex
		\begin{equation*}
			\bigoplus_{\substack{\beta\in Q_1 \\ t(\beta) = i}}T_{s(\beta)}^{|\underline{\beta}|} \xrightarrow{\Phi((\partial_{\overline{\beta^*}, \underline{\alpha^*}}W)_{\alpha, \beta})} \bigoplus_{\substack{\alpha\in Q_1 \\ s(\alpha) = i}}T_{t(\alpha)}^{|\overline{\alpha}|}\xrightarrow{\Phi((\overline{\alpha})_\alpha)} T_i\to 0
		\end{equation*}
		is a right $2$-almost split sequence in $\mathrm{add}(T)$, i.e. it induces an exact sequence
		\begin{equation}\label{eq - a->b sequence}
			\mathrm{Hom}_\mathcal{C}(T, \bigoplus_{\substack{\beta\in Q_1 \\ t(\beta) = i}}T_{s(\beta)}^{|\underline{\beta}|}) \xrightarrow{\Phi((\partial_{\overline{\beta^*}, \underline{\alpha^*}}W)_{\alpha, \beta})\circ -} \mathrm{Hom}_\mathcal{C}(T, \bigoplus_{\substack{\alpha\in Q_1 \\ s(\alpha) = i}}T_{t(\alpha)}^{|\overline{\alpha}|}) \xrightarrow{\Phi((\overline{\alpha})_\alpha)\circ -} \mathrm{rad}_\mathcal{C}(T, T_i)\to 0.
		\end{equation}
		\item For all $i\in Q_0$ the complex
		\begin{equation*}
			0\to T_i\xrightarrow{\Phi((\underline{\beta})_\beta)} \bigoplus_{\substack{\beta\in Q_1 \\ t(\beta) = i}}T_{s(\beta)}^{|\underline{\beta}|} \xrightarrow{\Phi((\partial_{\overline{\beta^*}, \underline{\alpha^*}}W)_{\alpha, \beta})} \bigoplus_{\substack{\alpha\in Q_1 \\ s(\alpha) = i}}T_{t(\alpha)}^{|\overline{\alpha}|}
		\end{equation*}
		is a left $2$-almost sequence in $\mathrm{add}(T)$, i.e. it induces an exact sequence
		\begin{equation*}
			\mathrm{Hom}_\mathcal{C}(\bigoplus_{\substack{\alpha\in Q_1 \\ s(\alpha) = i}}T_{t(\alpha)}^{|\overline{\alpha}|}, T)\xrightarrow{-\circ \Phi((\partial_{\overline{\beta^*}, \underline{\alpha^*}}W)_{\alpha, \beta})} \mathrm{Hom}_\mathcal{C}(\bigoplus_{\substack{\beta\in Q_1 \\ t(\beta) = i}}T_{s(\beta)}^{|\underline{\beta}|}, T) \xrightarrow{-\circ \Phi((\underline{\beta^*})_\beta)} \mathrm{rad}_\mathcal{C}(T_i, T)\to 0.
		\end{equation*}
	\end{enumerate}
	We only prove that (a)$\iff$(b) since (a)$\iff$(c) is similar due to (1) in Definition/Proposition~\ref{lemma - complex and derivation}.
	
	(a)$\implies$(b): The sequence in \eqref{eq - a->b sequence} is a complex by Lemma~\ref{lemma - complex and derivation 2}. We only need to show exactness at $\mathrm{Hom}_\mathcal{C}(T, \bigoplus_{\substack{\alpha\in Q_1 \\ s(\alpha) = i}}T_{t(\alpha)}^{|\overline{\alpha}|})$ since exactness at $\mathrm{rad}_\mathcal{C}(T, T_i)$ follows from
	\begin{equation*}
		\Phi(\langle M\rangle) = \Phi(\mathrm{rad}(\widehat{T}(S))) = \mathrm{rad}(\mathrm{End}_\mathcal{C}(T)).
	\end{equation*}
	Assume that $(p_{\overline{\alpha}})_\alpha\in \bigoplus_{\substack{\alpha\in Q_1 \\ s(\alpha) = i}}(\widehat{T}(S)e_{t(\alpha)})^{|\overline{\alpha}|}$, where $p_{\overline{\alpha}} = (p_a)_{a\in \overline{\alpha}}$, such that
	\begin{equation*}
		\sum_{\substack{\alpha\in Q_1 \\ s(\alpha) = i \\ a\in \overline{\alpha}}} p_aa \in \ker \Phi = \mathcal{J}(S, W),
	\end{equation*}
	i.e. $\Phi(p_{\overline{\alpha}})$ is an arbitrary element in $\ker \Phi((\overline{\alpha}))\circ -$. By Lemma~\ref{Lemma - kernel ideal rewritten} there exist $q_b\in \widehat{T}(S)$ such that
	\begin{equation*}
		\sum_{\substack{\alpha\in Q_1 \\ s(\alpha) = i \\ a\in \overline{\alpha}}} p_aa - \sum_{\substack{\beta\in Q_1 \\ t(\beta) = i \\ b\in \underline{\beta}}} q_b\partial_{b^*}W \in \mathcal{J}(S, W)\cdot M.
	\end{equation*}
	Applying $\partial^r_{a^*}$ yields
	\begin{equation*}
		p_a - \sum_{\substack{\beta\in Q_1 \\ t(\beta) = i \\ b\in \underline{\beta}}} q_b\partial_{b^*, a^*}W \in \mathcal{J}(S, W).
	\end{equation*}
	Thus
	\begin{equation*}
		(\Phi(p_{\overline{\alpha}}))_\alpha = \left(\sum_{\substack{\beta\in Q_1 \\ t(\beta) = i \\ b\in \underline{\beta}}} \Phi(\partial_{b^*, \underline{\alpha^*}}W)\circ \Phi(q_b)\right)_\alpha
	\end{equation*}
	shows exactness at $\mathrm{Hom}_\mathcal{C}(T, \bigoplus_{\substack{\alpha\in Q_1 \\ s(\alpha) = i}}T_{t(\alpha)}^{|\overline{\alpha}|})$.
	
	(b)$\implies$(a): First we argue why $\Phi$ is surjective. Exactness at $\mathrm{rad}_\mathcal{C}(T, T_i)$ in \eqref{eq - a->b sequence} implies that
	\begin{equation*}
	 	\mathrm{rad}(\mathrm{End}_\mathcal{C}(T)^{op}) = \Phi(\mathrm{rad}(\widehat{T}(S))).
	\end{equation*}
	By the assumption that $\dim_\K \mathrm{End}_\mathcal{C}(T_i)^{op}/\mathrm{rad}(\mathrm{End}_\mathcal{C}(T_i)^{op}) = \dim_\K D_i$ for all $i\in Q_0$ we also have that $\Phi$ induces an isomorphism
	\begin{equation*}
		\widehat{T}(S) / \mathrm{rad}(\widehat{T}(S))\cong \mathrm{End}_\mathcal{C}(T)^{op} / \mathrm{rad}(\mathrm{End}_\mathcal{C}(T)^{op})
	\end{equation*}
	Therefore we can conclude that $\Phi$ is surjective. Note also $\ker\Phi\subseteq \mathrm{rad}(\widehat{T}(S)) = \widehat{T}(S)_{\ge 1}$.

	It is left to show that $\ker \Phi = \mathcal{J}(S, W)$. Since \ref{eq - a->b sequence} is exact we have that
	\begin{equation*}
		\Phi(\partial_{b^*}(W)) = \sum_{\substack{\alpha\in Q_1 \\ s(\alpha) = i \\ a\in \overline{a}}} \Phi(\partial_{b^*,a^*}(W) a) = \sum_{\substack{\alpha\in Q_1 \\ s(\alpha) = i \\ a\in \overline{a}}} \Phi(a)\circ \Phi(\partial_{b^*,a^*}(W)) = 0,
	\end{equation*}
	for all $b\in \overline{\beta}$. Since $\ker \Phi$ is closed $\mathcal{J}(S, W)\subseteq \ker \Phi$.
	
	Before we show that $\ker \Phi\subseteq \mathcal{J}(S, W)$ we will first prove that for any $p\in \ker \Phi$ there exists $p'\in \mathcal{J}(S, W)$ such that $p - p'\in \ker \Phi\cdot M$. Without loss of generality, we can assume that $p\in \widehat{T}(S)_{\ge 1}e_i$. Since
	\begin{equation*}
		(\Phi(\overline{\alpha}))_\alpha = \Phi(p)\circ (\Phi(\partial^r_{\underline{\alpha^*}} p))_\alpha = 0
	\end{equation*}
	holds, we know that $(\Phi(\partial^r_{\underline{\alpha^*}}p))_\alpha$ factors through $(\Phi(\partial_{\overline{\beta^*}, \underline{\alpha^*}} W))_{\alpha, \beta}$. Thus there exists $q_{\beta} = (q_b)_{b\in \underline{\beta}}$ with $q_b\in \widehat{T}(S)_{\ge 1}e_{s(\beta)}$ such that
	\begin{equation*}
		(\Phi(\partial^r_{\underline{\alpha^*}}p))_\alpha = (\Phi(\partial_{\overline{\beta^*}, \underline{\alpha^*}} W))_{\alpha, \beta}\circ (\Phi(q_{\underline{\beta}}))_\beta.
	\end{equation*}
	Now
	\begin{equation*}
		\partial^r_{a^*} p - \sum_{\substack{\beta\in Q_1 \\ t(\beta) = i \\ b\in \underline{\beta}}} q_b \partial_{b^*, a^*} W\in \ker \Phi.
	\end{equation*}
	Hence
	\begin{equation*}
		p - \sum_{\substack{\beta\in Q_1 \\ b\in \underline{\beta}}} q_b \partial_{b^*} W = \sum_{\substack{\alpha\in Q_1 \\ s(\alpha) = i \\ a\in \overline{a}}} \left(\partial^r_{a^*} p - \sum_{\substack{\beta\in Q_1 \\ t(\beta) = i \\ b\in \underline{\beta}}} q_b \partial_{b^*, a^*} W\right)a\in \ker\Phi\cdot M.
	\end{equation*}
	Thus $p' = \sum_{\substack{\beta\in Q_1 \\ b\in \underline{\beta}}} q_b \partial_{b^*} W$ satisfies the desired condition.
	
	Consequently we have $\ker \Phi \subseteq \mathcal{J}(S, W) + \ker\Phi\cdot M$. Using this fact iteratively we get
	\begin{equation*}
		\ker \Phi\subseteq \mathcal{J}(S, W) + \ker \Phi\cdot M \subseteq \cdots \subseteq \mathcal{J}(S, W) + \ker\Phi\cdot M^l
	\end{equation*}
	for any $l\ge 0$. Now it follows that $\ker\Phi\subseteq \overline{\mathcal{J}(S, W)}$ by \eqref{eq - closure of a set}.
\end{proof}

Recall the conditions in Theorem~\ref{Theorem A}, i.e. $T = \bigoplus_{i=1}^nT_i$ is a cluster tilting object in $\mathcal{C}$ and $T$ satisfies the vanishing condition at all $i\in Q_0$. Moreover $(S, W)$ is a reduced species with potential and $k\in Q_0$ is mutable, i.e. it is not contained in any loops or $2$-cycles in $Q$ and that $T$ satisfies the mutation compatibility condition at $k$ (see Definition~\ref{definition - mut condition at k}). Let $(S', W') = \mu_{k}(S, W)$ over $Q'$ and $T' = \mu_{T_k}(T) = \bigoplus_{i=1}^nT_i'$. Let $\Phi:\mathrm{proj}\widehat{T}(S)\to \mathcal{C}$ be a functor as in Theorem~\ref{Theorem - iso and weak 2}. In this setting, Theorem~\ref{Theorem A} follows from the following theorem, by Theorem~\ref{Theorem - iso and weak 2}.

\begin{mythm}\label{Theorem - theorem implies theorem A}\cite[Theorem 5.6]{BIRS2011Mutation}
	There exists a $\K$-linear functor $\Phi': \mathrm{proj}\widehat{T}(S')\to \mathcal{C}$ such that the sequence
	\begin{equation}\label{eq - phi' 2 almost split}
		0\to T_i'\xrightarrow{\Phi'((\underline{\beta})_\beta)} \bigoplus_{\substack{\beta\in Q'_1 \\ t(\beta) = i}}(T_{s(\beta)}')^{|\underline{\beta}|} \xrightarrow{\Phi'((\partial_{\overline{\beta^*}, \underline{\alpha^*}}W')_{\alpha, \beta})} \bigoplus_{\substack{\alpha\in Q'_1 \\ s(\alpha) = i}}(T_{t(\alpha)}')^{|\overline{\alpha}|}\xrightarrow{\Phi'((\overline{\alpha})_\alpha)} T_i'\to 0
	\end{equation}
	is a weak $2$-almost split sequence in $\mathrm{add}(\mu_{T_k}(T))$ for all $i\in Q_0'$, where $T_j' = \Phi'(D_j)$, and $P_j = \widehat{T}(S')e_j$.
\end{mythm}

Before we define $\Phi'$ we need the following lemma.

\begin{mylemma}\label{Lemma - 5 conditions}\cite[Lemma 5.7]{BIRS2011Mutation}
	If $T$ satisfies the vanishing condition at $i$ for all $i\in Q_0$, then the following holds:
	\begin{enumerate}[label=(\roman*)]
		\item For any $i\in Q_0$, we have a weak $2$-almost split sequence
		\begin{equation}\label{eq - weak 2 almost split at i}
			0\to T_i\xrightarrow{\Phi((\underline{\beta})_\beta)}\bigoplus_{\substack{\beta\in Q_1 \\ t(\beta) = i}}T_{s(\beta)}^{|\underline{\beta}|} \xrightarrow{\Phi((\partial_{(\overline{\beta^*}, \underline{\alpha^*})} (W))_{\alpha, \beta})} \bigoplus_{\substack{\alpha\in Q_1 \\ s(\alpha) = i}}T_{t(\alpha)}^{|\overline{\alpha}|}\xrightarrow{\Phi((\overline{\alpha})_\alpha)} T_i\to 0
		\end{equation}
		in $\mathrm{add}(T)$, which we denote by
		\begin{equation*}
			0\to T_i\xrightarrow{f_{i2}}U_{i1}\xrightarrow{f_{i1}}U_{i0}\xrightarrow{f_{i0}}T_i\to 0.
		\end{equation*}
		\item There exist exchange triangles
		\begin{equation*}
			T_i\xrightarrow{f_{i2}}U_{i1}\xrightarrow{h_i}T_i^*\to T_i[1] \qquad \text{and} \qquad T_i^*\xrightarrow{g_i}U_{i0}\xrightarrow{f_{i0}}T_i\to T_i^*[1]
		\end{equation*}
		in $\mathcal{C}$ such that $f_{i1} = g_ih_i$.
		\item We have the following weak $2$-almost split sequence in $\mathrm{add}(\mu_{T_k}(T))$.
		\begin{equation*}
			0\to T_k^*\xrightarrow{g_k}U_{k0}\xrightarrow{f_{k2}f_{k0}}U_{k1}\xrightarrow{h_k}T_k^*\to 0
		\end{equation*}
		\item For any $k\not= i\in Q_0$, we have that $T_k\not\in (\mathrm{add}(U_{i1}) \cap \mathrm{add}(U_{i0}))$ and the following sequences are exact.
		\begin{equation*}
			\begin{aligned}
				\mathrm{Hom}_{\mathcal{C}}(T_k^*, U_{i1}) &\xrightarrow{f_{i1}\circ -}\mathrm{Hom}_{\mathcal{C}}(T_k^*, U_{i0})\xrightarrow{f_{i0}\circ -}\mathrm{Hom}{_\mathcal{C}}(T_k^*, T_i) \\
				\mathrm{Hom}_{\mathcal{C}}(U_{i0}, T_k^*) &\xrightarrow{f_{i1}\circ -}\mathrm{Hom}_{\mathcal{C}}(U_{i1}, T_k^*)\xrightarrow{f_{i2}\circ -}\mathrm{Hom}{_\mathcal{C}}(T_i, T_k^*)
			\end{aligned}
		\end{equation*}
		\item The sequences
		\begin{equation*}
			\begin{aligned}
				\mathrm{Hom}_\mathcal{C}(\mu_{T_k}(T), U_{i1}) \xrightarrow{f_{i1}\circ-}\mathrm{Hom}_\mathcal{C}(\mu_{T_k}(T), U_{i0}) \xrightarrow{f_{i0}\circ-}\mathrm{Hom}_\mathcal{C}(\mu_{T_k}(T), T_i) \\
				\mathrm{Hom}_\mathcal{C}(U_{i0}, \mu_{T_k}(T)) \xrightarrow{-\circ f_{i1}}\mathrm{Hom}_\mathcal{C}(U_{i1}, \mu_{T_k}(T)) \xrightarrow{-\circ f_{i2}}\mathrm{Hom}_\mathcal{C}(T_i, \mu_{T_k}(T))
			\end{aligned}
		\end{equation*}
		are exact.
	\end{enumerate}
\end{mylemma}

\begin{proof}
	\begin{enumerate}[label=(\roman*)]
		\item Follows from Theorem~\ref{Theorem - iso and weak 2}.
		\item Recall that $W$ is assumed to be reduced, which in turn implies that $f_{i1}\in \mathrm{rad}_\mathcal{C}(U_{i1}, U_{i0})$. Since $Q$ has no loops and $T$ satisfies the vanishing condition at $i$, the claim follows from \cite[Theorem 4.9]{BIRS2011Mutation}. 
		\item Follows from \cite[Theorem 4.8]{BIRS2011Mutation} and also from \cite[Theorem 3.10]{IyamaYoshino2008Mutation}.
		\item We only prove that the former sequence is exact since the proof for the latter sequence to be exact is similar. Since no $2$-cycles starts at the vertex $k$ we have that $T_k\not\in (\mathrm{add}(U_{i1}) \cap \mathrm{add}(U_{i0}))$. 

		By (i) and (ii) we have the following commutative diagram.
		\begin{equation*}
			\begin{tikzcd}
				T_i \arrow[r, "f_{i2}"] & U_{i1} \arrow[rr, "f_{i1}"] \arrow[dr, "h_i"] & & U_{i0} \arrow[r, "f_{i0}"] & T_i \\
				& & T_i^* \arrow[ur, "g_i"]
			\end{tikzcd}
		\end{equation*}
		Applying $\mathrm{Hom}_{\mathcal{C}}(T_k^*, -)$ yields the following diagram.
		\begin{equation}\label{eq - exactness of hom sequence lemma}
			\begin{tikzcd}
				\mathrm{Hom}_{\mathcal{C}}(T_k^*, U_{i1}) \arrow[rr, "f_{i1}\circ -"] \arrow[dr, "h_i\circ -"] & & \mathrm{Hom}_{\mathcal{C}}(T_k^*, U_{i0}) \arrow[r, "f_{i0}\circ -"] & \mathrm{Hom}_{\mathcal{C}}(T_k^*, T_i) \\
				& \mathrm{Hom}_{\mathcal{C}}(T_k^*, T_i^*) \arrow[ur, "g_i\circ -"] \arrow[dr] \\
				& & \mathrm{Ext}^1_{\mathcal{C}}(T_k^*, T_i)
			\end{tikzcd}
		\end{equation}
		Here $\im(h_i\circ -)$ is equal to the kernel of
		\begin{equation*}
			\mathrm{Hom}_{\mathcal{C}}(T_k^*, T_i^*)\to \mathrm{Ext}^1_{\mathcal{C}}(T_k^*, T_i)
		\end{equation*}
		by (ii), i.e. follows from that $T_i\xrightarrow{f_{i2}}U_{i1}\xrightarrow{h_i}T_i^*\to T_i[1]$ is a triangle. We need to show that the top row of \eqref{eq - exactness of hom sequence lemma} is exact. Using (ii) again, that $T_i^*\xrightarrow{g_i}U_{i0}\xrightarrow{f_{i0}}T_i\to T_i^*[1]$ is an exchange triangle we immediately have that $\im(g_i\circ-) = \ker(f_{i0}\circ-)$. It is enough to show that $\mathrm{Ext}^1_{\mathcal{C}}(T_k^*, T_i) = 0$. Indeed, if $\mathrm{Ext}^1_{\mathcal{C}}(T_k^*, T_i)=0$ then $h_i\circ -$ is an epimorphism. Now $\im(f_{i1}\circ-) = \im(g_i\circ h_i\circ-) = \im(g_i\circ -) = \ker(f_{i0}\circ -)$, hence the top row of \eqref{eq - exactness of hom sequence lemma} is exact. To show that $\mathrm{Ext}^1_{\mathcal{C}}(T_k^*, T_i)=0$ we consider the following exchange triangle
		\begin{equation*}
			T_k^* \to U_{k0}\to T_k\to T_k^*[1]
		\end{equation*}
		and apply $\mathrm{Hom}_\mathcal{C}(T_i, -)$ to get the exact sequence
		\begin{equation*}
			\mathrm{Hom}_\mathcal{C}(T_i, T_k^*)\to \mathrm{Hom}_\mathcal{C}(T_i, U_{k0}) \xrightarrow{f_{k0}\circ -} \mathrm{Hom}_\mathcal{C}(T_i, T_k) \to \mathrm{Ext}^1_\mathcal{C}(T_i, T_k^*)\to \mathrm{Ext}^1_\mathcal{C}(T_i, U_{k0}).
		\end{equation*}
		Now $\mathrm{Ext}^1_\mathcal{C}(T_i, U_{k0}) = 0$ since $T_i, U_{k0}\in \mathrm{add}(T)$ and $\mathrm{Hom}_\mathcal{C}(T_i, f_{k0})$ is an epimorphism since $f_{k0}$ is a right $\mathrm{add}(T/T_k)$-approximation. Thus $\mathrm{Ext}^1_\mathcal{C}(T_i, T_k^*) = 0$. Using that $\mathcal{C}$ is $2$-Calabi--Yau yields $\mathrm{Ext}^1_\mathcal{C}(T_k^*, T_i) = 0$.
		\item As in the proof of (iv), to prove that
		\begin{equation}\label{eq - exact sequence form triangle (iv)}
			\mathrm{Hom}_\mathcal{C}(T/T_k, U_{i1}) \xrightarrow{f_{i1}\circ -}\mathrm{Hom}_\mathcal{C}(T/T_k, U_{i0}) \xrightarrow{f_{i0}\circ -}\mathrm{Hom}_\mathcal{C}(T/T_k, T_i)
		\end{equation}
		is exact, it is enough to prove that
		\begin{equation*}
			\mathrm{Hom}_\mathcal{C}(T/T_k, U_{i1}) \xrightarrow{h_i\circ -}\mathrm{Hom}_\mathcal{C}(T/T_k, T_i^*)
		\end{equation*}
		is an epimorphism. This directly follows from $h_i$ being a left $\mathrm{add}(T/T_k)$-approximation. By combining \eqref{eq - exact sequence form triangle (iv)} and (iv) we get that
		\begin{equation}
			\mathrm{Hom}_\mathcal{C}(\mu_{T_k}(T), U_{i1}) \xrightarrow{f_{i1}}\mathrm{Hom}_\mathcal{C}(\mu_{T_k}(T), U_{i0}) \xrightarrow{f_{i0}}\mathrm{Hom}_\mathcal{C}(\mu_{T_k}(T), T_i)
		\end{equation}
		is exact.

		The argument for showing that
		\begin{equation*}
			\mathrm{Hom}_\mathcal{C}(U_{i0}, \mu_{T_k}(T)) \xrightarrow{f_{i1}}\mathrm{Hom}_\mathcal{C}(U_{i1}, \mu_{T_k}(T)) \xrightarrow{f_{i2}}\mathrm{Hom}_\mathcal{C}(T_i, \mu_{T_k}(T))
		\end{equation*}
		is exact is dual to the above argument.
	\end{enumerate}
\end{proof}

We now proceed to define $\Phi': \mathrm{proj}\widehat{T}(S')\to \mathcal{C}$ using Lemma~\ref{lemma - bijection sets algebra and pairs}. We do this in several steps.
\begin{enumerate}
	\item We define $\Phi'$ on objects by setting $\Phi'(P_i) = T_i$ for $i\not= k$ and $\Phi'(P_k) = T_k^*$.
	\item We let $\Phi'$ coincide with $\Phi$ on $D_i\subseteq \mathrm{Hom}_{\mathrm{proj}\widehat{T}(S')}(P_i, P_i)^{op}$ whenever $i\neq k$. Similarly on $M_\alpha\subseteq \mathrm{Hom}_{\mathrm{proj}\widehat{T}(S')}(P_j, P_i)^{op}$ where $\alpha: i\to j\in Q_1$, with $i\neq k$ and $j\neq k$ we also let $\Phi'$ coincide with $\Phi$.
\end{enumerate}

To determine $\Phi'$ it remains to define $\Phi'$ on $D_k\subseteq \mathrm{Hom}_{\mathrm{proj}\widehat{T}(S')}(P_k, P_k)^{op}$ and, for arrows $\alpha:k\to i$ and $\beta: j\to k$, on $M_{\alpha^*}\subseteq \mathrm{Hom}_{\mathrm{proj}\widehat{T}(S')}(P_k, P_i)^{op}$, $M_{\beta^*} \subseteq \mathrm{Hom}_{\mathrm{proj}\widehat{T}(S')}(P_j, P_k)^{op}$ and $M_{[\alpha\beta]} \subseteq \mathrm{Hom}_{\mathrm{proj}\widehat{T}(S')}(P_i, P_j)^{op}$.

\begin{enumerate}
	\item[(3)] On $M_{[\alpha\beta]} = M_\alpha \otimes_{D_k}M_\beta$ we set $\Phi'([x\otimes y]) = \Phi(y)\circ \Phi(x)$. This ensures that for $i\neq k$ not adjacent to $k$, \eqref{eq - phi 2 almost split sequence} and \eqref{eq - phi' 2 almost split} are the same sequence.
\end{enumerate}

For the final step we need the following lemma.

\begin{mylemma}\label{lemma - lambda tilde and lambda hat}
	For any $\lambda\in D_k$ there exist unique elements
	\begin{equation*}
		\tilde{\lambda}\in \bigoplus_{\substack{\alpha\in Q_1 \\ s(\alpha)=k}} \Mat_{\dim_{D_{t(\alpha)}} M_\alpha}\left(D_{t(\alpha)}\right), \quad \hat{\lambda}\in \bigoplus_{\substack{\beta\in Q_1 \\ t(\beta)=k}} \Mat_{\dim_{D_{s(\beta)}} M_\beta} \left(D_{s(\beta)}\right)
	\end{equation*}
	such that
	\begin{equation*}
		\lambda(\underline{\alpha^*})_{\alpha\in Q_1, s(\alpha) = k} = (\underline{\alpha^*})_{\alpha\in Q_1, s(\alpha) = k} \tilde{\lambda} \qquad (\overline{\beta^*})_{\beta\in Q_1, t(\beta) = k} \lambda = \hat{\lambda} (\overline{\beta^*})_{\beta\in Q_1, t(\beta) = k}.
	\end{equation*}
\end{mylemma}

\begin{proof}
	We only show it for $\tilde{\lambda}$ since the argument for $\hat{\lambda}$ is similar. From the definition of $\underline{\alpha^*}$ we have that
	\begin{equation*}
		M_{\alpha^*} = \bigoplus_{a^*\in \underline{\alpha^*}} a^* D_{t(\alpha)}.
	\end{equation*}
	Since $\lambda\cdot -: M_{\alpha^*} \to M_{\alpha^*}$ is a $D_{t(\alpha)}$-module morphism, the existence and uniqueness of $\tilde{\lambda}$ follows from the fact that $\underline{\alpha^*}$ is a $D_{t(\alpha)}$-basis of $M_{\alpha^*}$.
\end{proof}

\begin{enumerate}
	\item[(4)] Using Lemma~\ref{Lemma - 5 conditions} and Lemma~\ref{lemma - lambda tilde and lambda hat}, we define
	\begin{equation}\label{eq - phi 1}
		\Phi'((\underline{\alpha^*})_{\alpha\in Q_1, s(\alpha) = k}) = g_k\in \mathrm{Hom}_\mathcal{C}(T_k^*, \bigoplus_{\substack{\alpha\in Q_1 \\ s(\alpha)=k}} T_{t(\alpha)}^{|\overline{\alpha}|})^{op}
	\end{equation}
	and we extend linearly with respect to $\bigoplus_{\substack{\alpha\in Q_1 \\ s(\alpha)=k}} D_{t(\alpha)}$, in particular
	\begin{equation*}
		\Phi'((\underline{\alpha^*})_{\alpha\in Q_1, s(\alpha) = k} \tilde{\lambda}) = \Phi'((\underline{\alpha^*})_{\alpha\in Q_1, s(\alpha) = k})\circ \Phi(\tilde{\lambda}).
	\end{equation*}
	Similarly we define
	\begin{equation}\label{eq - phi 2}
		\Phi'((\overline{\beta^*})_{\beta\in Q_1, t(\beta) = k}) = h_k\in \mathrm{Hom}_\mathcal{C}(\bigoplus_{\substack{\beta\in Q_1 \\ t(\beta)=k}} T_{s(\beta)}^{|\underline{\beta}|}, T_k^*)^{op}
	\end{equation}
	and we extend linearly with respect to $\bigoplus_{\substack{\beta\in Q_1 \\ t(\beta)=k}} D_{s(\beta)}$.
\end{enumerate}

What remains is to define $\Phi':D_k\to \mathrm{End}_\mathcal{C}(T_k^*)^{op}$ so that
\begin{equation*}
	\Phi': M_{\alpha^*} \to \mathrm{Hom}_\mathcal{C}(T_k^*, T_i)^{op} \quad \Phi': M_{\beta^*} \to \mathrm{Hom}_\mathcal{C}(T_j, T_k^*)^{op},
\end{equation*}
defined via \eqref{eq - phi 1} and \eqref{eq - phi 2} respectively, are $D$-bimodule morphisms. This will determine $\Phi$ uniquely by Lemma~\ref{lemma - bijection sets algebra and pairs}. Note that these are automatically $D_i$-module respectively $D_j$-module morphisms. Moreover, $\mathrm{End}_\mathcal{C}(T_k^*)^{op}/\rad(\mathrm{End}_\mathcal{C}(T_k^*)^{op})$ is isomorphic to $D_k$ by Proposition~\ref{proposition - end mod rad equivalence}. However, to obtain a compatible morphism $D_k\to \mathrm{End}_\mathcal{C}(T_k^*)^{op}$ we need the following assumption on $\Phi$.

\begin{mydef}\label{definition - mut condition at k}
	We say that $T$, or $\Phi$, satisfies the mutation compatibility condition at $k$ if we can define a $\K$-algebra morphism $q_-: D_k\to \mathrm{End}_\mathcal{C}(T_k^*)^{op}$ such that
	\begin{equation}\label{eq - mut cond diagram}
		\begin{tikzcd}
			T_k \arrow[r, "f_{k2}"] \arrow[ddd, "\Phi(\lambda)"] & U_{k1} \arrow[ddd, "\Phi(\hat{\lambda})"] \arrow[rr, "f_{k1}"] \arrow[rd, "h_k"] & & U_{k0} \arrow[r, "f_{k0}"] \arrow[ddd, "\Phi(\tilde{\lambda})"] & T_k \arrow[lld, dotted] \arrow[ddd, "\Phi(\lambda)"] \\
			& & T_k^* \arrow[ru, "g_k"] \arrow[llu, dotted] \arrow[ddd, "q_\lambda"{pos=0.3}] \\ \\
			T_k \arrow[r, "f_{k2}"] & U_{k1} \arrow[rr, "f_{k1}\qquad"] \arrow[rd, "h_k"] & & U_{k0} \arrow[r, "f_{k0}"] & T_k \arrow[lld, dotted] \\
			& & T_k^* \arrow[ru, "g_k"] \arrow[llu, dotted]
		\end{tikzcd}
	\end{equation}
	commutes for all $\lambda\in D_k$, where the dotted arrows are the connecting morphisms in the exchange triangles. Here $\Phi(\hat{\lambda})$ is interpreted as applying $\Phi$ on each matrix element of $\hat{\lambda}$. Thus $\Phi(\hat{\lambda})\in \mathrm{End}_\mathcal{C}(U_{k1})$. Similarly we have that $\Phi(\tilde{\lambda})\in \mathrm{End}_\mathcal{C}(U_{k0})$.
\end{mydef}

\begin{myrem}\label{remark - mutation comp involutative}
	Note that if $T$ satisfies the mutation compatibility condition at $k$ then
	\begin{equation}
		\begin{tikzcd}
			T_k^* \arrow[r, "g_k"] \arrow[ddd, "q_\lambda"] & U_{k0} \arrow[ddd, "\Phi(\tilde{\lambda})"] \arrow[rr, "f_{k2}\circ f_{k0}"] \arrow[rd, "f_{k0}"] & & U_{k1} \arrow[r, "h_k"] \arrow[ddd, "\Phi(\hat{\lambda})"] & T_k \arrow[lld, dotted] \arrow[ddd, "q_\lambda"] \\
			& & T_k \arrow[ru, "f_{k2}"] \arrow[llu, dotted] \arrow[ddd, "\Phi(\lambda)"{pos=0.3}] \\ \\
			T_k^* \arrow[r, "g_k"] & U_{k0} \arrow[rr, "f_{k2}\circ f_{k0}\qquad\quad"] \arrow[rd, "f_{k0}"] & & U_{k1} \arrow[r, "h_k"] & T_k \arrow[lld, dotted] \\
			& & T_k \arrow[ru, "f_{k2}"] \arrow[llu, dotted]
		\end{tikzcd}
	\end{equation}
	automatically commutes for all $\lambda\in D_k$.
\end{myrem}

Note that the left most and right most squares in \eqref{eq - mut cond diagram} commute by Lemma~\ref{lemma - lambda tilde and lambda hat}.

\begin{myrem}\label{remark - (Q, W) satisfy mut condition}
	If $D_k = \K$, then we can take $q_-$ to be the canonical inclusion $\K \to \mathrm{End}_\mathcal{C}(T_k^*)$ and so the mutation compatibility condition at $k$ is automatically satisfied in this case. In particular this holds if $(S, W) = (Q, W)$ is a quiver with potential.
\end{myrem}

\begin{mylemma}
	There is a $\K$-linear functor $\Phi': \mathrm{proj}\widehat{T}(S')\to \mathcal{C}$ agreeing with the definition in the steps (1)-(4) above if and only if $T$ satisfies the mutation compatibility condition at $k$.
\end{mylemma}

\begin{proof}
	If $T$ satisfies the mutation compatibility condition at $k$, then we define $\Phi'(\lambda) = q_\lambda$. Since $\Phi(\tilde{\lambda})\circ g_k = g_k \circ q_\lambda$ we get $\Phi'(\tilde{\lambda}) \circ g_k = g_k\circ \Phi'(\lambda)$. Now $\Phi'((\underline{\alpha^*})_{\alpha\in Q_1, s(\alpha) = k}) = g_k$ yields
	\begin{equation}\label{eq - mut cond comm diagram computation}
		\begin{aligned}
			\Phi'(\lambda (\underline{a^*})_{\alpha\in Q_1, s(\alpha) = k}) &= \Phi'((\underline{a^*})_{\alpha\in Q_1, s(\alpha) = k} \tilde{\lambda}) = \Phi'(\tilde{\lambda})\circ \Phi'((\underline{a^*})_{\alpha\in Q_1, s(\alpha) = k}) = \\
			&= \Phi'((\underline{a^*})_{\alpha\in Q_1, s(\alpha) = k})\Phi'(\lambda),
		\end{aligned}	
	\end{equation}
	which implies that $\Phi': M_{\alpha^*} \to \mathrm{Hom}_\mathcal{C}(T_k^*, T_i)^{op}$ is a $D$-bimodule morphism. Similarly $\Phi': M_{\beta^*}\to \Hom_\mathcal{C}(T_j, T_k^*)^{op}$ is a $D$-bimodule morphism. Together with the above steps (1)-(4) this defines $\Phi'$ on $D$ as an $\K$-algebra morphism and on $M'$ as a compatible $D$-bimodule morphism, which extends uniquely to a well-defined $\K$-linear functor $\Phi': \mathrm{proj}\widehat{T}(S')\to \mathcal{C}$.

	On the other hand, assuming $\Phi': \mathrm{proj}\widehat{T}(S')\to \mathcal{C}$ is a $\K$-linear functor agreeing with the definition above, we can define $q_\lambda = \Phi'(\lambda)$ and the commutativity of the diagram \eqref{eq - mut cond diagram} follows from the same computation as in \eqref{eq - mut cond comm diagram computation}. Thus $T$ satisfies the mutation compatibility condition at $k$.
\end{proof}

We now continue under the assumption that $\Phi'$ is well-defined and given as in steps (1)-(4). To ease up the notation we omit writing $\Phi$ and $\Phi'$. Theorem~\ref{Theorem - theorem implies theorem A} follows from the following lemma, by calculating the sequence \eqref{eq - phi' 2 almost split} in four distinct cases.

\begin{mylemma}\label{Lemma - Main}\cite[Lemma 5.9, Lemma 5.10, Lemma 5.11]{BIRS2011Mutation}
	Let $i\in Q_0$. The following holds:
	\begin{enumerate}
		\item We have the following weak $2$-almost split sequences in $\mathrm{add}(\mu_{T_k}(T))$
		\begin{equation}\label{eq - Lemma mutation 1}
			0\to T^*_k \xrightarrow{(\underline{\beta^*})_\beta}\bigoplus_{\beta\in Q_1, t(\beta^*) = k}T_{s(\beta^*)}^{|\underline{\beta^*}|}\xrightarrow{([\overline{\beta}\underline{\alpha}])_{(\alpha, \beta)}} \bigoplus_{\alpha\in Q_1, s(\alpha^*) = k} T_{t(\alpha^*)}^{|\overline{\alpha^*}|} \xrightarrow{(\overline{\alpha^*})_\alpha} T^*_k\to 0.
		\end{equation}
		\item If $i\neq k$ and $i$ is not adjacent to $k$, then
		\begin{equation}
			0\to T_i\xrightarrow{(\underline{\beta})_\beta}\bigoplus_{\substack{\beta\in Q_1 \\ t(\beta) = i}}T_{s(\beta)}^{|\underline{\beta}|} \xrightarrow{(\partial_{(\overline{\beta^*}, \underline{\alpha^*})} (W'))_{\alpha, \beta}} \bigoplus_{\substack{\alpha\in Q_1 \\ s(\alpha) = i}}T_{t(\alpha)}^{|\overline{\alpha}|}\xrightarrow{(\overline{\alpha})_\alpha} T_i\to 0
		\end{equation}
		is a weak $2$-almost split sequence.
		\item If $i\neq k$ and there is an arrow $\beta: i\to k$ in $Q_1$, then
		\begin{equation}\label{eq - Lemma mutation 2}
			0\to T_i \xrightarrow{f_3} \begin{matrix}
				(T^*_k)^{|\underline{\beta^*}|} \\ \\ \oplus \left(\displaystyle\bigoplus_{\substack{\delta\in Q_1 \\ t(\delta)=i}} T_{s(\delta)}^{|\underline{\delta}|}\right)
			\end{matrix}\xrightarrow{f_2}\begin{matrix}
				\left(\displaystyle\bigoplus_{\substack{\alpha \in Q_1 \\ s(\alpha) = k}} T_{t(\alpha)}^{\dim_{D_{t(\alpha)}}M_\alpha\otimes_{D_k}M_\beta} \right) \\ \oplus \left(\displaystyle\bigoplus_{\substack{\gamma\in Q_1 \\ s(\gamma) = i \\ t(\gamma)\neq k}} T_{t(\gamma)}^{|\overline{\gamma}|}\right)
			\end{matrix} \xrightarrow{f_1} T_i \to 0
		\end{equation}
		is a weak $2$-almost split sequence, where
		\begin{equation*}
			\begin{aligned}
				f_1 &= \begin{bmatrix}
					(\overline{[\alpha\beta]})_\alpha & (\overline{\gamma})_\gamma
				\end{bmatrix}, \\
				f_2 &= \begin{bmatrix}
					((\underline{\alpha^*})^{\oplus |\overline{\beta}|})_\alpha & (\partial_{(\overline{\delta^*}, \underline{[\alpha\beta]}^*)} [W])_{\alpha, \delta} \\
					0 & (\partial_{(\overline{\delta^*}, \underline{\gamma^*})} [W])_{(\gamma, \delta)}
				\end{bmatrix}, \\
				f_3 &= \begin{bmatrix}
					\underline{\beta^*} \\
					(\underline{\delta})_\delta
				\end{bmatrix}.
			\end{aligned}
		\end{equation*}
		\item If $i\neq k$ and there is an arrow $\alpha: k\to i$ in $Q_1$, then
		\begin{equation*}
			0\to T_i \xrightarrow{f_3} \begin{matrix}
				\left(\displaystyle\bigoplus_{\substack{\beta\in Q_1 \\ t(\beta) = k}} T_{s(\beta)}^{\dim_{D_{s(\beta)}}M_\alpha\otimes_{D_k}M_\beta} \right) \\ \oplus \left(\displaystyle\bigoplus_{\substack{\delta\in Q_1 \\ t(\delta)=i \\ s(\delta)\neq k}} T_{s(\delta)}^{|\underline{\delta}|}\right)
			\end{matrix}\xrightarrow{f_2}\begin{matrix}
				(T_k^*)^{|\overline{\alpha^*}|} \\ \\ \oplus \left(\displaystyle\bigoplus_{\substack{\gamma\in Q_1 \\ s(\gamma) = i}} T_{t(\gamma)}^{|\overline{\gamma}|}\right)
			\end{matrix} \xrightarrow{f_1} T_i \to 0
		\end{equation*}
		is a weak $2$-almost split sequence, where
		\begin{equation*}
			\begin{aligned}
				f_1 &= \begin{bmatrix}
					\overline{\alpha^*} & (\overline{\gamma})_\gamma
				\end{bmatrix}, \\
				f_2 &= \begin{bmatrix}
					((\overline{\beta^*})^{\oplus |\underline{\alpha}|})_\beta & 0 \\
					(\partial_{(\overline{[\alpha\beta]^*}, \underline{\gamma^*})} [W])_{\gamma, \beta} & (\partial_{(\overline{\delta^*}, \underline{\gamma^*})} [W])_{(\gamma, \delta)}
				\end{bmatrix}, \\
				f_3 &= \begin{bmatrix}
					(\underline{[\alpha\beta]})_\beta \\
					(\underline{\delta})_\delta
				\end{bmatrix}.
			\end{aligned}
		\end{equation*}
	\end{enumerate}
\end{mylemma}

\begin{proof}
	(1): Note that $\Phi'$ is defined so that the complex \eqref{eq - Lemma mutation 1} is precisely
	\begin{equation*}
		T_k^*\xrightarrow{g_k}U_{k0}\xrightarrow{f_{k2}\circ f_{k0}}U_{k1}\xrightarrow{h_k}T_k^*,
	\end{equation*}
	which is a weak $2$-almost split sequence by (iii) in Lemma \ref{Lemma - 5 conditions}.

	(2): Note that $k$ is not adjacent to $i$ implies that we have that $T_k\not\in \mathrm{add}(U_{i0}\oplus U_{i1})$. Therefore, since \eqref{eq - weak 2 almost split at i} is a weak $2$-almost split sequence in $\mathrm{add}(T)$, the sequence
	\begin{equation}
		0\to T_i\xrightarrow{(\underline{\beta})_\beta}\bigoplus_{\substack{\beta\in Q_1 \\ t(\beta) = i}}T_{s(\beta)}^{|\underline{\beta}|} \xrightarrow{(\partial_{(\overline{\beta^*}, \underline{\alpha^*})} ([W]))_{\alpha, \beta}} \bigoplus_{\substack{\alpha\in Q_1 \\ s(\alpha) = i}}T_{t(\alpha)}^{|\overline{\alpha}|}\xrightarrow{(\overline{\alpha})_\alpha} T_i\to 0
	\end{equation}
	is a weak $2$-almost split sequence in $\mathrm{add}(\mu_{T_k}(T))$. Now we are done by noting that
	\begin{equation*}
		(\partial_{(\overline{\beta^*}, \underline{\alpha^*})} (W'))_{\alpha, \beta} = (\partial_{(\overline{\beta^*}, \underline{\alpha^*})} ([W]))_{\alpha, \beta}
	\end{equation*}
	if $s(\alpha), t(\alpha), s(\beta), t(\beta) \neq k$ where $\alpha, \beta\in Q_1$.
	
	(3): First we prove that \eqref{eq - Lemma mutation 2} is a complex. Let us compute $f_1\circ f_2$.
	\begin{equation*}
		\begin{aligned}
			f_1\circ f_2 &= \begin{bmatrix}
				(\overline{[\alpha\beta]})_\alpha & (\overline{\gamma})_\gamma
			\end{bmatrix}\begin{bmatrix}
				((\underline{\alpha^*})^{\oplus |\overline{\beta}|})_\alpha & (\partial_{(\overline{\delta^*}, \underline{[\alpha\beta]}^*)} [W])_{\alpha, \delta} \\
				0 & (\partial_{(\overline{\delta^*}, \underline{\gamma^*})} [W])_{(\gamma, \delta)}
			\end{bmatrix} = \\
			&=\begin{bmatrix}
				\partial_{\overline{\beta}} \Delta & (\partial_{\overline{\delta^*}} [W])_{\delta}
			\end{bmatrix} = \\ 
			&=\begin{bmatrix}
				\partial_{\overline{\beta}} W' & (\partial_{\overline{\delta^*}} W')_{\delta}
			\end{bmatrix} = 0
		\end{aligned}
	\end{equation*}
	Let us compute $f_2\circ f_3$.
	\begin{equation*}
		\begin{aligned}
			f_2\circ f_3 &= \begin{bmatrix}
				((\underline{\alpha^*})^{\oplus |\overline{\beta}|})_\alpha & (\partial_{(\overline{\delta^*}, \underline{[\alpha\beta]}^*)} [W])_{\alpha, \delta} \\
				0 & (\partial_{(\overline{\delta^*}, \underline{\gamma^*})} [W])_{(\gamma, \delta)}
			\end{bmatrix}\begin{bmatrix}
				\underline{\beta^*} \\
				(\underline{\delta})_\delta
			\end{bmatrix} = \\
			&= \begin{bmatrix}
				(\partial_{\underline{[\alpha\beta]^*}} \Delta)_\alpha + (\partial_{\underline{[\alpha\beta]^*}} [W])_\alpha \\
				(\partial_{\underline{\gamma^*}} [W])_\gamma
			\end{bmatrix} = \\
			&= \begin{bmatrix}
				(\partial_{\underline{[\alpha\beta]^*}} W')_\alpha \\
				(\partial_{\underline{\gamma^*}} W')_\gamma
			\end{bmatrix} = 0
		\end{aligned}
	\end{equation*}
	Thus \eqref{eq - Lemma mutation 2} is a complex.

	Next we will show that \eqref{eq - Lemma mutation 2} is indeed a weak $2$-almost split sequence. Let us first rewrite \eqref{eq - Lemma mutation 2} using exchange triangles to ease up notation. Recall that
	\begin{equation*}
		0\to T_i\xrightarrow{f_{i2}}U_{i1}\xrightarrow{f_{i1}}U_{i0}\xrightarrow{f_{i0}}T_i\to 0
	\end{equation*}
	is a weak $2$-almost split sequence in $\mathrm{add}(T)$. Write $U_{i0} = T_k^{|\overline{\beta}|}\oplus U_{i0}''$ with $T_k\not\in \mathrm{add}(U_{i0}'')$ and
	\begin{equation*}
		\begin{aligned}
			f_{i0} = \begin{bmatrix}
				f_{i0}' & f_{i0}''
			\end{bmatrix}: U_{i0}=T_k^{|\overline{\beta}|}\oplus U_{i0}''\to T_i \\
			f_{i1} = \begin{bmatrix}
				f_{i1}' \\
				f_{i1}''
			\end{bmatrix}: U_{i1} \to U_{i0} = T_k^{|\overline{\beta}|}\oplus U_{i0}''.
		\end{aligned}
	\end{equation*}
	We can now write \eqref{eq - Lemma mutation 2} as
	\begin{equation*}
		T_i \xrightarrow{\begin{bmatrix}
			s \\
			f_{i2}
		\end{bmatrix}} (T_k^*)^{|\overline{\beta}|}\oplus U_{i1} \xrightarrow{\begin{bmatrix}
			g_{k}^{\oplus |\overline{\beta}|} & t \\
			0 & f_{i1}''
		\end{bmatrix}} U_{k0}^{|\overline{\beta}|}\oplus U_{i0}''\xrightarrow{\begin{bmatrix}
			f_{i0}' \circ f_{k0}^{\oplus |\overline{\beta}|} & f_{i0}''
		\end{bmatrix}} T_i
	\end{equation*}
	where $f_{k0}^{\oplus |\overline{\beta}|}\circ t = f_{i1}': U_{i1}\to T_k^{|\overline{\beta}|}$.

	To show that \eqref{eq - Lemma mutation 2} is a weak $2$-almost split sequence means that we have to show that the sequences
	\begin{equation}\label{eq - main lemma show 2 weak 1}
		\mathrm{Hom}_\mathcal{C}(\mu_{T_k}(T), (T_k^*)^{|\overline{\beta}|}\oplus U_{i1}) \to \mathrm{Hom}_\mathcal{C}(\mu_{T_k}(T), U_{k0}^{|\overline{\beta}|}\oplus U_{i0}'')\to \mathrm{rad}_\mathcal{C}(\mu_{T_k}(T), T_i)\to 0
	\end{equation}
	and
	\begin{equation}\label{eq - main lemma show 2 weak 2}
		\mathrm{Hom}_\mathcal{C}(U_{k0}^{|\overline{\beta}|}\oplus U_{i0}'', \mu_{T_k}(T))\to \mathrm{Hom}_\mathcal{C}((T_k^*)^{|\overline{\beta}|}\oplus U_{i1}, \mu_{T_k}(T))\to \mathrm{rad}_\mathcal{C}(T_i, \mu_{T_k}(T))\to 0
	\end{equation}
	are exact.

	We start with the former sequence \eqref{eq - main lemma show 2 weak 1}. Note that $\mathrm{rad}_\mathcal{C}(\mu_{T_k}(T), T_i) = \mathrm{rad}_\mathcal{C}(T/T_k, T_i)\oplus \mathrm{rad}_\mathcal{C}(T_k^*, T_i)$. First we let $p\in \mathrm{rad}_\mathcal{C}(T/T_k, T_i)$. Using that $f_{i0}$ is right almost split in $\add(T)$ yields a morphism $\begin{bmatrix}
		p_1 \\ 
		p_2
	\end{bmatrix}\in \mathrm{Hom}_\mathcal{C}(T/T_k, T_k^{|\overline{\beta}|}\oplus U_{i0}'')$ such that
	\begin{equation*}
		f_{i0}'\circ p_1 + f_{i0}''\circ p_2 = p.
	\end{equation*}
	Now since $f_{k0}$ is right almost split in $\add(T)$ there is a morphism $q\in \mathrm{Hom}_\mathcal{C}(U_{k0}^{|\overline{\beta}|}, T_k^{|\overline{\beta}|})$ such that $f_{k0}^{\oplus |\overline{\beta}|}\circ q = p_2$. Therefore
	\begin{equation*}
		\begin{bmatrix}
			f_{i0}'\circ f_{k0}^{\oplus |\overline{\beta}|} & f_{i0}''
		\end{bmatrix}\circ \begin{bmatrix}
			q & 0 \\
			0 & p_2
		\end{bmatrix} = p.
	\end{equation*}
	In other words, $p$ factors through $\begin{bmatrix}
		f_{i0}' \circ f_{k0}^{\oplus |\overline{\beta}|} & f_{i0}''
	\end{bmatrix}$. The above argument is summarised in the following diagram.
	\begin{equation*}
		\begin{tikzcd}
			& & & T/T_k \arrow[dd, "p"] \arrow[llldd, "q", bend right=40, swap] \arrow[lldd, "p_1", bend right=40, swap] \arrow[dl, "p_2", swap] \\
			& & U_{i0}'' \arrow[dr, "f_{i0}''"] \\
			U_{k0}^{|\overline{\beta}|} \arrow[r, "f_{k0}^{\oplus |\overline{\beta}|}"] & T_k^{|\overline{\beta}|} \arrow[rr, "f_{i0}'"] & & T_i
		\end{tikzcd}
	\end{equation*}

	On the other hand, if $p\in \mathrm{rad}_\mathcal{C}(T_k^*, T_i)$ then using that $g_k$ is a left $\mathrm{add}(T/T_k)$-approximation there is a morphism $p'\in \mathrm{Hom}_\mathcal{C}(U_{k0}, T_i)$ such that $p'\circ g_k = p$. Since $\beta: i\to k$ there is no arrow $k\to i$, and so $T_i\not\in \add(U_{k0})$. Thus $p'\in \mathrm{rad}_\mathcal{C}(U_{k0}, T_i)$ and $p'$ factors through $\begin{bmatrix}
		f_{i0}' \circ f_{k0}^{\oplus |\overline{\beta}|} & f_{i0}''
	\end{bmatrix}$ using the above argument. This show exactness at $\mathrm{rad}_\mathcal{C}(\mu_{T_k}(T), T_i)$.

	Now we will show exactness at $\mathrm{Hom}_\mathcal{C}(\mu_{T_k}(T), U_{k0}^{|\overline{\beta}|}\oplus U_{i0}'')$. Let $\begin{bmatrix}
		p_1 \\
		p_2
	\end{bmatrix}$ be a morphism such that $\begin{bmatrix}
		f_{i0}' \circ f_{k0}^{\oplus |\overline{\beta}|} & f_{i0}''
	\end{bmatrix}\circ \begin{bmatrix}
		p_1 \\
		p_2
	\end{bmatrix} = 0$. By (v) in Lemma~\ref{Lemma - 5 conditions} there exists a morphism $q\in \mathrm{Hom}_\mathcal{C}(\mu_{T_k}(T), U_{i1})$ such that
	\begin{equation*}
		\begin{bmatrix}
			f_{i1}' \\
			f_{i1}''
		\end{bmatrix}\circ q = \begin{bmatrix}
			f_{k0}^{\oplus |\overline{\beta}|} & 0 \\
			0 & 1
		\end{bmatrix}\circ \begin{bmatrix}
			p_1 \\
			p_2
		\end{bmatrix}.
	\end{equation*}
	Now note that $f_{k0}^{\oplus |\overline{\beta}|}\circ p_1 = f_{i1}'\circ q = f_{k0}^{\oplus |\overline{\beta}|}\circ t\circ q$. Therefore $f_{k0}^{\oplus |\overline{\beta}|}\circ (p_1 - t\circ q) = 0$. By (ii) in Lemma~\ref{Lemma - 5 conditions} the following sequence
	\begin{equation*}
		\mathrm{Hom}_\mathcal{C}(\mu_{T_k}(T), T_k^*) \xrightarrow{g_k\circ -} \mathrm{Hom}_\mathcal{C}(\mu_{T_k}(T), U_{k0}) \xrightarrow{f_{k0}\circ -} \mathrm{Hom}_\mathcal{C}(\mu_{T_k}(T), T_k)
	\end{equation*}
	is exact, which together with $f_{k0}^{\oplus |\overline{\beta}|}\circ (p_1 - t\circ q) = 0$ implies that there is a morphism $r\in \mathrm{Hom}_\mathcal{C}(\mu_{T_k}(T), (T_k^*)^{|\overline{\beta}|})$ such that $g_k^{\oplus |\overline{\beta}|}\circ r = p_1 - t\circ q$. Rearranging the terms we get $p_1 = g_k^{\oplus |\overline{\beta}|}\circ r + t\circ q$. Hence
	\begin{equation*}
		\begin{bmatrix}
			p_1 \\
			p_2
		\end{bmatrix} = \begin{bmatrix}
			g_k^{\oplus |\overline{\beta}|} & t \\
			0 & f_{i1}''
		\end{bmatrix}\circ \begin{bmatrix}
			r \\
			q
		\end{bmatrix}.
	\end{equation*}
	This is summarised in the following diagram.
	\begin{equation*}
		\begin{tikzcd}
			(T_k^*)^{|\overline{\beta}|} \arrow[r, "g_k^{\oplus |\overline{\beta}|}"] & U_{k0}^{|\overline{\beta}|} \arrow[rr, "f_{k0}^{\oplus |\overline{\beta}|}"] & & T_k^{|\overline{\beta}|} \arrow[dr, "f_{i0}'"] \\
			& & & & T_i \\
			U_{i1} \arrow[ruu, "t"] \arrow[rrruu, "f_{i1}'"] \arrow[rrr, "f_{i1}''"] & & & U_{i1}'' \arrow[ru, "f_{i0}''"] \\
			& & \mu_{T_k}(T) \arrow[ru, "p_2"] \arrow[uuul, "p_1", crossing over, bend left] \arrow[llu, "q"] \arrow[lluuu, "r", out=200, in=230, looseness = 1.5]
		\end{tikzcd}
	\end{equation*}
	
	Before we prove that \eqref{eq - main lemma show 2 weak 2} we need the following lemma.

	\begin{mylemma}\label{lemma - s is bijective}
		The morphism
		\begin{equation}\label{eq - beta map}
			-\circ s = -\circ \Phi'(\underline{\beta^*}) : \mathrm{Hom}_\mathcal{C}((T^*_k)^{|\overline{\beta}|}, T_k^*) / \mathrm{rad}_\mathcal{C}((T^*_k)^{|\overline{\beta}|}, T_k^*) \to \mathrm{rad}_\mathcal{C}(T_i, T_k^*) / \mathrm{rad}^2_{\mathrm{add}(\mu_{T_k}(T))}(T_i, T_k^*)
		\end{equation}
		is bijective.
	\end{mylemma}

	\begin{proof}
		Since $h_k$ is minimal right almost split in $\mathrm{add}_{\mu_k}(T)$ by (iii) in Lemma~\ref{Lemma - 5 conditions}, we have that
		\begin{equation}\label{eq - s map bij equation}
			\mathrm{rad}_\mathcal{C}(T_i, T_k^*) / \mathrm{rad}^2_{\mathrm{add}(\mu_{T_k}(T))}(T_i, T_k^*)
		\end{equation}
		has the $D_i$-basis $\{\Phi'(b^*) \mid b^*\in \overline{\beta^*}\}$, by using similar arguments as in the proof of (2)$\implies$(1) in Theorem~\ref{Theorem - iso and weak 2}. Thus the composition
		\begin{equation*}
			M_{\beta^*}\xrightarrow{\Phi'}\rad_\mathcal{C}(T_i, T_k^*)^{op} \to \rad_\mathcal{C}(T_i, T_k^*)^{op}/\rad^2_{\mathrm{add}(\mu_{T_k}(T))}(T_i, T_k^*)^{op}
		\end{equation*}
		is an isomorphism of $D_i$-$D_k$-bimodules. Hence, \eqref{eq - s map bij equation} has the $D_k$-basis $\{\Phi'(b^*) \mid b^*\in \underline{\beta^*}\}$. This proves that \eqref{eq - beta map} is bijective.
	\end{proof}

	Now we are ready to prove that \eqref{eq - main lemma show 2 weak 2} is exact. We begin by showing exactness at $\mathrm{rad}_\mathcal{C}(T_i, \mu_{T_k}(T)) = \rad_\mathcal{C}(T_i, T/T_k)\oplus \rad_\mathcal{C}(T_i, T_k^*)$. First let $p\in \mathrm{rad}_\mathcal{C}(T_i, T/T_k)$. Since $f_{i2}$ is left almost split in $\mathrm{add}(T)$ and thus $p$ factors via $U_{i1}$. Now let $p\in \mathrm{rad}_\mathcal{C}(T_i, T_k^*)$. By Lemma~\ref{lemma - s is bijective} there is morphism $p_1\in \mathrm{Hom}_\mathcal{C}((T^*_k)^{|\overline{\beta}|}, T_k^*)$ such that $p - p_1\circ s\in \mathrm{rad}^2_{\add(\mu_{T_k}(T))}(T_i, T_k^*)$. Using that $h_k$ is right almost split we have a morphism $q\in \mathrm{rad}_\mathcal{C}(T_i, U_{k0})$ such that
	\begin{equation*}
		\begin{tikzcd}[column sep=2cm]
			T_i \arrow[r, "p - p_1\circ s"] \arrow[dr, "q"] & T_k^* \\
			& U_{k0} \arrow[u, "h_k", swap]
		\end{tikzcd}
	\end{equation*}
	commutes. The morphism $q$ lies in the radical since, if $q$ was a split monomorphism, $h_k$ being minimal right almost split, would contradict $h_k\circ q\in \mathrm{rad}^2_{\add(\mu_{T_k}(T))}(T_i, T_k^*)$. Now using that $f_{i2}$ is left almost split in $\mathrm{add}(T)$ we get a morphism $r\in \mathrm{Hom}_\mathcal{C}(U_{i1}, U_{k0})$ such that $q = r\circ f_{i2}$. Combining everything we see that
	\begin{equation*}
		p = \begin{bmatrix}
			p_1 & h_k\circ r
		\end{bmatrix}\circ \begin{bmatrix}
			s \\
			f_{i2}
		\end{bmatrix},
	\end{equation*}
	which shows exactness at $\mathrm{rad}_\mathcal{C}(T_i, \mu_{T_k}(T))$. The argument is summarised in the following diagram.
	\begin{equation*}
		\begin{tikzcd}
			& & (T_k^*)^{|\overline{\beta}|} \arrow[dddl, "p_1", out=-50, in=40] \\
			T_i \arrow[rr, "f_{i2}"] \arrow[rd, "q"] \arrow[ddr, "p", bend right] \arrow[rru, "s"] & & U_{i1} \arrow[dl, "r"] \\
			& U_{k0} \arrow[d, "h_k"] \\
			& T_k^*
		\end{tikzcd}
	\end{equation*}

	It is left to show exactness at $\mathrm{Hom}_\mathcal{C}((T_k^*)^{|\overline{\beta}|}\oplus U_{i1}, \mu_{T_k}(T))$. Let $p = \begin{bmatrix}
		p_1 & p_2
	\end{bmatrix}\in \mathrm{Hom}_\mathcal{C}((T_k^*)^{|\overline{\beta}|}\oplus U_{i1}, \mu_{T_k}(T))$ such that $\begin{bmatrix}
		p_1 & p_2
	\end{bmatrix}\circ \begin{bmatrix}
		s \\
		f_{i2}
	\end{bmatrix} = 0$. Let us introduce the notation
	\begin{equation*}
		p_1 = \begin{bmatrix}
			p_1' \\
			p_1''
		\end{bmatrix}: (T_k^*)^{|\overline{\beta}|} \to T/T_k\oplus T_k^*
	\end{equation*}
	and
	\begin{equation*}
		p_2 = \begin{bmatrix}
			p_2' \\
			p_2''
		\end{bmatrix}: U_{i1} \to T/T_k\oplus T_k^*.
	\end{equation*}
	Since $g_k$ is left almost split in $\mathrm{add}(\mu_{T_k}(T))$ by (ii) in Lemma~\ref{Lemma - 5 conditions}, it suffices to show that $p_1\in \mathrm{rad}_\mathcal{C}((T_k^*)^{|\overline{\beta}|}, \mu_{T_k}(T))$ to obtain a morphism $q\in \mathrm{Hom}_\mathcal{C}(U_{k0}^{|\overline{\beta}|}, T/T_k)$ such that $p_1 = q\circ g_k^{\oplus |\overline{\beta}|}$. We immediately have that $p_1'\in \mathrm{rad}_\mathcal{C}((T_k^*)^{|\overline{\beta}|}, T/T_k)$ since $T_k^*\not\in \mathrm{add}(T/T_k)$. The assumption on $p$ gives us $p_1''\circ s = - p_2''\circ f_{i2}$. We have that $p_2''\in \mathrm{rad}_\mathcal{C}(U_{i1}, T_k^*)$ since $T_k^*\not\in \mathrm{add}(U_{i1})$. Therefore $p_1''\in\mathrm{rad}_\mathcal{C}((T_k^*)^{|\overline{\beta}|}, T_k^*)$ by Lemma~\ref{lemma - s is bijective}.

	Now note that
	\begin{equation*}
		\begin{aligned}
			(p_2 - q\circ t)\circ f_{i2} &= p_2\circ f_{i2} - q\circ t\circ f_{i2} = \\
			&= p_2\circ f_{i2} + q\circ g_k^{\oplus |\overline{\beta}|}\circ s = \\
			&= p_2\circ f_{i2} + p_1\circ s = 0.
		\end{aligned}
	\end{equation*}
	Thus by exactness of the latter sequence in (v) in Lemma~\ref{Lemma - 5 conditions} there exists $\begin{bmatrix}
		q_1 & q_2
	\end{bmatrix}\in \mathrm{Hom}_\mathcal{C}(T_k^{|\overline{\beta}|}\oplus U_{i0}'', \mu_{T_k}(T))$ such that
	\begin{equation*}
		p_2 - q\circ t = \begin{bmatrix}
			q_1 & q_2
		\end{bmatrix}\circ \begin{bmatrix}
			f_{i1}' \\
			f_{i1}''
		\end{bmatrix}.
	\end{equation*}
	Recalling that $f_{k0}^{\oplus |\overline{\beta}|}\circ t = f_{i1}'$ we can rewrite the above expression as
	\begin{equation*}
		p_2 = (q + q_1\circ f_{k0}^{\oplus |\overline{\beta}|})\circ t + q_2\circ f_{i1}''.
	\end{equation*}
	Hence
	\begin{equation*}
		\begin{bmatrix}
			p_1 & p_2
		\end{bmatrix} = \begin{bmatrix}
			q + q_1\circ f_{k0}^{\oplus |\overline{\beta}|} & q_2
		\end{bmatrix}\circ \begin{bmatrix}
			g_k^{\oplus |\overline{\beta}|} & t \\
			0 & f_{i1}''
		\end{bmatrix},
	\end{equation*}
	and therefore we have shown exactness at $\mathrm{Hom}_\mathcal{C}((T_k^*)^{|\overline{\beta}|}\oplus U_{i1}, \mu_{T_k}(T))$. We summarise the above argument in the following diagram.
	
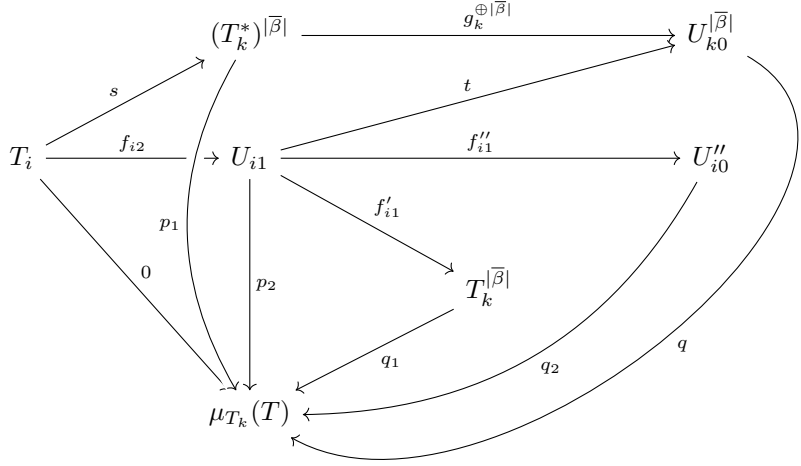
\begin{equation*}
		\begin{tikzcd}[column sep=2cm, row sep=1cm]
			& (T_k^*)^{|\overline{\beta}|} \arrow[rr, "g_k^{\oplus |\overline{\beta}|}"] & & U_{k0}^{|\overline{\beta}|} \\
			T_i \arrow[ru, "s"] \arrow[r, "f_{i2}"] \arrow[ddr, "0"] & U_{i1} \arrow[dd, "p_2"] \arrow[rru, "t"] \arrow[rr, "f_{i1}''"] \arrow[dr, "f_{i1}'"] & & U_{i0}'' \arrow[ddll, "q_2", bend left] \\
			& & T_k^{|\overline{\beta}|} \arrow[dl, "q_1"] \\
			& \mu_{T_k}(T) \arrow[from=uuu, "p_1", bend right, crossing over, swap] \arrow[from=uuurr, "q", crossing over, in=-30, out=-30]
		\end{tikzcd}
	\end{equation*}
	This completes the proof of (2).

	(4): This proof is dual to the proof of (3).
\end{proof}

\section{Self-Injective Species with Potential and Nakayama Automorphism}\label{Section - nakayama automorphism}
In this section we introduce self-injective species with potentials. For a given self-injective species with potential its Jacobian algebra will be finite-dimensional and self-injective and therefore admits a Nakayama automorphism. Under certain restrictions we give a description of the Nakayama automorphism after mutating along orbits of the Nakayama permutation which is explicitly given as the second part of the main theorem of this article (Theorem~\ref{Theorem B}). The theorem is divided into two parts. The first part shows that if some vertex is mutable, then every vertex along its Nakayama permutation orbit is mutable. The second part computes the Nakayama automorphism.

\subsection{Self-Injective Species with Potential}
In this subsection we generalise various results of \cite{HerschendOsamu2011quiverwithpotential}. The following definition is rewritten version of \cite[Definition 3.6]{HerschendOsamu2011quiverwithpotential} for species with potentials.

\begin{mydef}
	Let $(S, W)$ be a species with potential. 
	\begin{enumerate}
		\item We say that $(S, W)$ is finite dimensional if the Jacobian algebra $\mathcal{P}(S, W)$ is a finite dimensional algebra.
		\item We say that $(S, W)$ is self-injective if the Jacobian algebra $\mathcal{P}(S, W)$ is a finite dimensional self-injective algebra.
	\end{enumerate}
\end{mydef}

Let $(S, W)$ be a species with potential and set $P_i = \mathcal{P}(S, W)e_i$. Consider the following complexes
\begin{align}\label{eq - Exact except at one place sequences 1}
	P_i \xrightarrow{(\underline{\beta})_\beta} \bigoplus_{\substack{\beta\in Q_1 \\ t(\beta) = i}} P_{s(\beta)}^{|\underline{\beta}|} \xrightarrow{(\partial_{(\overline{\beta^*}, \underline{\alpha^*})} W)_{\alpha, \beta}} \bigoplus_{\substack{\alpha\in Q_1 \\ s(\alpha) = i}} P_{t(\alpha)}^{|\overline{\alpha}|} \xrightarrow{(\overline{\alpha})_\alpha} P_i \to &S_i\to 0 \\ \label{eq - Exact except at one place sequences 2}
	P_i^\dagger \xrightarrow{(\overline{\alpha})_\alpha} \bigoplus_{\substack{\alpha\in Q_1 \\ s(\alpha) = i}} (P_{t(\alpha)}^\dagger)^{|\overline{\alpha}|} \xrightarrow{(\partial_{(\overline{\beta^*}, \underline{\alpha^*})} W)_{\beta, \alpha}} \bigoplus_{\substack{\beta\in Q_1 \\ t(\beta) = i}} (P_{s(\beta)}^\dagger)^{|\underline{\beta}|} \xrightarrow{(\underline{\beta})_\beta} P_i^\dagger \to &S_i\to 0
\end{align}
of left $\mathcal{P}(S, W)$-modules and right $\mathcal{P}(S, W)$-modules respectively, where $(-)^\dagger = \mathrm{Hom}_{\mathcal{P}(S, W)}(-, \mathcal{P}(S, W))$.

It is straightforward to check that \eqref{eq - Exact except at one place sequences 1} is exact at $S_i$ and $P_i$. To show exactness at $\bigoplus_{\substack{\alpha\in Q_1 \\ s(\alpha) = i}} P_{t(\alpha)}^{|\overline{\alpha}|}$ we do the following. Let $x\in \widehat{T}(S)$ and assume that $x\in \ker((\overline{\alpha})_\alpha)$. This means that $x(\overline{\alpha})_\alpha\in \mathcal{J}(S, W)$, and by Lemma~\ref{Lemma - kernel ideal rewritten} we have
\begin{equation*}
	x(\overline{\alpha})_\alpha = \sum_{\substack{\beta\in Q_1 \\ b\in \underline{\beta}}}x_b\partial_{b^*}(W) + y
\end{equation*}
where $x_b\in \widehat{T}(S)e_{s(\beta)}$ and $y\in \mathcal{J}(S, W)M$. Now applying $\partial^r_{\underline{\alpha^*}}$ and using Lemma~\ref{lemma - complex and derivation 2} yields
\begin{equation*}
	x = \sum_{\substack{\alpha, \beta\in Q_1 \\ (a, b)\in \overline{\alpha}\times \underline{\beta}}}x_b\partial_{b^*, a^*}(W) + \partial^r_{\underline{\alpha^*}}(y).
\end{equation*}
By noting that $\partial^r_{\underline{\alpha^*}}(y)\in \mathcal{J}(S, W)$ we can conclude that $\sum_{\substack{\beta\in Q_1 \\ t(\beta) = i \\ b\in \underline{\beta}}} x_b$ maps to $x$ via $(\partial_{(\overline{\beta^*}, \underline{\alpha^*})} W)_{\alpha, \beta}$, which proves that $\eqref{eq - Exact except at one place sequences 1}$ is exact at $\bigoplus_{\substack{\alpha\in Q_1 \\ s(\alpha) = i}} P_{t(\alpha)}^{|\overline{\alpha}|}$. In other words, it is exact except possibly at $\bigoplus_{\substack{\beta\in Q_1 \\ t(\beta) = i}} P_{s(\beta)}^{|\underline{\beta}|}$. Similarly, \eqref{eq - Exact except at one place sequences 2} is exact except possibly at $\bigoplus_{\substack{\alpha\in Q_1 \\ s(\alpha) = i}} (P_{t(\alpha)}^\dagger)^{|\overline{\alpha}|}$.

The following proposition, which is a generalisation of \cite[Theorem 3.7]{HerschendOsamu2011quiverwithpotential}, shows that exactness of the sequences \ref{eq - Exact except at one place sequences 1} and \ref{eq - Exact except at one place sequences 2} correlates to $(S, W)$ being self-injective.

\begin{myprop}
	Let $(S, W)$ be a species with potential. The following are equivalent:
	\begin{enumerate}
		\item $(S, W)$ is self-injective.
		\item $(S, W)$ is finite dimensional and \ref{eq - Exact except at one place sequences 1} is exact for all $i\in Q_0$.
		\item $(S, W)$ is finite dimensional and \ref{eq - Exact except at one place sequences 2} is exact for all $i\in Q_0$.
	\end{enumerate}
\end{myprop}

\begin{proof}
	We only show $(1)$ is equivalent to $(3)$, since the proof of $(1)$ being equivalent to $(2)$ is similar. The algebra $\mathcal{P}(S, W)$ being self-injective is equivalent to $\mathrm{Ext}^1_{\mathcal{P}(S, W)}(S_i, \mathcal{P}(S, W))=0$ for all $i\in Q_0$ since all objects in $\mathrm{mod}(\mathcal{P}(S, W))$ have finite length. This in turn is equivalent to that applying $\mathrm{Hom}_{\mathcal{P}(S, W)}(-, \mathcal{P}(S, W))$ to
	\begin{equation*}
		\bigoplus_{\substack{\beta\in Q_1 \\ t(\beta) = i}} P_{s(\beta)}^{|\underline{\beta}|} \xrightarrow{(\partial_{(\underline{\beta^*}, \overline{\alpha^*})})_{\beta, \alpha}} \bigoplus_{\substack{\alpha\in Q_1 \\ s(\alpha) = i}} P_{t(\alpha)}^{|\overline{\alpha}|} \xrightarrow{(\underline{\alpha})_\alpha} P_i
	\end{equation*}
	yields an exact sequence. In other words, $(1)$ is equivalent to \ref{eq - Exact except at one place sequences 2} being exact at $\bigoplus_{\substack{\alpha\in Q_1 \\ s(\alpha) = i}} (P_{t(\alpha)}^\dagger)^{|\overline{\alpha}|}$, which is equivalent to (3).
\end{proof}

\subsection{Mutation of self-injective species with potential and Nakayama Automorphism}
Let $(S, W)$ be a reduced species with potential, where $S = (D, M)$ has the quiver $Q$. We assume that $Q$ has no multiple arrows. This can always be achieved by changing the bimodule $M$. The Jacobian algebra $\mathcal{P}(S, W)$ is Frobenius, i.e. there exists an automorphism $\gamma: \mathcal{P}(S, W)\xrightarrow{\sim} \mathcal{P}(S, W)$, called the Nakayama automorphism of $\mathcal{P}(S, W)$, such that $\mathcal{P}(S, W)\cong D\mathcal{P}(S, W)_\gamma$ as $\mathcal{P}(S, W)$-bimodules. Here $D\mathcal{P}(S, W)_\gamma$ is the $\K$-dual of $\mathcal{P}(S, W)$ twisted by $\gamma$ on the right. Recall that every self-injective species with potential satisfies the vanishing condition at every vertex, which will be used implicitly. 

For a given $k\in Q_0$ we let $(k) = \{k, \sigma(k), \sigma^2(k), \dots\}$, where $\sigma$ is the Nakayama permutation (see Proposition~\ref{Prop - selfinjectie T=T[2]}). We want to mutate along the Nakayama orbit $(k)$ and show that we obtain another self-injective species with potential as well as describe its Nakayama automorphism. In each step along the way we want to apply Theorem~\ref{Theorem A}, which requires that the potential is reduced. Moreover, to obtain a self-injective species with potential we need to assure that there are no arrows between the vertices in $(k)$. This motivates the following definition.

\begin{mydef}\label{definition - sparse orbit}
	We say that an orbit $(k)$ of the Nakayama permutation is sparse if there are no arrows between the vertices $k, \sigma(k), \dots, \sigma^n(k) = k$, and for each $0\le i\le n-2$ the potential $\mu_{\sigma^i(k)}\circ \mu_{\sigma^{i-1}(k)} \circ \mu_k(S, W)$ is reduced. In this case we define
	\begin{equation*}
		\mu_{(k)}(S, W) = \mu_{\sigma^{n-1}(k)}\circ \mu_{\sigma^{n-2}(k)} \circ \mu_k(S, W).
	\end{equation*}
\end{mydef}

Note that if $k$ is fixed by the Nakayama permutation, then $(k)$ is automatically sparse.

\begin{mythm}\label{theorem - wedderburn-maltsev}(Wedderburn-Mal'tsev Theorem)\cite{mal1942representation,wedderburn1908hypercomplex}
	Let $\Lambda$ be a finite dimensional $\K$-algebra with radical $N$, and let the quotient $\K$-algebra $\Lambda/N$ be a separable $\K$-algebra. Then $\Lambda$ can be decomposed into a direct sum
	\begin{equation*}
		\Lambda = N\oplus D
	\end{equation*}
	for some semisimple subalgebra $D$, and if there exists another decomposition $\Lambda = N\oplus D'$, where $D'$ is a semisimple subalgebra, then there exists an inner automorphism of $\Lambda$ which maps $D$ onto $D'$.
\end{mythm}

By Theorem~\ref{theorem - wedderburn-maltsev} we can assume that $\gamma_0: D\xrightarrow{\sim}D$ since the Nakayama automorphism is unique up to some inner automorphism. Furthermore $(\gamma_0)|_{D_i}: D_i \to D_{\sigma(i)}$. We also assume the following.

\begin{enumerate}
	\item[(A)] $\gamma_1: M\to M$
	\item[(B)] $\gamma(W) = W$
\end{enumerate}

Note that there is an arrow $\alpha: i\to j$ if and only if there is an arrow $\sigma(i)\to \sigma(j)$, which we denote by $\sigma(\alpha)$. Condition (A) implies that
\begin{equation*}
	\gamma_1(\underline{\alpha}) = \underline{\sigma(\alpha)}A
\end{equation*}
where $A$ is some invertible matrix with coefficients in $\bigoplus_{t(\alpha) = k}D_{s(\alpha)}$.

The above conditions are satisfied when we study tensor products of tensor algebras of species with certain conditions in Section~\ref{Section - Jasso-muro plus examples}.

Following \cite{Pasquali2020selfinjective} we can construct a Nakayama automorphism of $\mathrm{End}_\mathcal{C}(T)$. Let $\varphi: T\xrightarrow{\sim}T[2]$ be an isomorphism. Then by \cite[Proposition 4.3]{Pasquali2020selfinjective}
\begin{equation*}
	\psi(f) = \varphi^{-1}\circ f[2]\circ \varphi, \quad f\in \mathrm{End}_\mathcal{C}(T),
\end{equation*}
is a Nakayama automorphism. In particular, $T_i[2]\cong T_{\sigma(i)}$ so we may assume $\varphi = \bigoplus_{i\in Q_0} \varphi_i$ where $\varphi_i: T_{\sigma(i)}\to T_i[2]$. We may further assume that $\psi$ corresponds to $\gamma$ via $\Phi$, i.e. $\psi\circ \Phi = \Phi \circ \gamma$. Note that $\psi(\Phi(\lambda)) = \Phi(\gamma_0(\lambda))$ for $\lambda\in D$.

\begin{mylemma}\label{lemma - tildevarphi_k}
	Let $T$ be as in Theorem~\ref{Theorem - iso and weak 2}. The exchange triangle
	\begin{equation*}
		T_k \xrightarrow{f_{k2}} U_{k1}\to T_k^* \to T_k[1]
	\end{equation*}
	for $T_k$ gives an exchange triangle
	\begin{equation}\label{eq - psi exchange triangle 1}
		T_{\sigma(k)} \xrightarrow{\psi(f_{k2})} U_{\sigma(k)1} \to T^*_{\sigma(k)} \to T_{\sigma(k)}[1]
	\end{equation}
	for $T_{\sigma(k)}$. Moreover, there is an isomorphism $\tilde{\varphi}_k: T^*_{\sigma(k)} \to T_k^*[2]$ such that
	\begin{equation*}
		\begin{tikzcd}
			T_{\sigma(k)} \ar[r, "\psi(f_{k2})"] \ar[d, "\varphi_k"] & U_{\sigma(k)1} \ar[r] \ar[d, "\varphi_{U_{k1}}"] & T_{\sigma(k)}^* \ar[r] \ar[d, "\tilde{\varphi}_k"] & T_k[1]  \ar[d, "\varphi_k{[1]}"] \\
			T_k[2] \ar[r, "f_{k2}{[2]}"] & U_{k1}[2] \ar[r, "h_k{[2]}"]  & T_k^*[2] \ar[r] & T_k[3]
		\end{tikzcd}
	\end{equation*}
	commutes, where $\varphi_{U_{k1}}$ is defined as applying $\varphi$ on each summand of $U_{\sigma(k)1}$.
\end{mylemma}

\begin{proof}
	By the definition of $\psi$ we have that
	\begin{equation*}
		\begin{tikzcd}
			T_{\sigma(k)} \ar[r, "\psi(f_{k2})"] \ar[d, "\varphi_k"] & U_{\sigma(k)1} \ar[d, "\varphi_{U_{k1}}"] \\
			T_k[2] \ar[r, "f_{k2}{[2]}"] & U_{k1}[2]
		\end{tikzcd}
	\end{equation*}
	commutes. Since $f_{k2}$ is a minimal left $\add(T/T_k)$-approximation we have that $f_{k2}[2]$ is a minimal left $\add((T/T_k)[2])$-approximation. Now by noting that $(T/T_k)[2] \cong T/T_{\sigma(k)}$, the morphism $\psi(f_{k2})$ is a minimal left $\add(T/T_{\sigma(k)})$-approximation. Thus \eqref{eq - psi exchange triangle 1} is an exchange triangle.

	Now we can complete
	\begin{equation*}
		\begin{tikzcd}
			T_{\sigma(k)} \ar[r, "\psi(f_{k2})"] \ar[d, "\varphi_k"] & U_{\sigma(k)1} \ar[r] \ar[d, "\varphi_{U_{k1}}"] & T_{\sigma(k)}^* \ar[r] & T_k[1]  \ar[d, "\varphi_k{[1]}"] \\
			T_k[2] \ar[r, "f_{k2}{[2]}"] & U_{k1}[2] \ar[r, "h_k{[2]}"]  & T_k^*[2] \ar[r] & T_k[3]
		\end{tikzcd}
	\end{equation*}
	into an isomorphism of triangles with some isomorphism $\tilde{\varphi}_k: T_{\sigma(k)}^* \to T_k^*[2]$ using the properties of triangles.
\end{proof}

\begin{mylemma}\label{lemma - nakayama permutation mutable}
	Let $(S, W)$ be a self-injective species with potential. Let $T$ and $\Phi$ be as in Theorem~\ref{Theorem - iso and weak 2}. If $T$ satisfies the mutation compatibility condition at $k$, then it also satisfies the mutation compatibility condition at $\sigma(k)$.
\end{mylemma}

\begin{proof}
	Applying $[2]$ on \eqref{eq - mut cond diagram} yields the commutative diagram
	\begin{equation}\label{eq - mut cond diagram shifted 2}
		\begin{tikzcd}
			T_k[2] \arrow[r, "f_{k2}{[}2{]}"] \arrow[ddd, "\Phi(\lambda){[}2{]}"] & U_{k1}[2] \arrow[ddd, "\Phi(\hat{\lambda}){[}2{]}"] \arrow[rr, "f_{k1}{[}2{]}"] \arrow[rd, "h_k{[2]}"] & & U_{k0}[2] \arrow[r, "f_{k0}{[}2{]}"] \arrow[ddd, "\Phi(\tilde{\lambda}){[}2{]}"] & T_k[2] \arrow[lld, dotted] \arrow[ddd, "\Phi(\lambda){[}2{]}"] \\
			& & T_k^*[2] \arrow[ru, "g_k{[}2{]}"] \arrow[llu, dotted] \arrow[ddd, "q_\lambda{[}2{]}"{pos=0.3}] \\ \\
			T_k[2] \arrow[r, "f_{k2}{[}2{]}"] & U_{k1}[2] \arrow[rr, "f_{k1}{[}2{]}\qquad\quad"] \arrow[rd, "h_k{[}2{]}"] & & U_{k0}[2] \arrow[r, "f_{k0}{[}2{]}"] & T_k[2] \arrow[lld, dotted] \\
			& & T_k^*[2] \arrow[ru, "g_k{[}2{]}"] \arrow[llu, dotted]
		\end{tikzcd}
	\end{equation}

	Now using the definition of $\varphi$ and $\psi$ we have that the diagram
	\begin{equation}\label{eq - mut cond diagram varphi}
		\begin{tikzcd}
			T_{\sigma(k)} \arrow[r, "\psi(f_{k2})"] \arrow[ddd, "\psi(\Phi(\lambda))"] & U_{\sigma(k)1} \arrow[ddd, "\psi(\Phi(\hat{\lambda}))"] \arrow[rr, "\psi(f_{k1})"] & & U_{\sigma(k)0} \arrow[r, "\psi(f_{k0})"] \arrow[ddd, "\psi(\Phi(\tilde{\lambda}))"] & T_{\sigma(k)} \arrow[ddd, "\psi(\Phi(\lambda))"] \\
			& & \\ \\
			T_{\sigma(k)} \arrow[r, "\psi(f_{k2})"] & U_{\sigma(k)1} \arrow[rr, "\psi(f_{k1})"] & & U_{\sigma(k)0} \arrow[r, "\psi(f_{k0})"] & T_{\sigma(k)}
		\end{tikzcd}
	\end{equation}
	commutes. Now we will rewrite \eqref{eq - mut cond diagram varphi} such that it matches the mutation compatibility condition at $\sigma(k)$. First we note that
	\begin{equation*}
		\psi(f_{k2}) = \psi(\Phi((\underline{\alpha})_\alpha)) = \Phi(\gamma((\underline{\alpha})_\alpha)) = \Phi((\underline{\sigma(\alpha)})_\alpha A) = \Phi(A)\circ \Phi((\underline{\sigma(\alpha)})_\alpha) = \Phi(A)\circ f_{\sigma(k)2}
	\end{equation*}
	for some invertible matrix $A$ with coefficients in $\bigoplus_{\substack{\alpha\in Q_1 \\ t(\alpha) = k}} D_{s(\alpha)}$. Hence by \eqref{eq - mut cond diagram varphi} the diagram
	\begin{equation*}
		\begin{tikzcd}
			T_{\sigma(k)} \arrow[r, "f_{\sigma(k)2}"] \arrow[d, "\Phi(\gamma(\lambda))"] & U_{\sigma(k)1} \arrow[d, "\Phi(A\gamma(\hat{\lambda})A^{-1})"] \\
			T_{\sigma(k)} \arrow[r, "f_{\sigma(k)2}"] & U_{\sigma(k)1}
		\end{tikzcd}
	\end{equation*}
	commutes. Similarly, there is an invertible matrix $B$ with coefficients in $\bigoplus_{\substack{\beta\in Q_1 \\ s(\beta) = k}}D_{t(\beta)}$ such that the following diagram commutes.
	\begin{equation*}
		\begin{tikzcd}[column sep = 1.5cm]
			T_{\sigma(k)} \arrow[r, "f_{\sigma(k)2}"] \arrow[ddd, "\Phi(\gamma(\lambda))"] & U_{\sigma(k)1} \arrow[ddd, "\Phi(A\gamma(\hat{\lambda})A^{-1}))"] \arrow[rr, "\Phi(B)\circ f_{\sigma(k)1}\circ \Phi(A)"] & & U_{\sigma(k)0} \arrow[r, "f_{\sigma(k)0}"] \arrow[ddd, "\Phi(B^{-1}\gamma(\tilde{\lambda})B)"] & T_{\sigma(k)} \arrow[ddd, "\Phi(\gamma(\lambda))"] \\
			& & \\ \\
			T_{\sigma(k)} \arrow[r, "f_{\sigma(k)2}"] & U_{\sigma(k)1} \arrow[rr, "\Phi(B)\circ f_{\sigma(k)1}\circ \Phi(A)"] & & U_{\sigma(k)0} \arrow[r, "f_{\sigma(k)0}"] & T_{\sigma(k)}
		\end{tikzcd}
	\end{equation*}
	Now we claim the following:
	\begin{enumerate}
		\item $\Phi(B)\circ \psi(f_{k1})\circ \Phi(A) = f_{\sigma(k)1}$
		\item $\Phi(B\gamma(\tilde{\lambda})B^{-1})) = \Phi(\widetilde{\gamma(\lambda)})$
		\item $\Phi(A\gamma(\hat{\lambda})A^{-1})) = \Phi(\widehat{\gamma(\lambda)})$
	\end{enumerate}
	Given the above claim we have that
	\begin{equation*}
		\begin{tikzcd}[column sep = 1.5cm]
			T_{\sigma(k)} \arrow[r, "f_{\sigma(k)2}"] \arrow[ddd, "\Phi(\gamma(\lambda))"] & U_{\sigma(k)1} \arrow[ddd, "\Phi(\widehat{\gamma(\lambda)})"] \arrow[rr, "f_{\sigma(k)1}"] \arrow[dr, "h_{\sigma(k)}"] & & U_{\sigma(k)0} \arrow[r, "f_{\sigma(k)0}"] \arrow[ddd, "\Phi(\widetilde{\gamma(\lambda)})"] & T_{\sigma(k)} \arrow[ddd, "\Phi(\gamma(\lambda))"] \arrow[lld, dotted] \\
			& & T_{\sigma(k)}^* \arrow[ru, "g_{\sigma(k)}"] \arrow[llu, dotted] \\ \\
			T_{\sigma(k)} \arrow[r, "f_{\sigma(k)2}"] & U_{\sigma(k)1} \arrow[rr, "f_{\sigma(k)1}\qquad\quad"] \arrow[dr, "h_{\sigma(k)}"] & & U_{\sigma(k)0} \arrow[r, "f_{\sigma(k)0}"] & T_{\sigma(k)} \arrow[lld, dotted] \\
			& & T_{\sigma(k)}^* \arrow[ru, "g_{\sigma(k)}"] \arrow[llu, dotted] \arrow[from=uuu, "\tilde{\varphi}_{k}^{-1}\circ (q_\lambda{[2]})\circ\tilde{\varphi}_{k}"{pos=0.3}, crossing over]
		\end{tikzcd}
	\end{equation*}
	commutes, where $\tilde{\varphi}_k$ is given as in Lemma~\ref{lemma - tildevarphi_k}. Thus setting $q_{\gamma(\lambda)} = \tilde{\varphi}_{k}^{-1}\circ (q_\lambda{[2]})\circ\tilde{\varphi}_{k}$ shows that $T$ satisfies the mutation compatibility condition at $\sigma(k)$.

	Let us now prove the claims above. We only give the proof for (1) and (3) since the argument for (2) is similar as for (3).

	Proof of claim (1): Using that $\gamma(W) = W$ together with Lemma~\ref{lemma - complex and derivation} we have
	\begin{equation*}
		\begin{aligned}
			\underline{\sigma(\alpha)} \partial_{\overline{\sigma(\alpha)^*}, \underline{\sigma(\beta)^*}}(W)\overline{\sigma(\beta)} &= W = \gamma(W) = \gamma(\underline{\alpha})\gamma(\partial_{\overline{\alpha^*}}, \underline{\beta^*}(W))\gamma(\overline{\beta}) = \\
			&= \underline{\sigma(\alpha)}A\gamma(\partial_{\overline{\alpha^*}}, \underline{\beta^*}(W))B\overline{\sigma(\beta)} = \underline{\sigma(\alpha)}\gamma(A\partial_{\overline{\alpha^*}}, \underline{\beta^*}(W)B)\overline{\sigma(\beta)}.
		\end{aligned}
	\end{equation*}
	Thus $\gamma(A\partial_{\overline{\alpha^*}}, \underline{\beta^*}(W)B) = \partial_{\overline{\sigma(\alpha)^*}, \underline{\sigma(\beta)^*}}(W)$. The claim follows by applying $\Phi$.

	Proof of claim (3): First note that $(\underline{\sigma(\alpha)})_\alpha a = (\underline{\sigma(\alpha)})_\alpha a'$ implies that $a = a'$ since $(\underline{\sigma(\alpha)})_\alpha$ is a $D_{s(\sigma(\alpha))}$-basis of $\bigoplus_{\substack{\alpha\in Q_1 \\ t(\alpha) = k}}M_{\sigma(\alpha)}$. Now the claim follows from
	\begin{equation*}
		\begin{aligned}
			(\underline{\sigma(\alpha)})_\alpha\widetilde{\gamma_0(\lambda)}A &= \gamma_0(\lambda) (\underline{\sigma(\alpha)})_\alpha A = \gamma_0(\lambda)\gamma_1((\underline{\alpha})_\alpha) = \gamma(\lambda(\underline{\alpha})_\alpha) = \\
			&= \gamma((\underline{\alpha})_\alpha\tilde{\lambda}) = \gamma_1((\underline{\alpha})_\alpha)\gamma_0(\tilde{\lambda}) = (\underline{\sigma(\alpha)})_\alpha A\gamma_0(\tilde{\lambda}).
		\end{aligned}
	\end{equation*}
\end{proof}

\begin{myprop}\label{proposition - mutation along orbits self-injective}
	Let $(S, W)$ be a reduced self-injective species with potential with Nakayama automorphism as above. Assume that there exists a cluster tilting object $T\in \mathcal{C}$ such that $\mathrm{End}_\mathcal{C}(T)\cong \mathcal{P}(S, W)$. If $k\in Q_0$ is mutable and $(k)$ is a sparse orbit, then $\mu_{(k)}(S, W)$ is a self-injective species with potential.
\end{myprop}

\begin{proof}
	Consider the summand $U$ of $T$ that corresponds to the vertices $k, \sigma(k), \dots, \sigma^n(k)$. By Proposition~\ref{Prop - selfinjectie T=T[2]} we have that $U\cong U[2]$. The fact that $\mu_{U}^+(T)$ is self-injective follows from Proposition~\ref{Prop - U=U[2] implies selfinjective}. By Proposition~\ref{prop - no arrows mu + and mu -}, the assumptions on $k$ gives us
	\begin{equation*}
		\mu_{U}^+(T)\cong \mu_{T_{\sigma^n(k)}} \circ \cdots \circ \mu_{T_{\sigma(k)}}\circ \mu_{T_k}(T).
	\end{equation*}
	Every vertex in $(k)$ is mutable by Lemma~\ref{lemma - nakayama permutation mutable}. Thus using Theorem~\ref{Theorem A} repeatedly yields that $\mathcal{P}(\mu_{(k)}(S, W))\cong \mathrm{End}_\mathcal{C}(\mu^+_{U}(T))$ is a finite-dimensional self-injective algebra.
\end{proof}

\begin{myrem}
	Note that if $(k)$ is not a sparse orbit we still have that $\mu_{U}^+(T)$ is self-injective. When applying Theorem~\ref{Theorem A} we need our species with potential to be reduced. Thus in every mutation step we need to reduce our resulting species with potential and therefore modify our potential. Since it is not clear which $2$-cycles we remove in the reduction process, we cannot predict how the species with potential looks like after mutation along $(k)$.
\end{myrem}

Let us now define the morphism $\mu_{(k)}(\gamma):\mathcal{P}(\mu_{(k)}(S, W))\to \mathcal{P}(\mu_{(k)}(S, W))$ in the following way:
\begin{equation*}
	\begin{aligned}
		\mu_{(k)}(\gamma)(\lambda) &= \gamma(\lambda), \quad &&\text{if }\lambda\in D, \\
		\mu_{(k)}(\gamma)(a) &= \gamma(a), \quad &&\text{if } s(\alpha), t(\alpha)\not\in (k), a\in M_\alpha \\
		\mu_{(k)}(\gamma)(b\otimes a) &= \gamma(b\otimes a), \quad &&\text{if } t(\alpha) = s(\beta)\in (k), b\otimes a \in M_{[\beta\alpha]}.
	\end{aligned}
\end{equation*}
We define $\mu_{(k)}(\gamma)$ on $M_{\alpha^*}$ if $s(\alpha)\in (k)$ or $t(\alpha)\in (k)$ such that
\begin{equation*}
	\sum_{a\in \underline{\alpha}} \gamma(a) \otimes \mu_{(k)}(\gamma)(a^*) \in M_{\sigma(\alpha)} \otimes_D M_{\sigma(\alpha)^*}
\end{equation*}
is a Casimir element. With this we can extend $\mu_{(k)}(\gamma)$ to an automorphism of $\mathcal{P}(\mu_{(k)}(S, W))$. In fact, it is straightforward to verify that $\mu_{(k)}(\gamma)$ is an automorphism satisfying conditions (A) and (B).

The goal of this section is to prove the following theorem.

\begin{mythm}\label{Theorem B}
	Let $\mathcal{C}$ be a $2$-Calabi--Yau triangulated category and $T = \bigoplus_{i = 1}^n T_i\in \mathcal{C}$ a basic cluster tilting object with indecomposable summands $T_i$ such that $\mathrm{End}_\mathcal{C}(T)\cong \mathcal{P}(S, W)$. Assume that $k\in Q_0$ is mutable in $Q$ and $(k)$ is a sparse orbit. If $\gamma$ is a Nakayama automorphism of $\mathcal{P}(S, W)$ satisfying the conditions (A) and (B) above, then $\mu_{(k)}(\gamma)$ defined above is a Nakayama automorphism of $\mathcal{P}(\mu_{(k)}(S, W))$.
\end{mythm}

\begin{myrem}
	It is enough to prove the analogue version of Theorem~\ref{Theorem B} for $\mathrm{End}_\mathcal{C}(T)$ and $\mathrm{End}_\mathcal{C}(\mu_{(k)}(T))$.
\end{myrem}

\begin{proof}(Proof of Theorem~\ref{Theorem B})
	The Nakayama automorphism $\gamma$ will correspond to some Nakayama automorphism $\psi$ of $\mathrm{End}_\mathcal{C}(T)$, which in turn will be given in terms of a morphism $\varphi = \bigoplus_{i\in Q_0}\varphi_i: T\to T[2]$. We want to construct a morphism $\tilde{\varphi} = \bigoplus_{i\in Q_0}\tilde{\varphi_i}: \mu_{(k)}(T)\to \mu_{(k)}(T)[2]$. Let us define $\tilde{\varphi}_i = \varphi_i = \varphi|_{T_{\sigma(k)}}: T_{\sigma(k)}\to T_i[2]$ for $i\not\in (k)$. By applying Lemma~\ref{lemma - tildevarphi_k} for every $k'\in (k)$ we get an isomorphism $\tilde{\varphi}_{k'}: T_{\sigma(k')}^*\to T_{k'}^*[2]$ such that
	\begin{equation*}
		\begin{tikzcd}
			T_{\sigma(k')} \ar[r, "\psi(f_{k'2})"] \ar[d, "\varphi_{k'}"] & U_{\sigma(k')1} \ar[r, "\tilde{\psi}(h_{k'})"] \ar[d, "\varphi_{U_{k'1}}"] & T_{\sigma(k')}^* \ar[r] \ar[d, "\tilde{\varphi}_{k'}"] & T_{k'}[1]  \ar[d, "\varphi_{k'}{[1]}"] \\
			T_{k'}[2] \ar[r, "f_{k'2}{[2]}"] & U_{k'1}[2] \ar[r, "h_{k'}{[2]}"]  & T_{k'}^*[2] \ar[r] & T_{k'}[3]
		\end{tikzcd}
	\end{equation*}
	commutes for all $k'\in (k)$, where $\tilde{\psi}$ is the Nakayama automorphism defined by $\tilde{\varphi}$. Now from the description of the isomorphism $\mathrm{End}_\mathcal{C}(\mu_{(k)}(T))\cong \mathcal{P}(\mu_{(k)}(S, W))$ we get that $\tilde{\psi}$ corresponds to the Nakayama automorphism $\mu_{(k)}(\gamma)$.
\end{proof}

\begin{myrem}$\space$
	\begin{enumerate}
		\item Note that the Nakayama automorphism $\mu_{(k)}(\gamma)$ satisfies the conditions (A) and (B) above, and thus we can iterate the mutation process, as long as enough $2$-cycles can be removed.
		\item If $\mathcal{P}(S, W)$ is symmetric, i.e. it is self-injective with $\gamma = \id_{\mathcal{P}(S, W)}$, then $\gamma$ trivially satisfies (A) and (B). Applying Theorem~\ref{Theorem B} we set $\mu_{(k)}(\gamma) = \id_{\mathcal{P}(S, W)}$ as well, and so $\mu_{(k)}(S, W) = \mu_k(S, W)$ also has a symmetric Jacobian algebra.
	\end{enumerate}
\end{myrem}

\section{Application of the derived Auslander-Iyama correspondence}\label{Section - Jasso-muro plus examples}
In this section we will combine the results of \cite{soderberg2022preprojective,herschend2011n,amiot2008petites,amiotclustercategories} to provide new examples for the derived Auslander-Iyama correspondence \cite{jasso2023derived}. We explicitly compute some examples.

\begin{mythm}\cite[Theorem A]{jasso2023derived}\label{theorem - Jasso-Muro}
	Suppose that $\K$ is a perfect field and $d\ge 1$ an integer. There are bijective correspondences between the following:
	\begin{enumerate}
		\item Equivalence classes of pairs $(\mathcal{C}, T)$ consisting of,
		\begin{enumerate}
			\item an algebraic triangulated category $\mathcal{C}$ with finite-dimensional morphism spaces and split idempotents and
			\item a basic $d\Z$-cluster tilting object $T\in \mathcal{C}$, that is a $d$-cluster tilting object such that $\mathcal{C}(T, T[i]) = 0$ for all $i\not\in d\Z$.
		\end{enumerate}
		\item Equivalence classes of pairs $(\Pi, I)$ consisting of
		\begin{enumerate}
			\item a basic finite-dimensional self-injective algebra $\Pi$ that is twisted $(d+2)$-periodic and
			\item an invertible $\Pi$-bimodule $I$ such that $\Omega_{\Pi^e}^{d+2}(\Pi)\cong I$ in the stable category of $\Pi$-bimodules.
		\end{enumerate}
	\end{enumerate}
\end{mythm}

We compute non-trivial examples using $l$-homogeneous representation finite species.

We follow the construction in \cite{amiot2008petites, amiotclustercategories} (In particular Section~4 in \cite{amiotclustercategories}). Let $\Lambda$ be a finite dimensional algebra of global dimension at most $2$. We define $\Gamma$ to be the DG algebra $\Lambda\oplus D\Lambda[-3]$ with zero differential. We have a natural projection $\Gamma \to \Lambda$ which induces a restriction functor
\begin{equation}\label{eq - restriction functor}
	\mathcal{D}^b_\Lambda \to \mathcal{D}^b_\Gamma
\end{equation}
where $\mathcal{D}^b_\Lambda$ and $\mathcal{D}^b_\Gamma$ are the bounded derived categories of $\Lambda$ and $\Gamma$ respectively.

The Amiot cluster category is defined as
\begin{equation*}
	\mathcal{C}_\Lambda := \mathrm{thick}_{\mathcal{D}^b_\Gamma}(\Lambda)/\mathrm{perf}(\Gamma),
\end{equation*}
where $\mathrm{thick}_{\mathcal{D}^b_\Gamma}(\Lambda)$ is the smallest thick subcategory of $\mathcal{D}^b_\Gamma$ containing the image of $\Lambda$ through \eqref{eq - restriction functor} and $\mathrm{perf}(\Gamma)$ is the category containing all of the perfect complexes over $\Gamma$.

\begin{mylemma}\cite{amiot2008petites, amiotclustercategories}
	The restriction functor induces a functor $\pi:\mathcal{D}^b_\Lambda \to \mathcal{C}_\Lambda$ which commutes with the Serre functors of the two categories.
\end{mylemma}

\begin{mythm}\label{theorem - amiot}\cite{amiot2008petites, amiotclustercategories}
	If $\Lambda$ is $\tau_2$-finite, i.e. $\tau_2^{-i}\Lambda = 0$ for sufficiently large $i$, then:
	\begin{enumerate}
		\item The Amiot cluster category $\mathcal{C}_\Lambda$ is $2$-Calabi--Yau.
		\item $\pi\Lambda$ is a $2$-cluster tilting object in the Amiot cluster category $\mathcal{C}_\Lambda$. Moreover, the endomorphism algebra $\mathrm{End}_{\mathcal{C}_\Lambda}(\pi\Lambda)^{op}$ is isomorphic to $\Pi_3(\Lambda)$.
	\end{enumerate}
\end{mythm}

\begin{myrem}
	When the algebra $\Lambda$ is $2$-representation finite, it is also $\tau_2$-finite, and therefore satisfies the condition in Theorem~\ref{theorem - amiot}. Moreover, if $\Pi_3(\Lambda)$ self-injective then by Proposition~\ref{Prop - selfinjectie T=T[2]} we have that
	\begin{equation*}
		(\pi\Lambda)[2] \cong \pi\Lambda.
	\end{equation*}
	Hence $\pi\Lambda$ is a $2\Z$-cluster tilting object in $\mathcal{C}_\Lambda$ and in turn fits into Theorem~\ref{theorem - Jasso-Muro}.
\end{myrem}

We can construct examples of $2$-representation finite $\K$-algebras by using algebras that fall in the following definition and using the corollary that follows.

\begin{mydef}\cite[Definition 1.2]{herschend2011n}
	We say that an $d$-representation finite algebra $\Lambda$ is $l$-homogeneous if $\tau_d^{l-1} D\Lambda \cong \Lambda$.
\end{mydef}

\begin{mycor}\cite[Corollary 1.5]{herschend2011n}\label{cor - A x B is n_1 + n_2 RF}
	Let $\mathbb{K}$ be a perfect field and $l$ a positive integer. If $\Lambda_i$ is $l$-homogeneous $d_i$-representation finite for each $i\in \{1, 2\}$, then $\Lambda_1\otimes_\K \Lambda_2$ is an $l$-homogeneous $(d_1 + d_2)$-representation-finite algebra with a $(d_1 + d_2)$-cluster tilting module $\bigoplus_{i=0}^{l-1}(\tau_{d_1}^{-i}\Lambda_1\otimes_\K \tau_{d_2}^{-i}\Lambda_2)$.
\end{mycor}

The following theorem by Dlab and Ringel describes all species that are representation finite.

\begin{mythm}\label{theorem - rep finite species is dynkin}\cite[Theorem B]{dlab1975algebras}
	An acyclic species $S$ is representation finite, i.e. the modules category $T(S)\mathrm{-mod}$ is representation finite, if and only if its diagram if a finite disjoint union of Dynkin diagrams.
\end{mythm}

From the article \cite{soderberg2022preprojective} we have a complete list of species with their tensor algebras being representation finite and $l$-homogeneous.

\begin{mycor}\label{cor - l homogeneous species}\cite[Corollary 3.7]{soderberg2022preprojective}
	Let $S$ be a representation finite species. Then $S$ is $l$-homogeneous if $Q$ is stable under its Nakayama permutation $\sigma$. Moreover, for the different cases, the integer $l$ is
	\begin{center}
		\begin{tabular}{c|c|c|c|c|c|c|c|c|c}
			$\Delta$ & $A_n$ & $B_n$ & $C_n$ & $D_n$ & $E_6$ & $E_7$ & $E_8$ & $F_4$ & $G_2$ \\\hline
			$l$ & $\frac{n+1}{2}$ & $n$ & $n$ & $n-1$ & $6$ & $9$ & $15$ & $6$ & $3$
		\end{tabular}
	\end{center}
\end{mycor}

\begin{myrem}
	By \cite[Lemma 3.3]{soderberg2022preprojective} we can view $\sigma$ as a diagram automorphism, and therefore can check if the orientation of the quiver $Q$ is preserved under the action of $\sigma$.
\end{myrem}

Let $\Lambda = T(S^1)\otimes_\K T(S^2)$ where $S^i$ is a $l$-homogeneous representation finite species for each $i\in \{1, 2\}$. By Corollary~\ref{cor - A x B is n_1 + n_2 RF} there is a $2$-cluster tilting $\Lambda$-module $M = \bigoplus_{i=0}^{l-1}(\tau^{-i}\Lambda_1\otimes_\K \tau^{-i}\Lambda_2)$. We set our $2$-Calabi--Yau triangulated category to $\mathcal{C} = \mathcal{C}_\Lambda$. Also let $(S, W)$ be as in Proposition~\ref{prop - preprojective algebra of tensor species}. Since $\mathrm{End}_\mathcal{C}(\pi\Lambda)^{op} = \Pi_3(\Lambda)$ we get a functor $\Phi: \mathrm{proj}\widehat{T}(S)\to \mathcal{C}$ such that $\Phi(P) = T = \pi\Lambda\in \mathcal{C}$, which induces an isomorphism $\mathcal{P}(S, W)\cong \mathrm{End}_\mathcal{C}(T)^{op}$.

\begin{mylemma}\label{lemma - tensor product is mutable}
	In the above setting $T$ satisfies the mutation compatibility condition at every vertex $k\in Q_0$.
\end{mylemma}

\begin{proof}
	We will prove this in two steps. First we will show that weak $2$-almost split sequences in $\mathcal{C}$ correspond to $2$-almost split sequences in $\add(M)\subseteq\mathrm{mod}(\Lambda)$. Then using the properties of $2$-almost split sequences we can show that $T$ satisfies the mutation compatibility condition at $k = (i, j)\in (Q^1\tilde{\otimes}Q^2)_0$.

	Note that if $S^1$ and $S^2$ are $1$-homogeneous, then $\Lambda$ is semisimple and $T$ automatically satisfies the mutation compatibility condition for every vertex. If this is not the case, then $P_i^1$ and $P_j^2$ are non-injective. The almost split sequences starting at $P_i^1$ and $P_j^2$ are
	\begin{equation*}
		P^1_i \xrightarrow{\begin{bmatrix}
			(\underline{\alpha})_\alpha \\
			(\underline{\alpha^*})_{\alpha^*}
		\end{bmatrix}} \bigoplus_{\substack{\alpha\in Q^1_1 \\ t(\alpha) = i}} (P^1_{s(\alpha)})^{|\underline{\alpha}|} \oplus \bigoplus_{\substack{\alpha^*\in (Q^1_1)^* \\ t(\alpha^*) = i}} \tau^-(P^1_{s(\alpha^*)})^{|\underline{\alpha^*}|} \xrightarrow{\begin{bmatrix}
			(\overline{\alpha^*})_\alpha & -(\overline{\alpha})_{\alpha^*}
		\end{bmatrix}} \tau^- (P^1_i)
	\end{equation*}
	and
	\begin{equation*}
		P^2_j \xrightarrow{\begin{bmatrix}
			(\underline{\beta})_\beta \\
			(\underline{\beta^*})_{\beta^*}
		\end{bmatrix}} \bigoplus_{\substack{\beta\in Q^2_1 \\ t(\beta) = j}} (P^2_{s(\beta)})^{|\underline{\beta}|} \oplus \bigoplus_{\substack{\beta^*\in (Q^2_1)^* \\ t(\beta^*) = j}} \tau^-(P^2_{s(\beta^*)})^{|\underline{\beta^*}|} \xrightarrow{\begin{bmatrix}
			(\overline{\beta^*})_\beta & -(\overline{\beta})_{\beta^*}
		\end{bmatrix}} \tau^- (P^2_j).
	\end{equation*}
	Applying \cite[Theorem 4.11]{pasquali2019tensor} yields the $2$-almost split sequence starting at $P_i^1\otimes_\K P_j^2$ given as the mapping cone of \eqref{eq - mapping cone 2 almost split}.
	\begin{equation}\label{eq - mapping cone 2 almost split}
		\begin{adjustbox}{max width = \textwidth}
			\begin{tikzcd}[ampersand replacement=\&, column sep = 2 cm, row sep=2 cm]
				P_i^1\otimes_\K P_j^2 \arrow[r, "\begin{bmatrix}
					(\underline{\alpha})_\alpha \otimes e_j \\
					e_i\otimes (\underline{\beta})_\beta
				\end{bmatrix}"] \arrow[d, "(\underline{\alpha^*}\otimes \underline{\beta^*})_{\alpha^*, \beta^*}"] \& \begin{matrix}
					\underset{\substack{\alpha\in Q^1 \\ t(\alpha) = i}}{\displaystyle\bigoplus} \left((P^1_{s(\alpha)})^{|\underline{\alpha}|}\otimes_\K P_j^2\right) \\ \oplus \displaystyle\bigoplus_{\substack{\beta\in Q^2 \\ t(\beta) = j}} \left(P_i^1\otimes_\K (P^2_{s(\beta)})^{|\underline{\beta}|}\right)
				\end{matrix} \arrow[r, "{\begin{bmatrix}
					e_{s(\alpha)} \otimes (\underline{\beta})_\beta & -(\underline{\alpha})_\alpha \otimes e_{s(\beta)} 
				\end{bmatrix}}"] \arrow[d, "{\begin{bmatrix}
					0 & (\underline{\alpha^*}\otimes \overline{\beta^*})_{\alpha^*, \beta} \\
					(\overline{\alpha^*}\otimes \underline{\beta^*})_{\alpha, \beta^*} & 0
				\end{bmatrix}}"] \& \underset{\substack{(\alpha, \beta)\in Q^1_1\times Q^2_1 \\ t(\alpha) = i, t(\beta) = i}}{\displaystyle\bigoplus} (P^1_{s(\alpha)})^{|\underline{\alpha}|} \otimes_\K (P^2_{s(\beta)})^{|\underline{\beta}|} \arrow[d, "(\overline{\alpha^*}\otimes \overline{\beta^*})_{\alpha, \beta}"] \\
				\underset{\substack{(\alpha^*, \beta^*)\in (Q^1_1)^*\times (Q^2_1)^* \\ t(\alpha^*) = i, t(\beta^*) = j}}{\displaystyle\bigoplus} \tau^-(P^1_{s(\alpha^*)})^{|\underline{\alpha^*}|} \otimes_\K \tau^-(P^2_{s(\beta^*)})^{|\underline{\beta^*}|} \arrow[r, "\begin{bmatrix}
					e_{s(\alpha^*)}\otimes (\overline{\beta})_{\beta^*} \\
					(\overline{\alpha})_{\alpha^*}\otimes e_{s(\beta^*)}
				\end{bmatrix}"] \& \begin{matrix}
					\underset{\substack{\alpha^*\in (Q^1)^* \\ t(\alpha^*) = i}}{\displaystyle\bigoplus} \left(\tau^-(P^1_{s(\alpha^*)})^{|\underline{\alpha^*}|}\otimes_\K \tau^-(P_j^2)\right) \\ \oplus \displaystyle\bigoplus_{\substack{\beta^*\in (Q^2)^* \\ t(\beta^*) = j}} \left(\tau^-(P_i^1)\otimes_\K (P^2_{s(\beta^*)})^{|\underline{\beta^*}|}\right)
				\end{matrix} \arrow[r, "{\begin{bmatrix}
					-(\overline{\alpha})_{\alpha^*}\otimes e_j & e_i\otimes (\overline{\beta})_{\beta^*}
				\end{bmatrix}}"] \& \tau^-(P_i^1) \otimes_\K \tau^-(P_j^2).
			\end{tikzcd}
		\end{adjustbox}
	\end{equation}
	Under the natural functor from $\mathrm{mod}(\Lambda) \to \mathcal{C}_\Lambda = \mathcal{C}$ the image of the mapping cone of \eqref{eq - mapping cone 2 almost split} coincides with the weak $2$-almost split sequence \eqref{eq - phi 2 almost split sequence}. This proves the first part.

	With the above result the mutation compatibility condition diagram \eqref{eq - mut cond diagram} corresponds to
	\begin{equation}
		\begin{tikzcd}
			P_k \arrow[r, "f'_{k2}"] \arrow[ddd, "\lambda"] & U'_{k1} \arrow[ddd, "\hat{\lambda}"] \arrow[rr, "f'_{k1}"] \arrow[rd, "h'_k"] & & U'_{k0} \arrow[r, "f'_{k0}"] \arrow[ddd, "\tilde{\lambda}"] & \tau_2^-P_k \arrow[ddd, "\lambda"] \\
			& & P_k^* \arrow[ru, "g'_k"] \arrow[ddd, bend left, "q'_\lambda"] \arrow[ddd, bend right, swap, "q''_\lambda"] \\ \\
			P_k \arrow[r, "f'_{k2}"] & U'_{k1} \arrow[rr, "f'_{k1}"] \arrow[rd, "h'_k"] & & U'_{k0} \arrow[r, "f'_{k0}"] & \tau_2^-P_k \\
			& & P_k^* \arrow[ru, "g'_k"]
		\end{tikzcd}
	\end{equation}
	where $q'_\lambda$ and $q''_\lambda$ are given as follows. Since $P_k^* = \coker(f'_{k2})$ we set $q'_\lambda$ to be the unique morphism such that the left side of the diagram commutes. Similarly, we use that $P_k^* = \ker(f'_{k0})$ to define $q'_\lambda$. It remains to show that $q'_\lambda = q''_\lambda$. We have
	\begin{equation*}
		g'_k\circ q''_\lambda \circ h'_k = g'_k\circ h'_k\circ \hat{\lambda} = f'_{k1}\circ \hat{\lambda} = \tilde{\lambda}\circ f'_{k1} = \tilde{\lambda} \circ g'_k\circ h'_k = g'_k\circ q'_\lambda\circ h'_k
	\end{equation*}
	and using that $h'_k$ and $g'_k$ are a epimorphism and an monomorphism respectively we are done.
\end{proof}

Combining Corollary~\ref{cor - A x B is n_1 + n_2 RF} and Corollary~\ref{cor - l homogeneous species} we can generate a lot of $2$-representation finite algebras given by species and relations, whose self-injective preprojective algebras fit in Theorem~\ref{theorem - Jasso-Muro}. By mutating along Nakayama orbits Proposition~\ref{proposition - mutation along orbits self-injective} ensures that the result will still be self-injective. Since mutating our species with potential does not change the cluster category, it will still satisfy the desired condition for Theorem~\ref{theorem - Jasso-Muro}, and therefore expanding the collection of examples for Theorem~\ref{theorem - Jasso-Muro}. Additionally, the Nakayama automorphism for each example explicitly given by using \cite[Theorem 5.3]{soderberg2022preprojective}, \cite[Proposition 10.11]{soderberg2022preprojective} and Theorem~\ref{Theorem B}. Let us now compute some explicit examples.

\begin{myex}\label{example - A times B}
	Let us choose the quivers
	\begin{equation*}
		Q^1: 1\xrightarrow{\alpha} 2 \xleftarrow{\beta} 3, \qquad Q^2: 1\xrightarrow{\gamma}2
	\end{equation*}
	and define the $\R$-species $S^1$ and $S^2$ over $Q^1$ and $Q^2$ respectively as
	\begin{equation*}
		S^1: \R\xrightarrow{\lsub{2}\R_1}\R\xleftarrow{\lsub{2}\R_3}\R, \qquad S^2: \C \xrightarrow{\lsub{2}\C_1}\R.
	\end{equation*}
	We let $e_1, e_2$ and $e_3$ be the idempotents for $S^1$ and $f_1$ and $f_2$ be the idempotents for $S^2$. Reading from \cite[Theorem 3.1]{soderberg2022preprojective} the Nakayama permutations for $S^1$ and $S^2$ are $\sigma^1(i) = 4-i$ and $\sigma^2(i) = i$ respectively. Both of $T(S^1)$ and $T(S^2)$ are $1$-representation finite by Theorem~\ref{theorem - rep finite species is dynkin} and $2$-homogeneous by Corollary~\ref{cor - l homogeneous species}. Applying Corollary~\ref{cor - A x B is n_1 + n_2 RF} yields the $2$-representation finite $\R$-algebra $T(S^1)\otimes_\R T(S^2)$. Proposition~\ref{prop - preprojective algebra of tensor species} tells us that $\Pi_3(T(S^1)\otimes_\R T(S^2))\cong \mathcal{P}(S(S^1, S^2), W(S^1, W^2))$, where
	\begin{equation*}
		\begin{tikzcd}[row sep = 1 cm, column sep = 2 cm]
			& \R \otimes_\R \C \arrow[r, "\lsub{2}\R_1 \otimes_\R \C"] \arrow[dd, "\R \otimes_\R \lsub{2}\C_1"] & \R \otimes_\R \C \arrow[dd, "\R \otimes_\R \lsub{2}\C_1"] & \R \otimes_\R \C \arrow[l, "\lsub{2}\R_3 \otimes_\R \C"'] \arrow[dd, "\R \otimes_\R \lsub{2}\C_1"] \\
			S(S^1, S^2): \\
			& \R \otimes_\R \R \arrow[r, "\lsub{2}\R_1 \otimes_\R \R"] & \R \otimes_\R \R \arrow[luu, "\lsub{2}\R_1^* \otimes_\R \lsub{2}\C_1^*"'] \arrow[ruu, "\lsub{2}\R_3^* \otimes_\R \lsub{2}\C_1^*"'] & \R \otimes_\R \R \arrow[l, "\lsub{2}\R_3 \otimes_\R \R"']
		\end{tikzcd}
	\end{equation*}
	is an $\R$-species over the quiver
	\begin{equation*}
		\begin{tikzcd}[row sep = 1 cm, column sep = 2 cm]
			& (1, 1) \arrow[r, "(\alpha{,} e)"] \arrow[dd, "(e_1{,} \gamma)"] & (2,1) \arrow[dd, "(e_2{,} \gamma)"] & (3,1) \arrow[l, "(\beta{,} e_1)"'] \arrow[dd, "(e_3{,} \gamma)"] \\ 
			Q^1\tilde\otimes Q^2: \\
			& (1,2) \arrow[r, "(\alpha{,} e_2)"] & (2,2) \arrow[luu, "(\alpha^* {,} \gamma^*)"'] \arrow[ruu, "(\beta^* {,} \gamma^*)"'] & (3,2) \arrow[l, "(\beta{,} e_2)"']
		\end{tikzcd}
	\end{equation*}
	and the potential $W(S^1, S^2)$ given by \eqref{eq - potential for tensor species}, i.e.
	\begin{equation*}
		\begin{aligned}
			W(S^1, S^2) =& \sum_{(a, c')\in \underline{\alpha}\times \overline{\gamma}} (a\otimes f_1)\otimes(a^*\otimes {c'}^*)\otimes(e_2\otimes c') + \\
			&+ \sum_{(b, c')\in \underline{\beta}\times \overline{\gamma}} (b\otimes f_1)\otimes(b^*\otimes {c'}^*)\otimes(e_2\otimes c') + \\
			&- \sum_{(a', c)\in \overline{\alpha}\times \underline{\gamma}} (e_1\otimes c)\otimes({a'}^*\otimes c^*)\otimes(a'\otimes f_2) + \\
			&- \sum_{(b', c)\in \overline{\beta}\times \underline{\gamma}} (e_3\otimes c)\otimes({b'}^*\otimes c^*)\otimes(b'\otimes f_2) = \\
			=& (\lsub{2}1_1\otimes f_1) \otimes (\lsub{2}1_1^*\otimes \lsub{2}1_1^*)\otimes (e_2\otimes \lsub{2}1_1) + (\lsub{2}1_1\otimes f_1) \otimes (\lsub{2}1_1^*\otimes \lsub{2}i_1^*)\otimes (e_2\otimes \lsub{2}i_1) + \\
			&+ (\lsub{2}1_3\otimes f_1) \otimes (\lsub{2}1_3^*\otimes \lsub{2}1_1^*)\otimes (e_2\otimes \lsub{2}1_1) + (\lsub{2}1_3\otimes f_1) \otimes (\lsub{2}1_3^*\otimes \lsub{2}i_1^*)\otimes (e_2\otimes \lsub{2}i_1) + \\
			&- (e_1 \otimes \lsub{2}1_1) \otimes (\lsub{2}1_1^*\otimes \lsub{2}1_1^*) \otimes (\lsub{2}1_1\otimes f_2) + \\
			&- (e_3 \otimes \lsub{2}1_1) \otimes (\lsub{2}1_3^*\otimes \lsub{2}1_1^*) \otimes (\lsub{2}1_3\otimes f_2).
		\end{aligned}
	\end{equation*}

	Before we start to mutate along a Nayakama orbit, we will first simplify our species $S(S^1, S^2)$. We use the fact that $\R^*\cong \R$, $\C^*\cong \C$, $\R \otimes_\R \C \cong \C$ and $\R\otimes_\R \R$, where the former two isomorphisms are given as
	\begin{equation*}
		\begin{aligned}
			\R &\xrightarrow{\sim} \R^*, \\
			1_\R &\mapsto 1_\R^*,
		\end{aligned}
	\end{equation*}
	and
	\begin{equation*}
		\begin{aligned}
			\C &\xrightarrow{\sim}\C^*, \\
			a+bi &\mapsto a1^* - bi^*.
		\end{aligned}
	\end{equation*}
	We also relabel the vertices in our quiver $Q^1\tilde{\otimes}Q^2$ such that our data becomes
	\begin{equation*}
		\begin{tikzcd}[row sep = 0.3 cm, column sep = 1 cm]
			& \C \arrow[r, "\lsub{2}\C_1"] \arrow[dd, "\lsub{4}\C_1"] & \C \arrow[dd, "\lsub{5}\C_2"] & \C \arrow[l, "\lsub{2}\C_3"'] \arrow[dd, "\lsub{6}\C_3"] \\
			S(S^1, S^2): \\
			& \R \arrow[r, "\lsub{5}\R_4"] & \R \arrow[luu, "\lsub{1}\C_5"'] \arrow[ruu, "\lsub{3}\C_5"'] & \R \arrow[l, "\lsub{5}\R_6"']
		\end{tikzcd}
	\end{equation*}
	\begin{equation*}
		\begin{tikzcd}[row sep = 0.3 cm, column sep = 1 cm]
			& 1 \arrow[r] \arrow[dd] & 2 \arrow[dd] & 3 \arrow[l] \arrow[dd] \\ 
			Q^1\tilde\otimes Q^2: \\
			& 4 \arrow[r] & 5 \arrow[luu] \arrow[ruu] & 6 \arrow[l]
		\end{tikzcd}
	\end{equation*}
	\begin{equation}\label{eq - potential in example at end 1}
		\begin{aligned}
			W(S^1, S^2) =& \lsub{2}1_1\otimes \lsub{1}1_5\otimes \lsub{5}1_2 - \lsub{2}1_1\otimes \lsub{1}i_5\otimes \lsub{5}i_2 + \\
			&+ \lsub{2}1_3\otimes \lsub{3}1_5\otimes \lsub{5}1_2 - \lsub{2}1_3\otimes \lsub{3}i_5\otimes \lsub{5}i_2 + \\
			&- \lsub{4}1_1\otimes \lsub{1}1_5\otimes \lsub{5}1_4 + \\
			&- \lsub{6}1_3\otimes \lsub{3}1_5\otimes \lsub{5}1_6.
		\end{aligned}
	\end{equation}
	The Nakayama permutation of $(Q^1\tilde{\otimes}Q^2)_0$ is given by
	\begin{equation*}
		\sigma = \left(\begin{matrix}
			1 & 2 & 3 & 4 & 5 & 6 \\
			3 & 2 & 1 & 6 & 5 & 4
		\end{matrix}\right)
	\end{equation*}
	which is a consequence of \cite[Proposition 10.11]{soderberg2022preprojective}. Using \cite[Theorem 5.3]{soderberg2022preprojective} together with \cite[Proposition 10.11]{soderberg2022preprojective} we can compute the Nakayama automorphism $\gamma$ of $\mathcal{P}(S(S^1, S^2), W(S^1, S^2))$ to be explicitly given as $\gamma(\lsub{i}1_j) = \lsub{\sigma(i)}1_{\sigma(j)}$. In particular, conditions (A) and (B) are satisfied.

	Now let us mutate along an orbit of the Nakayama permutation. Consider the Nakayama orbit $(2) = \{2\}$, which is sparse as it has one element. Recall that the semi-mutation at $2$ is only defined when $e_2W = We_2 = 0$, and thus we apply $\varepsilon_r$ on the first two rows of $W$ before we mutate. By Proposition~\ref{proposition - mutation along orbits self-injective} $\mu_{2}(S(S^1, S^2), W(S^1, S^2)) = (S', W')$ is a self-injective species with potential, where the species $S'$ is given as
	
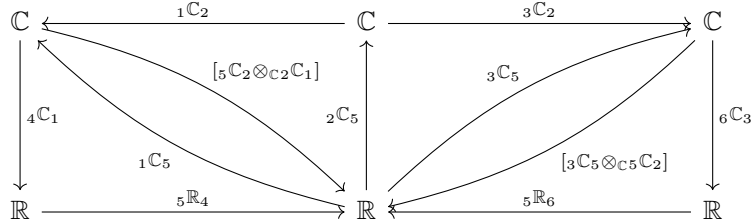
\begin{equation*}
		\begin{tikzcd}[row sep = 1 cm, column sep = 4 cm, bezier bounding box]
			\C \arrow[dd, "\lsub{4}\C_1"] \arrow[rdd, "{[}\lsub{5}\C_2\otimes_\C \lsub{2}\C_1{]}", bend left=15] & \C \arrow[l, "\lsub{1}\C_2"'] \arrow[r, "\lsub{3}\C_2"] & \C \arrow[dd, "\lsub{6}\C_3"] \arrow[ldd, "{[}\lsub{3}\C_5\otimes_\C \lsub{5}\C_2{]}", bend left=15] \\ \\
			\R \arrow[r, "\lsub{5}\R_4"] & \R \arrow[luu, "\lsub{1}\C_5", bend left=15] \arrow[ruu, "\lsub{3}\C_5", bend left=15] \arrow[uu, "\lsub{2}\C_5"] & \R \arrow[l, "\lsub{5}\R_6"']
		\end{tikzcd}
	\end{equation*}
	and its potential given as
	\begin{equation*}
		\begin{aligned}
			W' =& [\lsub{5}1_2\otimes \lsub{2}1_1]\otimes \lsub{1}1_5 + [\lsub{5}1_2\otimes \lsub{2}1_3]\otimes \lsub{3}1_5 - \lsub{4}1_1\otimes \lsub{1}1_5\otimes \lsub{5}1_4 - \lsub{6}1_3\otimes \lsub{3}1_5\otimes \lsub{5}1_6 + \\
			&+ \lsub{2}1_5\otimes [\lsub{5}1_2\otimes \lsub{2}1_1]\otimes \lsub{1}1_2 +  \lsub{2}i_5\otimes [\lsub{5}i_2\otimes \lsub{2}1_1]\otimes \lsub{1}1_2 + \\
			&+ \lsub{2}1_5\otimes [\lsub{5}1_2\otimes \lsub{2}1_3]\otimes \lsub{3}1_2 + \lsub{2}i_5\otimes [\lsub{5}i_2\otimes \lsub{2}1_3]\otimes \lsub{3}1_2.
		\end{aligned}
	\end{equation*}

	Now let us consider the Nakayama orbit $(5) = \{5 \}$. By Proposition~\ref{proposition - mutation along orbits self-injective} $\mu_{5}(S(S^1, S^2), W(S^1, S^2)) = (S', W')$ is a self-injective species with potential which is explicitly given as

	
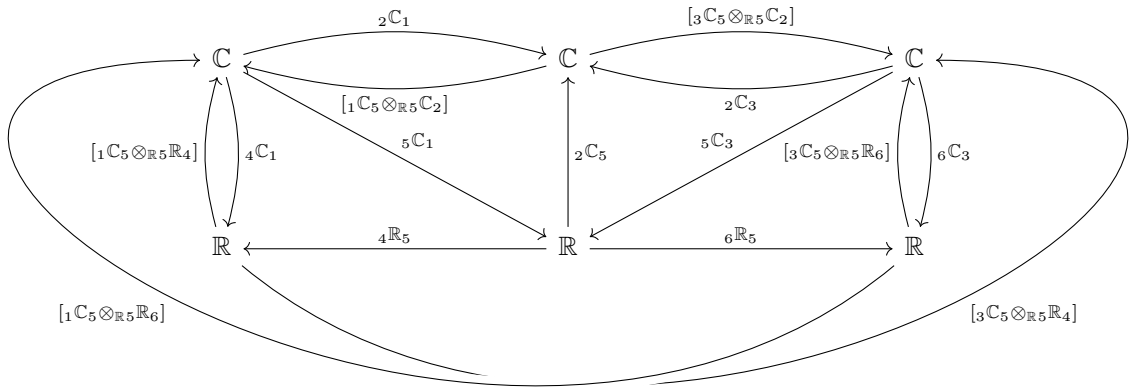
\begin{equation*}
		\begin{tikzcd}[row sep = 2 cm, column sep = 4 cm, bezier bounding box]
			\C \arrow[r, "\lsub{2}\C_1", bend left=15] \arrow[rd, "\lsub{5}\C_1"] \arrow[d, "\lsub{4}\C_1", bend left=15] & \C \arrow[l, "{[}\lsub{1}\C_5\otimes_\R \lsub{5}\C_2{]}", bend left=15] \arrow[r, "{[}\lsub{3}\C_5\otimes_\R \lsub{5}\C_2{]}", bend left=15] & \C \arrow[ld, "\lsub{5}\C_3"'] \arrow[l, "\lsub{2}\C_3", bend left=15] \arrow[d, "\lsub{6}\C_3", bend left=15] \\
			\R \arrow[u, "{[}\lsub{1}\C_5\otimes_\R \lsub{5}\R_4{]}", bend left=15] \arrow[urr, "{[}\lsub{3}\C_5\otimes_\R \lsub{5}\R_4{]}"', out=-40, in=0, looseness = 2] & \R \arrow[l, "\lsub{4}\R_5"'] \arrow[r, "\lsub{6}\R_5"] \arrow[u, "\lsub{2}\C_5"'] & \R \arrow[u, "{[}\lsub{3}\C_5\otimes_\R \lsub{5}\R_6{]}", bend left=15] \arrow[ull, "{[}\lsub{1}\C_5\otimes_\R \lsub{5}\R_6{]}", out=220, in=-180, looseness = 2, crossing over]
		\end{tikzcd}
	\end{equation*}
	with its potential 
	\begin{equation}\label{eq - potential example at end 2}
		\begin{aligned}
			W' =& \lsub{2}1_1\otimes [\lsub{1}1_5\otimes \lsub{5}1_2] - \lsub{2}1_1\otimes [\lsub{1}i_5\otimes \lsub{5}i_2] + \lsub{2}1_3\otimes [\lsub{3}1_5\otimes \lsub{5}1_2] - \lsub{2}1_3\otimes [\lsub{3}i_5\otimes \lsub{5}i_2] + \\
			&- \lsub{4}1_1\otimes [\lsub{1}1_5\otimes \lsub{5}1_4] - \lsub{6}1_3\otimes [\lsub{3}1_5\otimes \lsub{5}1_6] + \\
			&+ \lsub{5}1_1\otimes [\lsub{1}1_5\otimes \lsub{5}1_4]\otimes \lsub{4}1_5 + \lsub{5}1_1\otimes [\lsub{1}1_5\otimes \lsub{5}1_6]\otimes \lsub{6}1_5 + \lsub{5}1_1\otimes [\lsub{1}1_5\otimes \lsub{5}1_2]\otimes \lsub{2}1_5 + \\
			&+ \lsub{5}1_3\otimes [\lsub{3}1_5\otimes \lsub{5}1_2]\otimes \lsub{2}1_5 + \lsub{5}1_3\otimes [\lsub{3}1_5\otimes \lsub{5}1_4]\otimes \lsub{4}1_5 + \lsub{5}1_3\otimes [\lsub{3}1_5\otimes \lsub{5}1_6]\otimes \lsub{6}1_5.
		\end{aligned}
	\end{equation}

	In both of the examples above, from Theorem~\ref{Theorem B} we have that the Nakayama automorphism $\gamma'$ of $\mathcal{P}(S', W')$ is given by $\gamma'(\lsub{i}1_j) = \lsub{\sigma(i)}1_{\sigma(j)}$ and $\gamma([\lsub{i}1_2\otimes \lsub{2}1_j]) = [\lsub{\sigma(i)}1_k\otimes \lsub{k}1_{\sigma(j)}]$, where $k = 2$ or $k=5$ depending on which Nakayama orbit we mutated along.
\end{myex}

Let us partially deal with the case when we get a direct sum of division rings at each vertex instead of division rings when taking tensor products. We first introduce some notation and useful results. Let $M$ be a $\mathbb{L}$-vector space, where $\mathbb{L}$ is a Galois field extension of $\K$. We define $\lsub{\sigma}M$ as the vector space with the same additive structure as $M$ but twisting the $\mathbb{L}$-vector space structure with an automorphism $\sigma\in\mathrm{Gal}(\mathbb{L}/\K)$, i.e. $\lambda\cdot_{\lsub{\sigma}M}x = \sigma(\lambda)x$ for $\lambda\in \mathbb{L}$ and $x\in \lsub{\sigma}M$.

\begin{mylemma}\label{lemma - galois isomorphism}
	Let $\mathbb{L}$ be a Galois field extension of $\K$ and let $A$ and $M$ be $\mathbb{L}$-algebras. We have the following isomorphisms
	\begin{equation*}
		\begin{aligned}
			\phi_A: \mathbb{L}\otimes_\K A &\to \prod_{\sigma\in \mathrm{Gal}(\mathbb{L}/\K)} \lsub{\sigma}A, \\
			a\otimes b &\mapsto (a\cdot_{\lsub{\sigma}A}b \mid \sigma\in \mathrm{Gal}(\mathbb{L}/\K)), \\
			\phi_M: \mathbb{L}\otimes_\K M &\to \prod_{\sigma\in \mathrm{Gal}(\mathbb{L}/\K)} \lsub{\sigma}M, \\
			a\otimes b &\mapsto (a\cdot_{\lsub{\sigma}M}b \mid \sigma\in \mathrm{Gal}(\mathbb{L}/\K)),
		\end{aligned}
	\end{equation*}
	of $\K$-algebras. If, moreover, $M$ is an $A$-module, then these isomorphisms are compatible, meaning that if we take the natural $\mathbb{L}\otimes_\K A$-module structure on $\mathbb{L}\otimes_\K M$, then
	\begin{equation*}
		\phi_M((n\otimes a)(n'\otimes m)) = \phi_A(n\otimes a)\phi_M(n'\otimes m)
	\end{equation*}
	where $\prod_{\sigma\in \mathrm{Gal}(\mathbb{L}/\K)} \lsub{\sigma}M$ is assumed to have the natural $\prod_{\sigma\in \mathrm{Gal}(\mathbb{L}/\K)} \lsub{\sigma}A$-module structure.
\end{mylemma}

\begin{proof}
	Let $\phi = \phi_A$ or $\phi = \phi_M$.	The first part is proven by first using \cite[Proposition 8 at A.V.64]{bourbaki1981algebra} to get that $\phi$ is bijective and $\K$-linear. It is left to check that $\phi$ is a ring homomorphism. It follows from
	\begin{equation*}
		\begin{aligned}
			\phi(a\otimes b)\phi(c\otimes d) &= (\sigma(a)b \mid \sigma\in \mathrm{Gal}(\mathbb{L}/\K)) \cdot (\sigma(c)d \mid \sigma\in \mathrm{Gal}(\mathbb{L}/\K)) = \\
			&= (\sigma(a)\sigma(c)bd \mid \sigma\in \mathrm{Gal}(\mathbb{L}/\K)) = \\
			&= (\sigma(ac)bd \mid \sigma\in \mathrm{Gal}(\mathbb{L}/\K)) = \\
			&= \phi((a\otimes b)(c\otimes d)).
		\end{aligned}
	\end{equation*}

	The second part follows from the definition of $\phi_A$ and $\phi_M$.
\end{proof}

\begin{myrem}\label{remark - galois isomoprhism}
	Note that if $M$ is a right $\mathbb{L}$-vector space, then we could define a similar isomorphism
	\begin{equation*}
		\begin{aligned}
			\phi: M\otimes_\K \mathbb{L} &\to \prod_{\sigma\in \mathrm{Gal(\mathbb{L}/\K)}} M_\sigma, \\
			a\otimes b &\mapsto (a\cdot_{M_\sigma}b \mid \sigma\in \mathrm{Gal}(\mathbb{L}/\K)),
		\end{aligned}
	\end{equation*}
	where $M_\sigma$ is the $\mathbb{L}$-vector space twisted on the right with $\sigma$.
\end{myrem}

Let us now fix a Galois field extension $\K= F\subseteq G$, following the notation in \cite{dlab1975algebras}, and assume, from now on, that all species are of the form $S = (D_i, M_\alpha)_{i\in Q_0, \alpha\in Q_1}$ such that $D_i, M_\alpha\in \{F, G \}$. Let $S^1$ and $S^2$ be two species over the quivers $Q^1$ and $Q^2$ respectively. Locally at a vertex in $S(S^1, S^2)$ we have $D_i^1\otimes_\K D_j^2$. The assumptions on $S^1$ and $S^2$ ensure that either $D_i\subseteq D_j$ or $D_j\subseteq D_i$. Assume the former since the latter is similar. Naturally $D_j$ is a $D_i$-vector space and using Lemma~\ref{lemma - galois isomorphism} we get an isomorphism
\begin{equation*}
	\phi: D_i^1\otimes_\K D_j^2 \to \prod_{\sigma\in \mathrm{Gal}(D_i/\K)}\lsub{\sigma}D_j.
\end{equation*}
Any (right) $D_i^1\otimes_\K D_j^2$-module becomes a (right) $\prod_{\sigma\in \mathrm{Gal}(D_i/\K)}\lsub{\sigma}D_j$-module via $\phi$. Let us choose the idempotents $e_\sigma = (\delta_{\sigma\sigma'})_{\sigma'\mathrm{Gal}(D_i/\K)}\in \prod_{\sigma\in \mathrm{Gal}(D_i/\K)} \lsub{\sigma}D_j$, where
\begin{equation*}
	\delta_{\sigma\sigma'} = \begin{cases}
		1, & \text{if }\sigma = \sigma', \\
		0, & \text{otherwise}.
	\end{cases}
\end{equation*}
Note that here we have to fix an order of the elements in $\mathrm{Gal}(D_i/\K)$ in order for $e_\sigma$ to be well-defined. Using $e_\sigma$ we can construct a new species $S'$ such that $T(S')\cong T(S(S^1, S^2))$. It is defined by taking $S'$ to be equal to $S(S^1, S^2)$ everywhere except at and around the vertex for $D_i^1\otimes_\K D_j^2$ and at that vertex we make the following changes: We replace the vertex $D_i^1\otimes_\K D_j^2$ by vertices $\lsub{\sigma}D_j$ for $\sigma\in \mathrm{Gal}(D_i/\K)$ and define the arrows in $S'$ as $D_l\xrightarrow{e_\sigma M_\alpha}\lsub{\sigma}D_j$ for incoming arrows $D_l\xrightarrow{M_\alpha}D_i\otimes_\K D_j$ in $S(S^1, S^2)$ and similarly for outgoing arrows $D_i^1\otimes_\K D_j^2\xrightarrow{M_\beta}D_l'$ in $S(S^1, S^2)$ we define arrows in $S'$ by $\lsub{\sigma}D_j\xrightarrow{M_\beta e_\sigma} D_l'$ for all $\sigma\in \mathrm{Gal}(D_i/\K)$. Continuing this procedure for all vertices in $S(S^1, S^2)$ yields a $\K$-species, whose tensor algebra is isomorphic to $S(S^1, S^2)$, just as in the proof of Lemma~\ref{lemma - division algebra species}.

Table~\ref{table - tensor of species} is a table of species $S(S^1, S^2)$ to which we have applied the above procedure. It contains all species $S^1$ and $S^2$ with quivers of at most one arrow. The other cases can be computed by using the table locally. We have also assumed that $\mathrm{Gal}(G/\K) = 2$ since the more general case is given by a similar picture but with more arrows, and also to make our diagram smaller and more readable.

\begin{table}[!ht]
	\centering
	\begin{adjustbox}{width=\columnwidth,center}
		\begin{tabular}{c||c|c|c|c|c|c}
			& $F$ & $G$ & $F\xrightarrow{F}F$ & $F\xrightarrow{G} G$ & $G \xrightarrow{G}F$ & $G\xrightarrow{G}G$ \\ \hline \hline
			$F$ & $F$ & $G$ & $F\xrightarrow{F}F$ & $F\xrightarrow{G} G$ & $G \xrightarrow{G}F$ & $G\xrightarrow{G}G$ \\ \hline
			$G$ & $G$ & $G \quad G$ & $G\xrightarrow{G}G$ & \begin{tikzcd}
				& G \\
				G \arrow[ru, "G"] \arrow[rd, "\lsub{\sigma}G"] \\
				& \lsub{\sigma}G
			\end{tikzcd} & \begin{tikzcd}
				G \arrow[rd, "G"] \\
				& G \\
				G \arrow[ur, "\lsub{\sigma}G"]
			\end{tikzcd} & \begin{tikzcd}
				G \arrow[r, "G"] & G \\
				G \arrow[r, "\lsub{\sigma}G"] & \lsub{\sigma}G
			\end{tikzcd} \\ \hline
			$F\xrightarrow{F}F$ & $F\xrightarrow{F}F$ & $G\xrightarrow{G}G$ & \begin{tikzcd}
				F \arrow[r, "F"] \arrow[d, "F"] & F \arrow[d, "F"] \\
				F \arrow[r, "F"] & F 
			\end{tikzcd} & \begin{tikzcd}
				F \arrow[r, "G"] \arrow[d, "F"] & G \arrow[d, "G"] \\
				F \arrow[r, "G"] & G 
			\end{tikzcd} & \begin{tikzcd}
				G \arrow[r, "G"] \arrow[d, "G"] & F \arrow[d, "F"] \\
				G \arrow[r, "G"] & F
			\end{tikzcd} & \begin{tikzcd}
				G \arrow[r, "G"] \arrow[d, "G"] & G \arrow[d, "G"] \\
				G \arrow[r, "G"] & G 
			\end{tikzcd} \\ \hline
			$F\xrightarrow{G}G$ & $F\xrightarrow{G}G$ & \begin{tikzcd}
				& G \\
				G \arrow[ru, "G"] \arrow[rd, "\lsub{\sigma}G"] \\
				& G
			\end{tikzcd} & \begin{tikzcd}
				F \arrow[r, "F"] \arrow[d, "G"] & F \arrow[d, "G"] \\
				G \arrow[r, "G"] & G 
			\end{tikzcd} & \begin{tikzcd}
				F \arrow[rr, "G"] \arrow[dd, "G"] & & G \arrow[dl, "\lsub{\sigma}G"] \arrow[dd, "G"] \\
				& G \\
				G \arrow[ur, "\lsub{\sigma}G"] \arrow[rr, "G"] & & G
			\end{tikzcd} & \begin{tikzcd}
				G \arrow[rr, "G"] \arrow[dr, "\lsub{\sigma}G"] \arrow[dd, "G"] & & F \arrow[dd, "G"] \\
				& G \arrow[dr, "\lsub{\sigma}G"] \\
				G \arrow[rr, "G"] & & G
			\end{tikzcd} & \begin{tikzcd}
				G \arrow[rrr, "G"] \arrow[ddd, "G"] \arrow[ddr, "\lsub{\sigma}G"] & & & G \arrow[ddd, "G"] \arrow[ddl, "\lsub{\sigma}G"] \\ \\
				& G \arrow[r, "\lsub{\sigma}G"] & G \\
				G \arrow[rrr, "G"] & & & G
			\end{tikzcd} \\ \hline
			$G\xrightarrow{G}F$ & $G\xrightarrow{G}F$ & \begin{tikzcd}
				G \arrow[rd, "G"] \\
				& G \\
				G \arrow[ur, "\lsub{\sigma}G"]
			\end{tikzcd} & \begin{tikzcd}
				G \arrow[r, "G"] \arrow[d, "G"] & G \arrow[d, "G"] \\
				F \arrow[r, "F"] & F 
			\end{tikzcd} & \begin{tikzcd}
				G \arrow[rr, "G"] \arrow[dr, "\lsub{\sigma}G"] \arrow[dd, "G"] & & G \arrow[dd, "G"] \\
				& G \arrow[dr, "\lsub{\sigma}G"] \\
				F \arrow[rr, "G"] & & G
			\end{tikzcd} & \begin{tikzcd}
				G \arrow[rr, "G"] \arrow[dd, "G"] & & G \arrow[dd, "G"] \\
				& G \arrow[dl, "\lsub{\sigma}G"] \arrow[ur, "\lsub{\sigma}G"] \\
				G \arrow[rr, "G"] & & F
			\end{tikzcd} & \begin{tikzcd}
				G \arrow[rrr, "G"] \arrow[ddd, "G"] & & & G \arrow[ddd, "G"] \\
				& G \arrow[r, "\lsub{\sigma}G"] \arrow[ddl, "\lsub{\sigma}G"] & G \arrow[ddr, "\lsub{\sigma}G"] \\ \\
				G \arrow[rrr, "G"] & & & G
			\end{tikzcd} \\ \hline
			$G\xrightarrow{G}G$ & $G\xrightarrow{G}G$ & \begin{tikzcd}
				G \arrow[r, "G"] & G \\
				G \arrow[r, "\lsub{\sigma}G"] & G
			\end{tikzcd} & \begin{tikzcd}
				G \arrow[r, "G"] \arrow[d, "G"] & G \arrow[d, "G"] \\
				G \arrow[r, "G"] & G 
			\end{tikzcd} & \begin{tikzcd}
				G \arrow[rrr, "G"] \arrow[drr, "\lsub{\sigma}G"] \arrow[ddd, "G"] & & & G \arrow[ddd, "G"] \\
				& & G \arrow[d, "\lsub{\sigma}G"] \\
				& & G \\
				G \arrow[rrr, "G"] \arrow[rru, "\lsub{\sigma}G"] & & & G
			\end{tikzcd} & \begin{tikzcd}
				G \arrow[rrr, "G"] \arrow[ddd, "G"] & & & G \arrow[ddd, "G"] \\
				& G \arrow[d, "\lsub{\sigma}G"] \arrow[rru, "\lsub{\sigma}G"] \\
				& G \arrow[rrd, "\lsub{\sigma}G"] \\
				G \arrow[rrr, "G"] & & & G
			\end{tikzcd} & \begin{tikzcd}
				G \arrow[r, "G"] \arrow[d, "G"] & G \arrow[d, "G"] \\
				G \arrow[r, "G"] & G \\
				G \arrow[r, "\lsub{\sigma}G"] \arrow[d, "\lsub{\sigma}G"] & G \arrow[d, "\lsub{\sigma}G"] \\
				G \arrow[r, "\lsub{\sigma}G"] & G
			\end{tikzcd}
		\end{tabular}
	\end{adjustbox}
	\caption{Tensor products of $\K$-species with at most one arrow in their quivers.}
	\label{table - tensor of species}
\end{table}

Let us produce a field extension of $\Q$ of degree $3$. Consider the cyclotomic polynomial $p(x) = 1+x^2+x^3+x^4+x^5+x^6$. The polynomial $p$ is irreducible and generates a field extension $\Q[\xi]/\Q$ of degree $6$, where $\xi$ is the primitive $7^{th}$ root of unity. The Galois group is given as
\begin{equation*}
	\mathrm{Gal}(\Q[\xi]/\Q) = \{\xi\mapsto \xi^k \mid 1\le k\le 6 \}.
\end{equation*}
Note that the complex conjugation is given as the morphism $\xi\mapsto \xi^6$ which has order $2$. The element $\xi + \xi^{-1}$ is fixed by complex conjugation and no other morphism in the Galois group. Thus $\Q[\xi + \xi^{-1}]/\Q$ is a field extension of degree $3$ and its Galois group is
\begin{equation*}
	\mathrm{Gal}(\Q[\xi + \xi^{-1}]/\Q) = \{\xi + \xi^{-1} \mapsto \xi^k + \xi^{-k} \mid 1\le k\le 3 \}.
\end{equation*}
We can even describe $\Q[\xi + \xi^{-1}]$ explicitly as $\Q[\eta]$, where $\eta = \cos(\frac{2\pi}{7})$ with
\begin{equation*}
	\mathrm{Gal}(\Q[\eta]/\Q) = \{\sigma_k \mid 1\le k\le 3, \sigma_k(\cos(\frac{2\pi}{7})) = \cos(\frac{2k\pi}{7})\}.
\end{equation*}
Using that $\cos(\frac{4\pi}{7}) = 2\eta^2 - 1$ we can deduce that
\begin{equation*}
	\mathrm{Gal}(\Q[\eta]/\Q) = \{1, \sigma, \sigma^2 \mid \sigma(\eta) = 2\eta^2 - 1\}.
\end{equation*}

\begin{myex}\label{ex - G3 example tensor with itself}
	Let $S$ and $Q$ be
	\begin{equation*}
		\q\xrightarrow{\q_\alpha}\Q, \qquad 1\xrightarrow{\alpha} 2
	\end{equation*}
	respectively. By Proposition~\ref{prop - preprojective algebra of tensor species} $\Pi(T(S)\otimes_\Q T(S))\cong \mathcal{P}(S(S, S), W(S, S))$ where $S(S, S)$ and $Q\tilde{\otimes}Q$ are given as
	
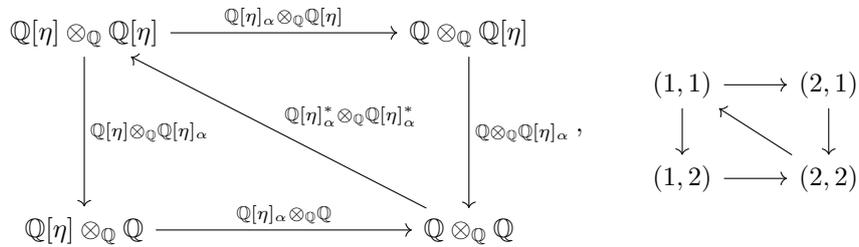
\begin{equation*}
		\begin{tikzcd}[column sep = 3 cm, row sep = 2 cm]
			\q\otimes_\Q \q \arrow[r, "\q_\alpha\otimes_\Q \q"] \arrow[d, "\q\otimes_\Q \q_\alpha"] & \Q \otimes_\Q \q \arrow[d, "\Q \otimes_\Q \q_\alpha"] \\
			\q\otimes_\Q \Q \arrow[r, "\q_\alpha\otimes_\Q \Q"] & \Q\otimes_\Q \Q \arrow[ul, "\q_\alpha^*\otimes_\Q \q_\alpha^*", swap]
		\end{tikzcd}, \qquad \begin{tikzcd}
			(1,1) \arrow[r] \arrow[d] & (2,1) \arrow[d] \\
			(1,2) \arrow[r] & (2,2) \arrow[lu]
		\end{tikzcd}.
	\end{equation*}
	Identifying $\q\otimes_\Q \Q \cong \q \cong \Q\otimes_\Q \q$ and using Table~\ref{table - tensor of species} as well as Proposition~\ref{proposition - morita equivalent k-species} we get that $\Pi(T(S)\otimes_\Q T(S))\cong \mathcal{P}(S', W')$ where $S'$ is given as
	
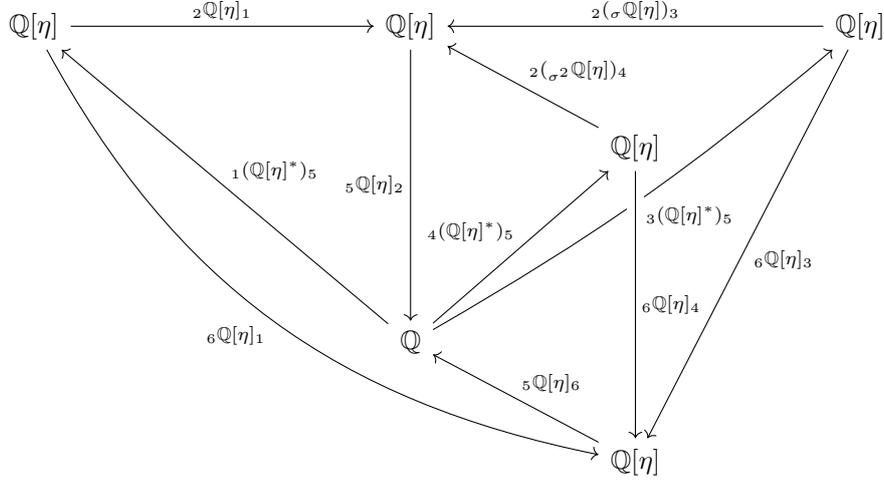
\begin{equation*}
		\begin{tikzcd}[column sep=2 cm, row sep = 1 cm]
			\q \arrow[rr, "\lsub{2}\q_1"] \arrow[ddddrrr, "\lsub{6}\q_1", bend right=25, swap] & & \q \arrow[ddd, "\lsub{5}\q_2", swap] & & \q \arrow[ll, "\lsub{2}(\lsub{\sigma}\q)_3", swap] \arrow[ddddl, "\lsub{6}\q_3"] \\
			& & & \q \arrow[ul, "\lsub{2}(\lsub{\sigma^2}\q)_4", swap] \\ \\
			& & \Q \arrow[uuull, "\lsub{1}(\q^*)_5", swap] \arrow[uuurr, "\lsub{3}(\q^*)_5", swap, bend right=5] \arrow[uur, "\lsub{4}(\q^*)_5"] \\
			& & & \q \arrow[ul, "\lsub{5}\q_6", swap] \arrow[from=uuu, "\lsub{6}\q_4", crossing over]
		\end{tikzcd}
	\end{equation*}
	over the quiver
	
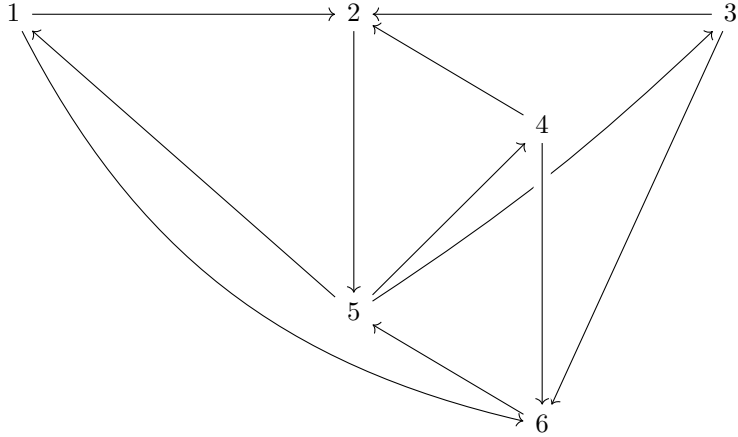
\begin{equation*}
		\begin{tikzcd}[column sep=2 cm, row sep = 1 cm]
			1 \arrow[rr] \arrow[ddddrrr, bend right=25, swap] & & 2 \arrow[ddd, swap] & & 3 \arrow[ll, swap] \arrow[ddddl] \\
			& & & 4 \arrow[ul, swap] \\ \\
			& & 5 \arrow[uuull, swap] \arrow[uuurr, swap, bend right=5] \arrow[uur] \\
			& & & 6 \arrow[ul, swap] \arrow[from=uuu, crossing over]
		\end{tikzcd}
	\end{equation*}
	and the potential given as
	\begin{equation*}
		\begin{aligned}
			W' =& \sum_{i\in \{1, 3, 4\}} \lsub{2}1_i\otimes (\lsub{i}1^*_5 \otimes \lsub{5}1_2 + \lsub{i}\eta^*_5 \otimes \lsub{5}\eta_2 + \lsub{i}(\eta^2)^*_5 \otimes \lsub{5}\eta^2_2) + \\
			&- \sum_{i\in \{1, 3, 4\}} \lsub{6}1_i\otimes (\lsub{i}1^*_5 \otimes \lsub{5}1_6 + \lsub{i}\eta^*_5 \otimes \lsub{5}\eta_6 + \lsub{i}(\eta^2)^*_5 \otimes \lsub{5}\eta^2_6)
		\end{aligned}
	\end{equation*}
	By \cite[Theorem 5.3]{soderberg2022preprojective} and \cite[Proposition 10.11]{soderberg2022preprojective} the Nakayama automorphism of $\mathcal{P}(S(S, S), W(S, S))$ is $\gamma = \id_{\mathcal{P}(S', W')}$, i.e. $\mathcal{P}(S(S^1, S^2), W(S^1, S^2))$ is symmetric. Thus by Theorem~\ref{Theorem B} the Nakayama automorphism will still be given as the identity morphism when mutating at vertex $1$. Mutating at vertex $1$ yields the species with potential given by
	
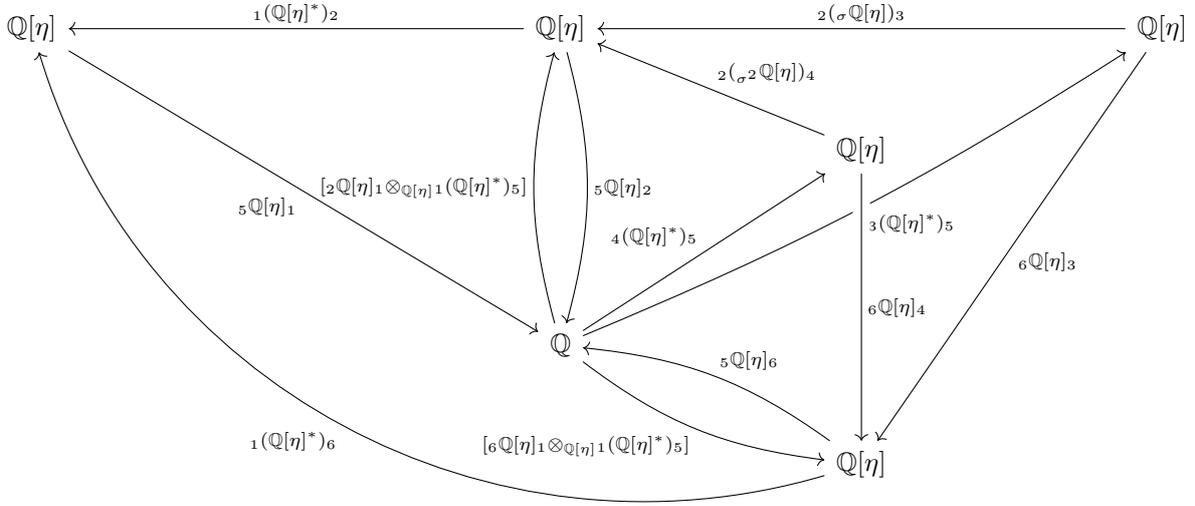
\begin{equation*}
		\begin{tikzcd}[column sep=3 cm, row sep = 1 cm]
			\q \arrow[dddrr, "\lsub{5}\q_1", swap] & & \q \arrow[ll, "\lsub{1}(\q^*)_2", swap] \arrow[ddd, "\lsub{5}\q_2", bend left=15] & & \q \arrow[ll, "\lsub{2}(\lsub{\sigma}\q)_3", swap] \arrow[ddddl, "\lsub{6}\q_3"] \\
			& & & \q \arrow[ul, "\lsub{2}(\lsub{\sigma^2}\q)_4", swap] \\ \\
			& & \Q \arrow[uuurr, "\lsub{3}(\q^*)_5", swap, bend right=5] \arrow[uur, "\lsub{4}(\q^*)_5"] \arrow[uuu, "{[}\lsub{2}\q_1\otimes_{\q}\lsub{1}(\q^*)_5{]}", bend left=15] \arrow[dr, "{[}\lsub{6}\q_1\otimes_{\q}\lsub{1}(\q^*)_5{]}", swap, bend right=15] \\
			& & & \q \arrow[ul, "\lsub{5}\q_6", swap, bend right=15] \arrow[from=uuu, "\lsub{6}\q_4", crossing over] \arrow[uuuulll, "\lsub{1}(\q^*)_6", bend left=45]
		\end{tikzcd}
	\end{equation*}
	with potential
	\begin{equation*}
		\begin{aligned}
			\widetilde{W} =& \sum_{i\in \{3, 4\}} \lsub{2}1_i\otimes (\lsub{i}1^*_5 \otimes \lsub{5}1_2 + \lsub{i}\eta^*_5 \otimes \lsub{5}\eta_2 + \lsub{i}(\eta^2)^*_5 \otimes \lsub{5}\eta^2_2) + \\
			&- \sum_{i\in \{3, 4\}} \lsub{6}1_i\otimes (\lsub{i}1^*_5 \otimes \lsub{5}1_6 + \lsub{i}\eta^*_5 \otimes \lsub{5}\eta_6 + \lsub{i}(\eta^2)^*_5 \otimes \lsub{5}\eta^2_6) + \\
			&+ [\lsub{2}1_1\otimes \lsub{1}1^*_5] \otimes \lsub{5}1_2 + [\lsub{2}1_1\otimes \lsub{1}\eta^*_5] \otimes \lsub{5}\eta_2 + [\lsub{2}1_1\otimes \lsub{1}(\eta^2)^*_5] \otimes \lsub{5}\eta^2_2) + \\
			&+ [\lsub{6}1_1\otimes \lsub{1}1^*_5] \otimes \lsub{5}1_6 + [\lsub{6}1_1\otimes \lsub{1}\eta^*_5] \otimes \lsub{5}\eta_6 + [\lsub{6}1_1\otimes \lsub{1}(\eta^2)^*_5] \otimes \lsub{5}\eta^2_6) + \\
			&+ \sum_{i \in \{2, 6\}} \lsub{1}1^*_i\otimes [\lsub{i}1_1\otimes \lsub{1}1^*_5] \otimes \lsub{5}1_1 + \lsub{1}1^*_i\otimes [\lsub{i}1_1\otimes \lsub{1}\eta^*_5] \otimes \lsub{5}\eta_1 + \lsub{1}1^*_i\otimes [\lsub{i}1_1\otimes \lsub{1}(\eta^2)^*_5] \otimes \lsub{5}\eta^2_1.
		\end{aligned}
	\end{equation*}
\end{myex}

\begin{myex}\label{ex - example G_3 with its inverse}
	Let $S^1$ be the species in Example~\ref{ex - G3 example tensor with itself} and let $S^2$ be the species over $Q^2$ where they are given as
	\begin{equation*}
		\Q \xrightarrow{\q_\beta}\q, \qquad 1 \xrightarrow{\beta} 2
	\end{equation*}
	respectively. Again, by Proposition~\ref{prop - preprojective algebra of tensor species} $\Pi(T(S^1)\otimes_\Q T(S^2))\cong \mathcal{P}(S(S^1, S^2), W(S^1, S^2))$ where $S(S^1, S^2)$ and $Q\tilde{\otimes}Q$ are given as
	
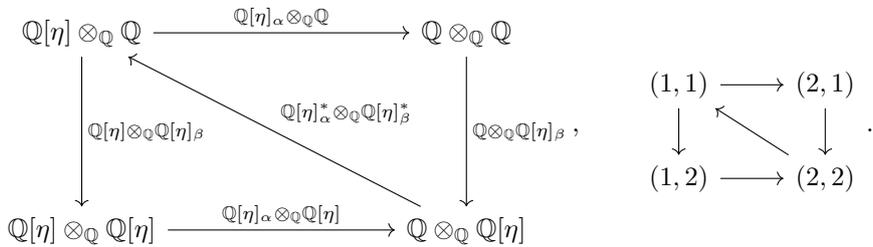
\begin{equation*}
		\begin{tikzcd}[column sep = 3cm, row sep = 2cm]
			\q\otimes_\Q \Q \arrow[r, "\q_\alpha\otimes_\Q \Q"] \arrow[d, "\q\otimes_\Q \q_\beta"] & \Q \otimes_\Q \Q \arrow[d, "\Q \otimes_\Q \q_\beta"] \\
			\q\otimes_\Q \q \arrow[r, "\q_\alpha\otimes_\Q \q"] & \Q\otimes_\Q \q \arrow[ul, "\q_\alpha^*\otimes_\Q \q_\beta^*", swap]
		\end{tikzcd}, \qquad \begin{tikzcd}
			(1,1) \arrow[r] \arrow[d] & (2,1) \arrow[d] \\
			(1,2) \arrow[r] & (2,2) \arrow[lu]
		\end{tikzcd}.
	\end{equation*}
	Similarly as in Example~\ref{ex - G3 example tensor with itself} we have that $\Pi(T(S^1)\otimes_\Q T(S^2))\cong \mathcal{P}(S(S^1, S^2), W(S^1, S^2))\cong \mathcal{P}(S', W')$, where $S'$ is given as
	
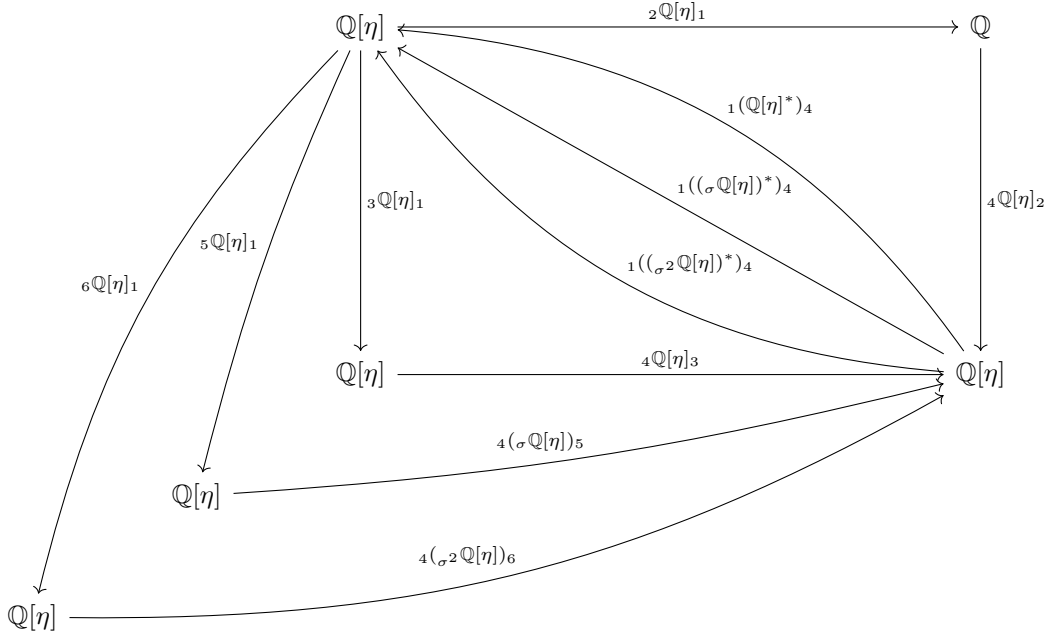
\begin{equation*}
		\begin{tikzcd}[column sep=1.2cm, row sep = 1cm]
			& & \q \arrow[rrrrrr, "\lsub{2}\q_1"] \arrow[dddd, "\lsub{3}\q_1"] \arrow[dddddl, "\lsub{5}\q_1", bend right=5, swap] \arrow[ddddddll, "\lsub{6}\q_1", bend right=15, swap] & & & & & & \Q \arrow[dddd, "\lsub{4}\q_2"] \\ \\ \\ \\
			& & \q \arrow[rrrrrr, "\lsub{4}\q_3"] & & & & & & \q \arrow[uuuullllll, "\lsub{1}((\lsub{\sigma^2}\q)^*)_4", bend left = 25, swap] \arrow[uuuullllll, "\lsub{1}((\lsub{\sigma}\q)^*)_4", swap] \arrow[uuuullllll, "\lsub{1}(\q^*)_4", bend right=25, swap] \\
			& \q \arrow[rrrrrrru, "\lsub{4}(\lsub{\sigma}\q)_5", bend right=5] \\
			\q \arrow[rrrrrrrruu, "\lsub{4}(\lsub{\sigma^2}\q)_6", bend right=15]
		\end{tikzcd}
	\end{equation*}
	over the quiver
	
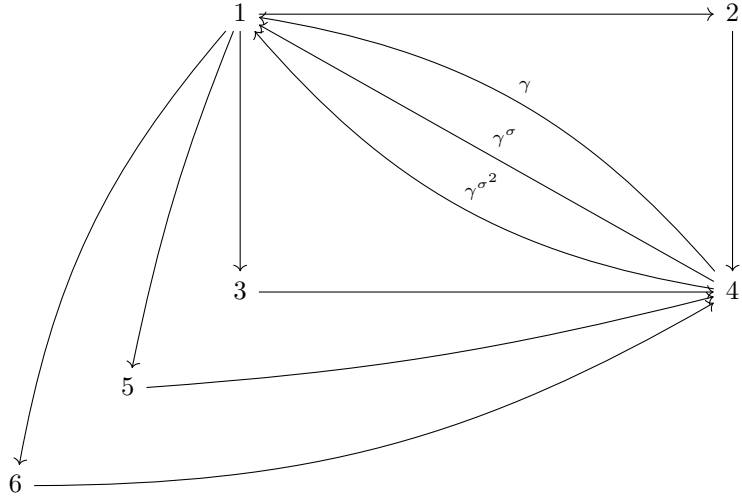
\begin{equation*}
		\begin{tikzcd}[column sep=1cm, row sep = 0.8cm]
			& & 1 \arrow[rrrrrr] \arrow[dddd] \arrow[dddddl, bend right=5] \arrow[ddddddll, bend right=15] & & & & & & 2 \arrow[dddd] \\ \\ \\ \\
			& & 3 \arrow[rrrrrr] & & & & & & 4 \arrow[uuuullllll, bend left = 20, "\gamma^{\sigma^2}", swap] \arrow[uuuullllll, "\gamma^\sigma", swap] \arrow[uuuullllll, bend right=20, "\gamma", swap] \\
			& 5 \arrow[rrrrrrru, bend right=5] \\
			6 \arrow[rrrrrrrruu, bend right=15]
		\end{tikzcd}
	\end{equation*}
	and the potential given as
	\begin{equation*}
		\begin{aligned}
			W' =& \lsub{2}1_1 \otimes (\lsub{1}1_4^* + (\lsub{1}1_4^\sigma)^* + (\lsub{1}1_4^{\sigma^2}))^*\otimes \lsub{4}1_2 + \\
			&- \sum_{x, y\in \{1, \eta, \eta^2 \}} (\lsub{3}y_1 + \lsub{5}y_1 + \lsub{6}y_1) \otimes ((\lsub{1}xy_4)^* + (\lsub{1}\sigma(x)y_4^\sigma)^* + (\lsub{1}\sigma^2(x)y_4^{\sigma^2})^*)\otimes (\lsub{4}x_3 + \lsub{4}\sigma(x)_5 + \lsub{4}\sigma^2(x)_6).
		\end{aligned}
	\end{equation*}
	Let us now mutate at vertex $2$. It yields the following species
	
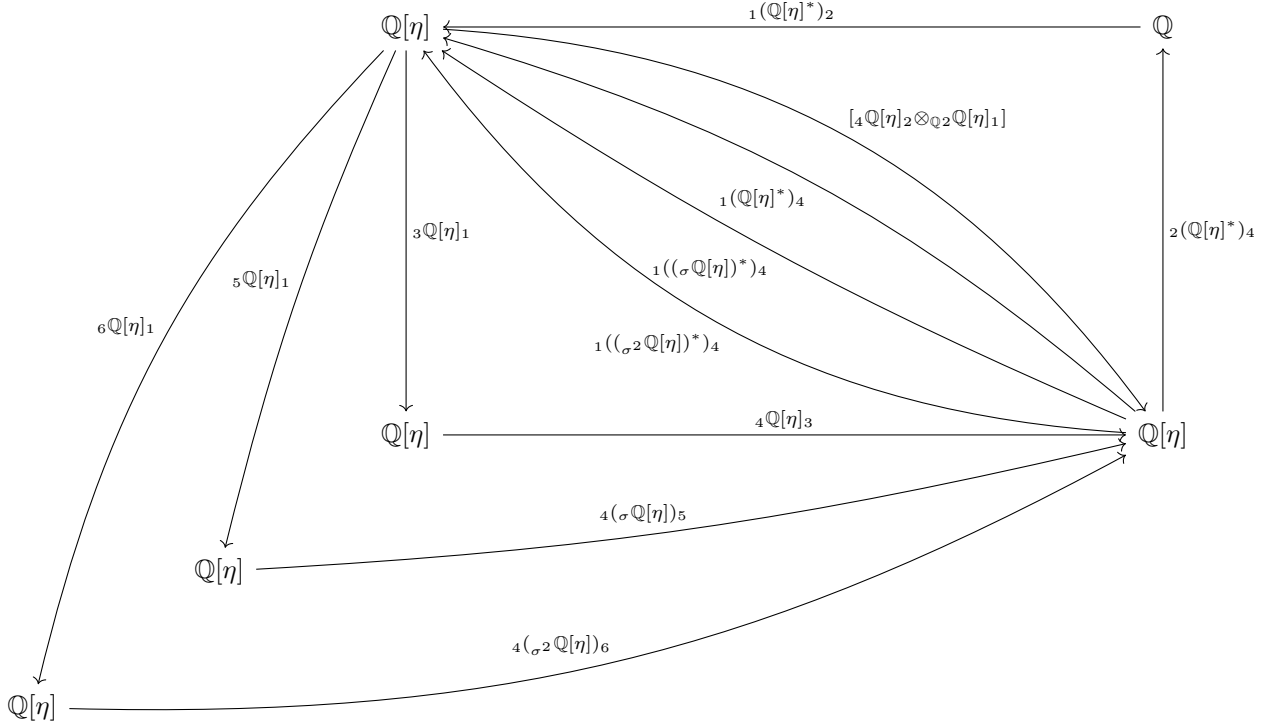
\begin{equation*}
		\begin{tikzcd}[column sep=1.5cm, row sep = 1.2cm]
			& & \q \arrow[dddd, "\lsub{3}\q_1"] \arrow[dddddl, "\lsub{5}\q_1", bend right=5, swap] \arrow[ddddddll, "\lsub{6}\q_1", bend right=15, swap] \arrow[ddddrrrrrr, "{[}\lsub{4}\q_2\otimes_{\Q}\lsub{2}\q_1{]}", bend left=25] & & & & & & \Q \arrow[llllll, "\lsub{1}(\q^*)_2", swap] \\ \\ \\ \\
			& & \q \arrow[rrrrrr, "\lsub{4}\q_3"] & & & & & & \q \arrow[uuuullllll, "\lsub{1}((\lsub{\sigma^2}\q)^*)_4", bend left = 25] \arrow[uuuullllll, "\lsub{1}((\lsub{\sigma}\q)^*)_4", bend left=5] \arrow[uuuullllll, "\lsub{1}(\q^*)_4", bend right=12] \arrow[uuuu, "\lsub{2}(\q^*)_4", swap] \\
			& \q \arrow[rrrrrrru, "\lsub{4}(\lsub{\sigma}\q)_5", bend right=5] \\
			\q \arrow[rrrrrrrruu, "\lsub{4}(\lsub{\sigma^2}\q)_6", bend right=15]
		\end{tikzcd}
	\end{equation*}
	with the potential
	\begin{equation*}
		\begin{aligned}
			\widetilde{W} =& (\lsub{1}1_4^* + (\lsub{1}1_4^\sigma)^* + (\lsub{1}1_4^{\sigma^2}))^*\otimes [\lsub{4}1_2 \otimes \lsub{2}1_1] + \\
			&- \sum_{x, y\in \{1, \eta, \eta^2 \}} (\lsub{3}y_1 + \lsub{5}y_1 + \lsub{6}y_1) \otimes ((\lsub{1}xy_4)^* + (\lsub{1}\sigma(x)y_4^\sigma)^* + (\lsub{1}\sigma^2(x)y_4^{\sigma^2})^*)\otimes (\lsub{4}x_3 + \lsub{4}\sigma(x)_5 + \lsub{4}\sigma^2(x)_6) + \\
			&+ \lsub{2}1_4^*\otimes [\lsub{4}1_2\otimes \lsub{2}1_1]\otimes \lsub{1}1_2^*.
		\end{aligned}
	\end{equation*}
	Note that here we had to apply $\varepsilon_r$ to the first term in $W'$ before we did the mutation. Similarly as in Example~\ref{ex - G3 example tensor with itself} the Nakayama automorphism will be the identity morphism.
\end{myex}

All of the examples above satisfy the criteria in Theorem~\ref{theorem - Jasso-Muro} due to the main theorems in this article. The upshot of our main theorems is that we can generate a lot of examples that can be quite hard to show to fit into Theorem~\ref{theorem - Jasso-Muro}. In particular, the mutated species in Example~\ref{ex - G3 example tensor with itself} and Example~\ref{ex - example G_3 with its inverse} are self-injective which is certainly not obvious from their descriptions.

\newpage
\bibliographystyle{alpha}
\bibliography{References}

\newcommand{\etalchar}[1]{$^{#1}$}
\begin{thebibliography}{BMR{\etalchar{+}}06}

\bibitem[AM69]{atiyahmacdonald1969}
Michael~Francis Atiyah and Ian~Grant Macdonald.
\newblock {\em {Introduction to Commutative Algebra}}.
\newblock Addison-Wesley Publishing Company, 1969.

\bibitem[Ami08]{amiot2008petites}
Claire Amiot.
\newblock {\em {Sur les Petites Cat{\'e}gories Triangul{\'e}es}}.
\newblock PhD thesis, Paris 7, 2008.

\bibitem[Ami09]{amiotclustercategories}
Claire Amiot.
\newblock {Cluster categories for Algebras of Global Dimension 2 and Quivers
  with Potential}.
\newblock {\em Universit\'{e} de Grenoble. Annales de l'Institut Fourier},
  59(6):2525--2590, 2009.

\bibitem[ASS06]{assem2006elements}
Ibrahim Assem, Andrzej Skowronski, and Daniel Simson.
\newblock {\em {Elements of the Representation Theory of Associative Algebras:
  Techniques of Representation Theory}}, volume~1 of {\em London Mathematical
  Society Student Texts}.
\newblock Cambridge University Press, 2006.

\bibitem[Ber11]{berg2011structure}
Carl~Fredrik Berg.
\newblock {Structure Theorems for Basic Algebras}.
\newblock 2011.
\newblock Preprint. arXiv: 1102.1100.

\bibitem[BFZ05]{FZ2005withberenstein}
Arkady Berenstein, Sergey Fomin, and Andrei Zelevinsky.
\newblock {Cluster Algebras. III. Upper Bounds and Double Bruhat Cells}.
\newblock {\em Duke Mathematical Journal}, 126(1):1--52, 2005.

\bibitem[BIRS11]{BIRS2011Mutation}
Aslak~Bakke Buan, Osamu Iyama, Idun Reiten, and David Smith.
\newblock {Mutation of Cluster-Tilting Objects and Potentials}.
\newblock {\em American Journal of Mathematics}, 133(4):835--887, 2011.

\bibitem[BMR{\etalchar{+}}06]{BMRRT2006tiltingtheory}
Aslak~Bakke Buan, Robert Marsh, Markus Reineke, Idun Reiten, and Gordana
  Todorov.
\newblock {Tilting theory and Cluster Combinatorics}.
\newblock {\em Advances in Mathematics}, 204(2):572--618, 2006.

\bibitem[BMR07]{BMR07clustertiltedalgebras}
Aslak~Bakke Buan, Robert~J. Marsh, and Idun Reiten.
\newblock {Cluster-Tilted Algebras}.
\newblock {\em Transactions of the American Mathematical Society},
  359(1):323--332, 2007.

\bibitem[BMR08]{BMR08clustermutationquiverrep}
Aslak~Bakke Buan, Bethany~R. Marsh, and Idun Reiten.
\newblock {Cluster Mutation via Quiver Representations}.
\newblock {\em Commentarii Mathematici Helvetici. A Journal of the Swiss
  Mathematical Society}, 83(1):143--177, 2008.

\bibitem[BMRT07]{BMRT07clustersandseedsacycliccluster}
Aslak~Bakke Buan, Robert~J. Marsh, Idun Reiten, and Gordana Todorov.
\newblock {Clusters and Seeds in Acyclic Cluster Algebras}.
\newblock {\em Proceedings of the American Mathematical Society},
  135(10):3049--3060, 2007.
\newblock With an appendix coauthored in addition by P. Caldero and B. Keller.

\bibitem[Bou81]{bourbaki1981algebra}
Nicolas Bourbaki.
\newblock {\em {Algebra II: Chapters 4-7}}.
\newblock Elements de Mathematique. Springer Science and Business Media, 1981.

\bibitem[BSW10]{BSWsuperpotential}
Raf Bocklandt, Travis Schedler, and Michael Wemyss.
\newblock {Superpotentials and Higher Order Derivations}.
\newblock {\em Journal of Pure and Applied Algebra}, 214(9):1501--1522, 2010.

\bibitem[CC06]{CC06clusteralgashallalgebras}
Philippe Caldero and Fr\'{e}d\'{e}ric Chapoton.
\newblock {Cluster Algebras as Hall Algebras of Quiver Representations}.
\newblock {\em Commentarii Mathematici Helvetici. A Journal of the Swiss
  Mathematical Society}, 81(3):595--616, 2006.

\bibitem[CK06]{CK06triangcattoclusteralg2}
Philippe Caldero and Bernhard Keller.
\newblock {From Triangulated Categories to Cluster Algebras. II}.
\newblock {\em Annales Scientifiques de l'\'{E}cole Normale Sup\'{e}rieure.
  Quatri\`eme S\'{e}rie}, 39(6):983--1009, 2006.

\bibitem[CK08]{CK08triangcattoclusteralg}
Philippe Caldero and Bernhard Keller.
\newblock {From Triangulated Categories to Cluster Algebras}.
\newblock {\em Inventiones Mathematicae}, 172(1):169--211, 2008.

\bibitem[DR75]{dlab1975algebras}
Vlastimil Dlab and Claus~M. Ringel.
\newblock {On Algebras of Finite Representation Type}.
\newblock {\em Journal of Algebra}, 33(2):306--394, 1975.

\bibitem[DR80]{Dlab_1980}
Vlastimil Dlab and Claus~Michael Ringel.
\newblock {The Preprojective Algebra of a Modulated Graph}.
\newblock In {\em Representation Theory {II}}, pages 216--231. Springer Berlin
  Heidelberg, 1980.

\bibitem[DWZ07]{dwz2008mutation}
Harm Derksen, Jerzy Weyman, and Andrei Zelevinsky.
\newblock {Quivers with Potentials and their Representations I: Mutations}.
\newblock {\em Selecta Mathematica}, 14, 2007.

\bibitem[FZ02]{FZ2002clustermutation}
Sergey Fomin and Andrei Zelevinsky.
\newblock {Cluster Algebras I: Foundations}.
\newblock {\em Journal of the American Mathematical Society}, 15(2):497--529,
  2002.

\bibitem[FZ03]{FZ2003clusteralgebras}
Sergey Fomin and Andrei Zelevinsky.
\newblock {Cluster Algebras. II. Finite Type Classification}.
\newblock {\em Inventiones Mathematicae}, 154(1):63--121, 2003.

\bibitem[FZ07]{FZ2007clusteralgebras}
Sergey Fomin and Andrei Zelevinsky.
\newblock {Cluster Algebras. IV. Coefficients}.
\newblock {\em Compositio Mathematica}, 143(1):112--164, 2007.

\bibitem[Gab73]{gabriel1973indecomposable}
Peter Gabriel.
\newblock {Indecomposable Representations II}.
\newblock {\em Symposia Mathematica, Vol XI}, pages 81--104, 1973.

\bibitem[HI11a]{herschend2011n}
Martin Herschend and Osamu Iyama.
\newblock {$n$-Representation-Finite Algebras and Twisted Fractionally
  Calabi-Yau Algebras}.
\newblock {\em Bulletin of the London Mathematical Society}, 43(3):449–466,
  Jun 2011.

\bibitem[HI11b]{HerschendOsamu2011quiverwithpotential}
Martin Herschend and Osamu Iyama.
\newblock {Selfinjective Quivers with Potential and 2-Representation-Finite
  Algebras}.
\newblock {\em Compositio Mathematica}, 147(6):1885--1920, 2011.

\bibitem[IO11]{IO11nRFalgandnAPRtilt}
Osamu Iyama and Steffen Oppermann.
\newblock {$n$-Representation-Finite Algebras and $n$-APR Tilting}.
\newblock {\em Transactions of the American Mathematical Society},
  363(12):6575--6614, 2011.

\bibitem[IO13]{IO13Stablecateofhigherpreproj}
Osamu Iyama and Steffen Oppermann.
\newblock {Stable Categories of Higher Preprojective Algebras}.
\newblock {\em Advances in Mathematics}, 244:23--68, 2013.

\bibitem[IY08]{IyamaYoshino2008Mutation}
Osamu Iyama and Yuji Yoshino.
\newblock {Mutation in Triangulated Categories and Rigid Cohen-Macaulay
  Modules}.
\newblock {\em Inventiones Mathematicae}, 172(1):117--168, 2008.

\bibitem[Iya07]{Iyama2003maximal}
Osamu Iyama.
\newblock {Higher-Dimensional Auslander-Reiten Theory on Maximal Orthogonal
  Subcategories}.
\newblock {\em Advances in Mathematics}, 210(1):22--50, 2007.

\bibitem[JKM23]{jasso2023derived}
Gustavo Jasso, Bernhard Keller, and Fernando Muro.
\newblock {The Derived Auslander-Iyama Correspondence}, 2023.

\bibitem[Kel11]{Keller2011CYcompletions}
Bernhard Keller.
\newblock {Deformed Calabi-Yau Completions}.
\newblock {\em Journal f\"{u}r die Reine und Angewandte Mathematik. [Crelle's
  Journal]}, 654:125--180, 2011.
\newblock With an appendix by Michel Van den Bergh.

\bibitem[KR08]{KR08acycliccalabiyaucat}
Bernhard Keller and Idun Reiten.
\newblock {Acyclic Calabi-Yau Categories}.
\newblock {\em Compositio Mathematica}, 144(5):1332--1348, 2008.
\newblock With an appendix by Michel Van den Bergh.

\bibitem[Kü17]{kulshammer2017pro}
Julian Külshammer.
\newblock {Pro-Species of Algebras I: Basic Properties}.
\newblock {\em Algebras and Representation Theory}, 20, 10 2017.

\bibitem[Mal42]{mal1942representation}
Anatoly~Ivanovich Mal’tsev.
\newblock {On the Representation of an Algebra as a Direct Sum of the Radical
  and a Semi-Simple Subalgebra}.
\newblock 36(1):42--45, 1942.

\bibitem[MRZ03]{MRZ03generalizedassociahedra}
Robert Marsh, Markus Reineke, and Andrei Zelevinsky.
\newblock {Generalized Associahedra via Quiver Representations}.
\newblock {\em Transactions of the American Mathematical Society},
  355(10):4171--4186, 2003.

\bibitem[Ngu12]{Nquefack2012PotentialsJacobian}
Bertrand Nguefack.
\newblock {Potentials and Jacobian algebras for Tensor Algebras of Bimodules}.
\newblock 2012.
\newblock arXiv:1004.2213.

\bibitem[Par01]{PareigisLnotes}
Bodo Pareigis.
\newblock {Lecture Notes for Advanced Algebra}.
\newblock
  \url{https://www.mathematik.uni-muenchen.de/~pareigis/Vorlesungen/01WS/advalg.pdf},
  2001.

\bibitem[Pas19]{pasquali2019tensor}
Andrea Pasquali.
\newblock {Tensor Products of {$n$}-Complete Algebras}.
\newblock {\em Journal of Pure and Applied Algebra}, 223(8):3537--3553, 2019.

\bibitem[Pas20]{Pasquali2020selfinjective}
Andrea Pasquali.
\newblock {Self-injective Jacobian Algebras from Postnikov Diagrams}.
\newblock {\em Algebras and Representation Theory}, 23(3):1197--1235, 2020.

\bibitem[Rin76]{ringel1976representations}
Claus~Michael Ringel.
\newblock {Representations of K-Species and Bimodules}.
\newblock {\em Journal of Algebra}, 41(2), 1976.

\bibitem[SY11]{skowronski2011frobenius}
Andrzej Skowronski and Kunio Yamagata.
\newblock {\em {Frobenius Algebras I: Basic Representation Theory}}.
\newblock European Mathematical Society, 2011.

\bibitem[Sö24]{soderberg2022preprojective}
Christoffer Söderberg.
\newblock {Preprojective Algebras of {$d$}-Representation Finite Species with
  Relations}.
\newblock {\em Journal of Pure and Applied Algebra}, 228(4):Paper No. 107520,
  54, 2024.

\bibitem[Wed08]{wedderburn1908hypercomplex}
JH~Maclagan Wedderburn.
\newblock {On Hypercomplex Numbers}.
\newblock {\em Proceedings of the London Mathematical Society}, 2(1):77--118,
  1908.

\end{thebibliography}
\end{document}